\documentclass[10pt]{amsart}

\usepackage{amsmath, amsthm, amssymb, graphicx, enumitem} 
\newcounter{citedtheorems}

\newcounter{theoremcounter}

\newtheorem{defn}[theoremcounter]{Definition}
\newtheorem{theorem}[theoremcounter]{Theorem}
\newtheorem*{theorem-m}{Theorem \ref{main-theorem}}
\newtheorem*{thm-m}{Main Theorem}
\newtheorem*{theorem-x}{Theorem}
\newtheorem*{concl-x}{Conclusion}
\newtheorem*{theorem-abs1}{Theorem \ref{ind-theorem}}
\newtheorem*{theorem-abs2}{Theorem \ref{a23}}
\newtheorem*{theorem-abs3}{Theorem \ref{ind-new}}
\newtheorem*{theorem-abs4}{Theorem \ref{m1}}
\newtheorem{thm-lit}[citedtheorems]{Theorem}
\newtheorem{defn-lit}[citedtheorems]{Definition}
\newtheorem{fact-lit}[citedtheorems]{Fact}
\newtheorem{fact}[theoremcounter]{Fact}

\newtheorem{cor}[theoremcounter]{Corollary}
\newtheorem{defn-claim}[theoremcounter]{Definition/Claim}
\newtheorem{cont}[theoremcounter]{Context}

\newtheorem{concl}[theoremcounter]{Conclusion}
\newtheorem{conv}[theoremcounter]{Convention}
\newtheorem{claim}[theoremcounter]{Claim}
\newtheorem{subclaim}[theoremcounter]{Subclaim}
\newtheorem{lemma}[theoremcounter]{Lemma}
\newtheorem{obs}[theoremcounter]{Observation}
\newtheorem{rmk}[theoremcounter]{Remark}

\newtheorem{disc}[theoremcounter]{Discussion}

\newtheorem{expl}[theoremcounter]{Example}

\newtheorem{qst}[theoremcounter]{Question}

\newcommand{\qfw}{q_{f,w}}
\newcommand{\rfw}{r_{f,w}}
\newcommand{\lost}{\L o\'s' }
\newcommand{\los}{\L o\'s }
\newcommand{\br}{\vspace{2mm}}

\newcommand{\kleq}{\trianglelefteq}

\newcommand{\ml}{\mathcal{L}}
\newcommand{\tlf}{\trianglelefteq}

\newcommand{\rn}{\operatorname{Range}}
\newcommand{\cf}{\operatorname{cof}}

\newcommand{\dom}{\operatorname{Dom}}

\newcommand{\mbi}{\mathbf{i}}
\newcommand{\mcv}{\mathcal{U}}

\newcommand{\mcf}{\mathcal{F}}
\newcommand{\acl}{\operatorname{acl}}

\newcommand{\mlx}{\mathcal{M}}

\newcommand{\tv}{\operatorname{tv}}

\newcommand{\supp}{\operatorname{supp}}

\newcommand{\code}{\operatorname{Code}}
\newcommand{\terms}{\operatorname{Terms}}

\newcommand{\tp}{\operatorname{tp}}
\newcommand{\otp}{\operatorname{otp}}

\newcommand{\mla}{\mathcal{M}}
\newcommand{\cla}{\operatorname{cl}_{\mla}}
\newcommand{\cls}{\operatorname{cl}_{\zs}}
\newcommand{\clmo}{\operatorname{cl}_{\zs}}
\newcommand{\zs}{\mathcal{S}}
\newcommand{\zm}{\mathcal{M}}
\newcommand{\clm}{\operatorname{cl}_{\zm}}
\newcommand{\xd}{\mathbf{d}}

\newcommand{\mpx}{} 
\newcommand{\npx}{\marginpar[]{}}

\newcommand{\de}{\mathcal{D}}

\newcommand{\fssm}{[\mu]^{<\aleph_0}}

\newcommand{\Lost}{\L o\'s' }

\newcommand{\ts}{\mathbf{S}}

\newcommand{\uu}{\mathcal{U}}
\newcommand{\vv}{\mathcal{V}}

\newcommand{\jj}{\mathbf{j}}

\newcommand{\mcp}{\mathcal{P}}

\newcommand{\ba}{\mathfrak{B}}

\newcommand{\lao}{[\lambda]^{<\aleph_0}}

\newcommand{\ee}{\mathcal{E}}
\newcommand{\mch}{\mathcal{H}}

\newcommand{\rstr}{\upharpoonright}

\newcommand{\mcr}{\mathcal{R}}

\newcommand{\vp}{\varphi}

\newcommand{\ma}{\mathbf{a}}
\newcommand{\mb}{\mathbf{b}}
\newcommand{\mc}{\mathbf{c}}
\newcommand{\md}{\mathbf{d}}

\newcommand{\mx}{\mathbf{x}}

\newcommand{\fin}{\operatorname{FI}}

\newcommand{\xs}{\mathfrak{s}}
\newcommand{\xm}{\mathfrak{m}}
\newcommand{\xn}{\mathfrak{n}}

\newcommand{{\xw}}{\mathbf{w}}
\newcommand{\xr}{\mathfrak{r}}

\newcommand{\fdp}{f} 

\newcommand{\xt}{\mathfrak{t}}

\usepackage{amssymb}

\begin{document}

\title[Existence of optimal ultrafilters and...]{Existence of optimal ultrafilters and the fundamental complexity of simple theories}

\author{M. Malliaris and S. Shelah}
\thanks{\emph{Acknowledgments:} 
We would like to thank the anonymous referee for a precise and thoughtful report which has significantly improved the presentation of the paper. 
Malliaris was partially supported by NSF grant DMS-1300634,  
by a G\"odel research prize fellowship, and by a research membership at MSRI via NSF 0932078 000 (Spring 2014). 
Shelah was partially supported by ISF grants 710/07 and 1053/11 and ERC grant 338821. 
This is paper 1030 in Shelah's list.}

\address{Department of Mathematics, University of Chicago, \\ 5734 S. University Avenue, Chicago, IL 60637, USA} 
%
\address{Einstein Institute of Mathematics, Edmond J. Safra Campus, Givat Ram, \\ The Hebrew
University of Jerusalem, Jerusalem, 91904, Israel, \\ and Department of Mathematics,
Hill Center - Busch Campus,\\  Rutgers, The State University of New Jersey, \\ 110
Frelinghuysen Road, Piscataway, NJ 08854-8019 USA}




\begin{abstract}
In the first edition of \emph{Classification Theory}, the second author 
characterized the stable theories in terms of saturation of ultrapowers. 
Prior to this theorem, stability had already been defined in
terms of counting types, and the unstable formula theorem was known. 
A contribution of the ultrapower characterization was that it involved 
sorting out the global theory, and introducing 
nonforking, seminal for the development of stability theory. Prior to the present paper, 
there had been no such characterization of an unstable class. 
In the present paper, we first establish the
existence of so-called optimal ultrafilters on Boolean algebras, 
which are to simple theories as Keisler's good ultrafilters \cite{keisler-1} 
are to all theories. Then, assuming a supercompact cardinal, we
characterize the simple theories in terms of saturation of ultrapowers.  
To do so, we lay the groundwork for analyzing the global structure 
of simple theories, in ZFC, via complexity of certain
amalgamation patterns. This brings into focus a fundamental complexity 
in simple unstable theories having no real analogue in stability. 
\end{abstract}

\maketitle 
\setcounter{tocdepth}{2}
\tableofcontents


\setcounter{theoremcounter}{0} \section{Introduction} \label{s:intro}

\subsection{Background}
We begin by giving some history and context of the power of ultraproducts as a tool in mathematics, and specifically in model theory. 
Ultrafilters on an infinite cardinal $\lambda$ are maximal (under inclusion) subsets of the power set of $\lambda$ which are closed under finite 
intersection, upward closed, and do not contain the empty set. These give a robust notion of largeness, allowing for infinite averaging arguments and the study of 
asymptotic or pseudofinite behavior in models. 
Early appearances were in the work of Tarski 1930 \cite{tarski-aleph} on measures and Cartan 1937 \cite{cartan-a}, \cite{cartan-b} in general topology. 
The groundwork for their use in model theory was laid in the 1950s and early 1960s by 
\los \cite{los}, Tarski, Keisler \cite{keisler-x}, Frayne-Morel-Scott \cite{fms}, Kochen \cite{kochen}, and Keisler \cite{keisler-x} in terms of the \emph{ultraproduct} construction. 
Given an ultrafilter $\de$ on $\lambda$, the {ultraproduct} $N$ of a sequence of models $\langle M_\alpha : \alpha < \lambda \rangle$ in a fixed language $\ml$
has as its domain the set of equivalence classes of elements of the Cartesian product $\prod_{\alpha<\lambda} M_\alpha$ under the equivalence relation of 
being equal on a set in $\de$. One then defines the relations, functions, and constants of $\ml$ on each tuple of elements of the ultraproduct to reflect the 
average behavior across the index models.
The fundamental theorem of ultraproducts, \Lost theorem, 
says that the set of statements of first order logic true in the ultraproduct are precisely the statements true in a $\de$-large set of index models, 
i.e. the theory of $N$ is the average theory of the models $M_\alpha$.  
Model theorists concentrated further on so-called regular ultrafilters, as will be explained in due course. 

This construction gave rise to some remarkable early transfer theorems. 
For example, Ax and Kochen \cite{ak1}-\cite{ak3} and independently Er\v{s}ov \cite{ersov} 
proved that for any nonprincipal (=containing all cofinite sets) ultrafilter $\de$ on the set of primes, the ultraproduct $\mathcal{Q}_p = \prod_p Q_p/\de$ of the $p$-adic fields $Q_p$ 
and the ultraproduct $\mathcal{S}_p = \prod_p \mathbb{F}_p((t))/\de$ of the fields of formal power series over $\mathbb{F}_p$ 
are elementarily equivalent, i.e. satisfy the same first-order statements. Then from Lang's theorem that every homogeneous polynomial of degree $>d$ with more than $d^2$ variables 
has a nontrivial zero in $\mathbb{F}_p((t))$ for each $p$ they deduce the corresponding theorem in $Q_p$ for all but finitely many $p$. 

Working with ultra{powers}, meaning that all the index models are the same, a similar averaging process happens.
The central ``algebraic characterization of elementary equivalence'' now appears: 
two models satisfy the same set of first order statements precisely when they have isomorphic ultrapowers, proved by Keisler 1961 under GCH \cite{keisler-x} and by 
Shelah 1971 \cite{Sh:13} in general.

The theorems of Ax-Kochen and Ersov just mentioned used only ultrafilters on $\omega$, \lost transfer of the first order theory and $\aleph_1$-saturation. 
The first order theory is a relatively superficial description of models 
shared by many in the same class, e.g. all algebraically closed fields of characteristic 0. 
From a model-theoretic point of view, the deeper structure of ultrapowers has to do with the Stone space of types and the property of saturation.\footnote{Saturation is a fullness condition. Informally, we may identify types over a given set $A$ with orbits under automorphisms of some much larger universal domain which fix $A$ pointwise, so call a model $M$ $\kappa$-saturated if it contains representatives of all orbits under any automorphism of the larger universal domain which fix some subset of $M$ of size $<\kappa$ pointwise.  Syntatically, an $A$-type $p$ is a maximal consistent set of formulas 
in a fixed number of free variables and parameters from $A$; it is realized in the model $M \supseteq A$ if for some $\bar{a} \in \dom(M)$, $M \models \vp(\bar{a})$ for all $\vp \in p$; and a $\kappa$-saturated model is one in which all types over sets $|A| < \kappa$ are realized.} 

Keisler, one of the major architects of model theory beginning in the 1960s, 
in particular has done much on ultrapowers, see \cite{keisler-survey}. 
He proved that  
one could define and build a family of so-called good regular ultrafilters \cite{keisler-1},\footnote{Keisler's proof assumed GCH, an hypothesis removed by Kunen \cite{kunen}, in a proof which introduced some central techniques in ultrafilter construction.}
and he noticed that for any $\de$ in this family of good regular ultrafilters and any model $M$ in a countable language, 
the ultrapower $M^\lambda/\de$ is sufficiently, i.e. $\lambda^+$-, saturated. Moreover, ultrapowers of certain theories are only saturated if the ultrafilter is good. 
In 1967 \cite{keisler}, Keisler proposed a means of comparing the complexity of theories according to the difficulty 
of saturating their ultrapowers (Definition \ref{keisler-order}). 
Progress on this far reaching program, known as Keisler's order, 
requires advances in model theory on one hand, and advances in ultrafilter construction on the other.  

From the model theoretic point of view, a major motivation for understanding Keisler's order comes from the search for dividing lines, that is, 
properties 
whose presence gives complexity and whose absence gives a good structure theory. 
For example, the dividing line of stable versus unstable theories has been fundamental since the 1970s \cite{Sh:c}. 
However, there are many unstable theories, and for some of them a `positive theory' may be analyzed; 
so if one hopes to generalize stability theory, a natural approach is to find and develop 
other dividing lines one by one in response to suitable questions. 
By a 1978 theorem of Shelah, Keisler's order independently detects the dividing line at stability. 
This suggests that a fruitful and moreover {uniform} way of 
looking for meaningful divisions in the enormous class of unstable theories is to progress, if possible, in the unstable classification of Keisler's order. 
In short, Keisler's order provides a {uniform point of view} from which to approach the problem of looking for dividing lines 
for a large and central family of theories, and a model-theoretic incentive to characterize equivalence classes.  
A natural target is the family of simple theories, a central and popular family in model theory for more than two decades. 
For background on simple theories and some history, see the survey \cite{GIL}. We mention here that simple unstable theories include examples such as pseudofinite fields \cite{hrushovski1} and have been a fertile ground of model-theoretic interaction with algebra, geometry, combinatorics 
and number theory.

For many years there was little progress on Keisler's order. Recently, a series of papers by the authors has significantly changed the landscape (see \cite{MiSh:E74} for some history) and given us the leverage for the present work. In the current paper, we establish the existence of a new family of ultrafilters 
in parallel with developing the model theory of simple theories in a direction very different from prior work. Combining the two, we prove a characterization of the class of simple theories in terms of saturation of ultrapowers, assuming a supercompact cardinal. (For a discussion of our use of a large cardinal, see \S \ref{s:why-compact} below.) 
As would be expected, our work here has two complementary parts: on one hand we define and establish the existence of {optimal} ultrafilters, 
and on the other we extend the model theory of simple theories in order to show that such ultrafilters affect saturation. 
We hope to eventually be able to eliminate the large cardinal hypothesis in the main theorem of this paper, 
but for now it clarifies the model-theoretic content by allowing us to work with models as closed sets (as will be explained in due course). 

Although there has been much interest and work on simple theories, many basic structural questions 
about simple theories and the extent to which they may differ essentially from a few canonical examples remain wide open.  
One theme of this paper has to do with finding the right frame for seeing divisions in complexity classes among the simple unstable theories.  
The complexity we detect has primarily to do with amalgamation; it appears built on non-forking, and has no real analogue in the stable case 
(where we have amalgamation even of $\mcp^{-}(n)$-diagrams, see \cite{Sh:c} Chapter XII).  
We give a model-theoretic formulation of this property, which we call explicit simplicity, in Section \ref{s:e-simple}.
The history of classification theory would suggest that future work may well 
reveal other formulations of this property; such a formulation would be likely to arise from 
progress on determining the identity of equivalence classes in Keisler's order among the simple theories (as opposed to the identification of dividing lines).

\br
\subsection{Results}
We prove the following theorems. In each case, the results will hold for any four-tuple of infinite cardinals 
$(\lambda, \mu, \theta, \sigma)$ satisfying the following hypotheses, plus any additional requirements given 
in the theorem. 

\begin{defn} \label{d:suitable} Call $\lambda, \mu, \theta, \sigma$ \emph{suitable} when:
 
\begin{enumerate}[label=\emph{(\alph*)}]
\item $\sigma \leq \theta \leq \mu < \lambda$.
\item \label{x:10} $\theta$ is regular, $\mu = \mu^{<\theta}$ and $\lambda = \lambda^{<\theta}$.  
\item \label{x:12} $(\forall \alpha < \theta)(2^{|\alpha|} < \mu)$. 
\end{enumerate}
\end{defn}

Conditions (b) and (c) will essentially guarantee that certain equivalence relations defined in the course of our proofs 
are not unnecessarily large. For example, these hypotheses hold when 
$\sigma = \theta = \aleph_0$ and $\mu < \lambda$ are any infinite cardinals, and they hold when 
$\sigma$ is uncountable and supercompact, $\sigma = \theta = \mu$, 
and $\mu^+ = \lambda$.  Regarding supercompact, usually ``$\sigma$ a compact cardinal'' will suffice 
(keeping in mind that we will be working with $\sigma$-complete filters and ultrafilters), 
but the existence theorem for optimal ultrafilters given below 
assumes existence of an uncountable supercompact cardinal. 
It will follow from the main definitions that these cardinals $\lambda, \mu, \theta, \sigma$ 
each control specific aspects of both the model-theoretic and the set-theoretic picture, 
and varying their values, modulo the basic constraints of \ref{d:suitable}, will give useful information.  

We first state the theorem which organizes our main results, before discussing ``explicitly simple.'' 

\begin{theorem-x}[Organizing theorem, Theorem \ref{t:summary} below]
Assume $(\lambda, \mu, \theta, \sigma)$ are suitable and that $\sigma$ is an uncountable supercompact cardinal.  
There exists a regular ultrafilter $\de$ over $\lambda$ such that for every model $M$ in a countable signature, 
$M^\lambda/\de$ is $\lambda^+$-saturated if $Th(M)$ is $(\lambda, \mu, \theta, \sigma)$-explicitly simple, and 
$M^\lambda/\de$ is not $\mu^{++}$-saturated if $Th(M)$ is not simple. 
\end{theorem-x}

As the statement of this theorem suggests, a model-theoretic contribution of the paper is the development of a notion we call 
$(\lambda, \mu, \theta, \sigma)$-explicitly simple, a measure of the complexity of amalgamation, discussed further 
in the introduction to Section \ref{s:e-simple}. It will be clear from Section \ref{s:e-simple} that 
$(\lambda, \mu, \theta, \sigma)$-explicitly simple becomes weaker as $\mu$ increases, and that every 
$(\lambda, \mu, \theta, \sigma)$-explicitly simple theory is simple, even when $\mu = \lambda$.  
More remarkable is that it is possible to capture simplicity in this way.
 
\begin{theorem-x}[Simple theories are explicitly simple, Theorem \ref{c:s-es} below] 
Assume $(\lambda, \mu, \theta, \sigma)$ are suitable.  If $\mu^+ = \lambda$, then every simple theory $T$ 
with $|T| < \sigma$ is $(\lambda, \mu, \theta, \sigma)$-explicitly simple, and moreover this 
characterizes simplicity of $T$. 
\end{theorem-x}

As discussed later in this paper, we believe this new characterization will re-open the research on simple theories, 
which has long been dominated by analogies 
to stability theory, by allowing for a classification of simple unstable theories according to 
the possible values of $\mu << \lambda$. 

Complementing the development of explicit simplicity, we define and prove existence of a new family of ultrafilters, 
called $(\lambda, \mu, \theta, \sigma)$-optimal. 
These ultrafilters, defined in Section \ref{s:optimal}, may be thought of as an analogue of 
Keisler's \emph{good} ultrafilters from \cite{keisler-1} which handle patterns arising from explicitly simple theories.  
The Boolean algebra in the statement of this theorem will be defined in \ref{d:ind-fns}(3). 

\begin{theorem-x}[Existence theorem for optimal ultrafilters, Theorem \ref{t:optimal} below]  
There exists a $(\lambda, \mu, \theta, \sigma)$-optimal ultrafilter on the Boolean algebra 
$\ba = \ba^1_{2^\lambda, \mu, \theta}$ whenever $(\lambda, \mu, \theta, \sigma)$ are suitable and $\sigma > \aleph_0$ is 
supercompact. 
\end{theorem-x}

On the connection of this ultrafilter on a Boolean algebra to a regular ultrafilter 
on $\lambda$, see Section \ref{s:int-st} below. 
Finally, leveraging explicit simplicity and optimality, we prove the algebraic characterization of simple theories. 

\begin{theorem-x}[Ultrapower characterization of simplicity, Theorem \ref{main-theorem} below] 
Suppose $(\lambda, \mu, \theta, \sigma)$ are suitable where $\sigma$ is an uncountable supercompact cardinal and $\mu^+ = \lambda$. 
Then there is a regular ultrafilter $\de$ on $\lambda$ such that for any model $M$ in a countable signature, 
$M^\lambda/\de$ is $\lambda^+$-saturated if $Th(M)$ is simple and $M^\lambda/\de$ is not $\lambda^+$-saturated if $Th(M)$ is not simple. 
\end{theorem-x}

Theorem \ref{main-theorem} has the following consequence for Keisler's order:

\begin{concl-x}[On Keisler's order $\tlf$, Conclusion \ref{consequences} below] 
Assume there exists an uncountable supercompact cardinal. If $T$, $T^\prime$ are complete
countable theories, $T$ is simple, and $T^\prime \trianglelefteq T$, 
then $T^\prime$ is simple. 
\end{concl-x}

The technology developed in this paper has some surprising consequences, which will appear in subsequent work.   
Notably, overturning a longstanding conjecture that Keisler's order has finitely many classes, 
we prove in \cite{MiSh:1050} already within the simple theories there is substantial complexity: 

\begin{thm-lit}[Malliaris and Shelah \cite{MiSh:1050}, in ZFC] 
Keisler's order has infinitely many classes. In fact, there is an infinite strictly descending chain of simple low theories in Keisler's order, above the random graph.   
\end{thm-lit}

Thus Keisler's order is sensitive to the fine structure of amalgamation as measured by our criterion of explicit simplicity.  
This framework raises questions which we plan to address in work in progress, related to the natural interpretation of 
the coloring criterion developed here within particular classes of simple theories.

\br
\subsection{Introduction for set theorists} \label{s:int-st}

Here we briefly outline the innovations of the paper which may be of interest to set theorists, independent of the model-theoretic 
questions of saturation of simple theories. These are of two kinds (as we will explain). First is constructing ultrafilters on $\lambda$ or on 
related Boolean algebras.  Second, the classes $\operatorname{Mod}_T$, $T$ simple, include examples natural for set theorists 
such as the random graphs and certain classes of hypergraphs. 

A major part of the paper has to do with the construction of regular ultrafilters \emph{using} large cardinals.  
Historically, model theorists had typically focused on regular ultrafilters because of the connections to the 
compactness theorem, and in particular their nice saturation properties 
(see e.g. Theorem \ref{f:str}, page \pageref{f:str} below) whereas set theorists had typically focused on understanding 
quite complete ultrafilters under the relevant large cardinal hypotheses (or remnants of this like $\aleph_0^{\aleph_1}/\de = \aleph_1$).  
Our present approach, following our earlier paper \cite{MiSh:999}, reunites the two. 
Let $\ba$ be a complete Boolean algebra of cardinality $\leq 2^\lambda$ with the $\leq \lambda^+$-c.c. 
Following \cite{MiSh:999}, 
we build regular ultrafilters on $\lambda$ by first building a regular, $\lambda^+$-good filter $\de_0$ so that $\mcp(\lambda)/\de_0$ 
is isomorphic to $\ba$, and then complete the construction by specifying an ultrafilter $\de_*$ on $\ba$, which need not be regular. 
In the present paper, 
our main case is $\ba = \ba^1_{2^\lambda, \mu, \theta}$, the completion of the free Boolean algebra generated by $2^\lambda$ independent partitions 
of size $\mu$, where intersections of size $<\theta$ are nonempty. 
(Further work considering the case where $\ba$ is not necessarily the completion of a free 
Boolean algebra will be developed in \cite{MiSh:F1484}.) In the present paper, we use $\ba$ exclusively to refer to one of these completions of a
free Boolean algebra. 

In this setup, our focus is on construction of appropriate ultrafilters $\de_*$ on $\ba$. 
The present paper introduces two new set-theoretic properties of ultrafilters on such Boolean algebras, ``optimal'' in \ref{d:optimal} 
and ``perfect'' in \ref{d:perfect-0}, and proves that such ultrafilters exist.  Both definitions capture in some sense being as good 
(c.f. Keisler's `good' ultrafilters) as possible modulo some background cardinal constraints. 
We succeed to prove that if $\sigma \leq \theta \leq \lambda$ and $\sigma$ is supercompact then on $\ba = \ba^1_{2^\lambda, \mu, \theta}$ an 
optimal ultrafilter exists, and we prove existence of a perfect ultrafilter on $\ba$ in ZFC assuming $\sigma = \theta = \aleph_0$. 
In the present paper, we require that $\sigma$, if uncountable, is supercompact rather than simply requiring existence of a 
$\sigma$-complete ultrafilter on $\lambda$, because this is what we use in the existence proof for optimal ultrafilters.  

The model-theoretic usefulness of this ``separation of variables'' approach was established by \cite{MiSh:999}, in particular Theorem \ref{t:separation} 
quoted in the next section, 
which says that the resulting saturation properties 
of the regular ultrafilter $\de$ induced on $\lambda$ by $\de_0$ and $\de_*$ may be characterized in terms of related conditions on the ultrafilter $\de_*$. 
So we are free to address saturation problems 
by working with $\sigma$-complete ultrafilters $\de_*$ on Boolean algebras (in the present case, completions of free Boolean algebras), 
a much richer context. But it has also pure set-theoretic meaning: for instance, finding ultrafilters on $\lambda$ which are flexible 
but not good, see \ref{c:dow} below, as asked by Dow 1985 \cite{dow}.

Readers unfamiliar with simple theories may prefer to keep in mind one of the many natural combinatorial examples of such theories. 
For each $n> k\geq 2$, let $T_{n,k}$ be the theory of the unique countable generic hypergraph in a language with 
a $k+1$-ary graph hyperedge (a symmetric, irreflexive $k+1$-place relation) where the axioms say that the theory is infinite and 
any configuration of edges and non-edges is allowed provided that there are no complete hypergraphs on $n+1$ vertices 
(i.e. there do not exist $n+1$ vertices of which every 
distinct subset of size $k+1$ is a hyperedge).  When $k=1$ such theories are not simple (e.g. the triangle-free random 
graph) but when $k \geq 2$ they are (e.g. the tetrahedron-free three-hypergraph), 
as proved by Hrushovski \cite{hrushovski1}. 
These examples will be central to a further analysis of simple theories via perfect ultrafilters 
in \cite{MiSh:1050}.  

The reader interested primarily in these new ultrafilters may skip ahead to Sections \ref{s:optimal} and \ref{s:perfect}. 
Such a reader may also find it useful to skim Section \ref{s:prep}
for insight into the saturation claims we make about these ultrafilters, which are largely combinatorial in nature.

\br 
\section{Overview and preparation} \label{s:prep}  
\setcounter{equation}{0}
 
\subsection{Overview} 
We begin by giving an overview of some main themes of the paper. 
For additional information on Keisler's order the reader may wish to consult Keisler 1967 \cite{keisler}, 
or the recent papers \cite{MiSh:E74}, \cite{MiSh:996}, and \cite{Moore}. 

\begin{conv}[on types] \label{c:on-types} Given $N := M^\lambda/\de$ an ultrapower,  
\begin{enumerate}[label=\emph{(\alph*)}]
\item Call a type or partial type $p$ over $A$, $A \subseteq N$ \emph{small} if $|A| \leq \lambda$. 
\item Any small type may be enumerated (possibly with repetitions) as $\{ \vp_i(x,a_i) : i < \lambda \}$, where $\ell(a_i)$ need not be 1. 
\item For each parameter $a \in A \subseteq N$, fix in advance some lifting of $a$ to $M^\lambda$. 
Then by the notation $a[t]$ we mean the $t$-th coordinate of this lifting of $a$. 
When $a$ is a tuple $a_1, \dots, a_n$, the notation $a[t]$ is understood to mean the tuple $a_1[t], \dots, a_n[t]$.\footnote{Informally, $a[t]$ is the ``projection of $a$ to index $t$'' or ``to the index model $M_t$''.} 
We will use this notation throughout the paper. 
\item By \lost theorem, if $p$ is a consistent partial type in $N$ then we may define the \emph{\los map}
$f: [p]^{<\aleph_0} \rightarrow \de$ by
\[ u \mapsto \{ t \in \lambda : M \models \exists x \bigwedge_{i \in u} \vp_i(x,a_i[t]) \}\]
\item $[A]^{<\kappa}$ denotes the set of all subsets of $A$ of cardinality $<\kappa$.
\item $\dom(M)$ denotes the universe of a structure $M$, and $||M|| = |\dom(M)|$. 
\end{enumerate}
\end{conv}

\begin{defn} \label{d:regular} The ultrafilter $\de$ on $\lambda$ is \emph{regular} if it contains a regularizing family, that is, 
a set $\{ X_i : i < \lambda \} \subseteq \de$ such that for any $u \subseteq \lambda$, $|u| \geq \aleph_0$, 
\[ \bigcap_{ i \in u } X_i = \emptyset.  \]
\end{defn}

\noindent Equivalently, $\de$ is regular if every set of size $\leq \lambda$ in any $\de$-ultrapower is covered by a pseudofinite set.

The hypothesis \emph{regular} entails that saturation of ultrapowers is a property of the (countable) theory, not the model chosen: 

\begin{thm-lit}[Keisler \cite{keisler-1} Cor. 2.1a] \label{f:str} When $\de$ is a regular ultrafilter on $\lambda$ and $M \equiv N$ in a countable signature, then 
$M^\lambda/\de$ is $\lambda^+$-saturated iff $N^\lambda/\de$ is $\lambda^+$-saturated.
\end{thm-lit}

Theorem \ref{f:str} justifies the quantification over all models in the next, central definition. 

\begin{defn}[Keisler's order, Keisler 1967 \cite{keisler}] \label{keisler-order}
Let $T_1, T_2$ be complete countable first order theories. Say $T_1 \tlf T_2$ if whenever $\lambda \geq \aleph_0$, $\de$ is a regular ultrafilter on $\lambda$, 
$M_1 \models T_1$, $M_2 \models T_2$ we have that 
\[ (M_2)^\lambda/\de \mbox{ is $\lambda^+$-saturated } \implies (M_1)^\lambda/\de \mbox{ is $\lambda^+$-saturated } \]
\end{defn}

Keisler's order $\kleq$ is a preorder on theories, 
often thought of as a partial order on the $\kleq$-equivalence classes.

\begin{qst}[Keisler 1967]
Determine the structure of Keisler's order. 
\end{qst}

The state of what was known about the structure Keisler's order through 2012 can be found in section 4 of the authors' paper \cite{MiSh:996}. Since that paper was written, and prior to the current paper, the following results have been obtained:

\begin{thm-lit}[Malliaris and Shelah \cite{MiSh:999}] \label{th:999} 
Keisler's order has at least two classes among the simple unstable theories. 
\end{thm-lit}

\begin{thm-lit}[Malliaris and Shelah \cite{MiSh:998}, announced in \cite{MiSh:E74}] \label{t:998}
Any theory with the model-theoretic tree property $SOP_2$ belongs to the maximum class in Keisler's order. 
\end{thm-lit}

To explain what is, in general, at stake in questions of saturation of ultrapowers, 
we now discuss types in regular ultrapowers. 
For transparency, all languages and thus all theories are countable. 

\begin{defn} \label{d:dist}
Let $\de$ be a regular ultrafilter on $\lambda$, $M$ a model in a countable signature, $p$ a small partial type over $A \subseteq N := M^\lambda/\de$. 
A \emph{distribution} of $p$ is a map $d: [p]^{<\aleph_0} \rightarrow \de$ such that:
\begin{enumerate}[label=\emph{(\alph*)}]
\item $d$ is monotonic, i.e. $u \subseteq v \implies d(v) \subseteq d(u)$, and $d(\emptyset) = \lambda$
\item $d$ refines the \los map $f$, meaning that $d(u) \subseteq f(u)$ for each $u \in [p]^{<\aleph_0}$
\item the image of $d$ is a regularizing family, $\ref{d:regular}$ above. 
\end{enumerate}
\end{defn}

In some sense, the problem of realizing types in ultrapowers is already visible in what \lost theorem does \emph{not} guarantee. Although a type 
is ``on average'' (in the ultrapower) consistent, \emph{i.e.} distributions exist, 
when we try to realize it by assigning finitely many formulas of the type to each index model 
via a distribution \ref{d:dist} it becomes apparent that there is no guarantee that the finite set of formulas 
$\{  \vp_i(x,a_i[t]) : t \in d(\{ \vp_i \}) \}$ assigned to index $t$ has a common realization.

Note that \ref{d:dist}(a) is not necessary, as it may always be ensured (refining a given map by induction on the size of $u$).

Specifically, the following fact explains a basic mechanism controlling saturation of regular ultrapowers.

\begin{fact}[\cite{MiSh:996} 1.8] \label{f:mult} 
Let $\de$ be a regular ultrafilter on $\lambda$, $M$ a model in a countable signature, $p$ a small partial type over $A \subseteq N := M^\lambda/\de$. 
Then the following are equivalent:

\begin{enumerate}
\item $p$ is realized in $N$. 
\item Some distribution $d$ of $p$ has a multiplicative refinement, that is, a map $d^\prime: [p]^{<\aleph_0} \rightarrow \de$ such that for any $u,v$, 
first, $d^\prime(u) \subseteq d(u)$, 
and second, $d^\prime(u) \cap d^\prime(v) = d^\prime(u \cup v)$. 
\end{enumerate}
\end{fact}

The property of (monotonic) maps from $[\lambda]^{<\aleph_0} \rightarrow \de$ admitting multiplicative refinements is a natural set-theoretic question:

\begin{defn} \emph{(Good ultrafilters, Keisler \cite{keisler-1})}
\label{good-filters}
The filter $\de$ on $\lambda$ is said to be \emph{$\mu^+$-good} if every $f: \fssm \rightarrow \de$ has
a multiplicative refinement, where this means that for some $f^\prime : \fssm \rightarrow \de$,
$u \in \fssm \implies f^\prime(u) \subseteq f(u)$, and $u,v \in \fssm \implies
f^\prime(u) \cap f^\prime(v) = f^\prime(u \cup v)$.

Note that we may assume the functions $f$ are monotonic.

$\de$ is said to be \emph{good} if it is $\lambda^+$-good.
\end{defn}

Keisler proved that {good} regular ultrafilters on $\lambda$ always exist assuming GCH \cite{keisler-1}; this was proved in ZFC by 
Kunen \cite{kunen}. Thus, by Fact \ref{f:mult}, for any $\lambda$ there exists a regular ultrafilter on $\lambda$ such 
that $M^\lambda/\de$ is $\lambda^+$-saturated for any $M$ in a countable signature. In the other direction, there exist $T$ able to code failures 
of goodness, e.g. $Th([\omega]^{<\aleph_0}, \subseteq)$, so that if $M\models T$ then $M^\lambda/\de$ is $\lambda^+$-saturated iff $\de$ is good
(Keisler \cite{keisler} Theorem 1.4c). 
This proves existence of a maximum class in Keisler's order.

\begin{defn} \label{conv:good}
Reflecting the saturation properties of good ultrafilters, when $\de$ is an ultrafilter on $\lambda$ 
we will say that ``$\de$ is good for $T$,'' or ``$\de$ is $(\lambda^+, T)$-good,'' to mean that for any $M \models T$, $M^\lambda/\de$ is $\lambda^+$-saturated. 
\end{defn}

We now know that it is also possible for a theory to be Keisler-maximal without explicitly coding all failures of goodness:
\begin{thm-lit}[Shelah 1978 \cite{Sh:c} VI.3.9] 
Any theory with the strict order property is maximal in Keisler's order, e.g. $Th(\mathbb{Q}, <)$.
\end{thm-lit}
In fact, $SOP_2$ suffices (Malliaris and Shelah, Theorem \ref{t:998} above). The ``basis'' of functions whose multiplicative refinements 
ensure that of all others is not yet understood. We know the only essential complexity is local:

\begin{fact}[Local saturation suffices, Malliaris \cite{mm1} Theorem 12] \label{phi-types}
Suppose $\de$ is a regular ultrafilter on $I$ and $T$ a countable complete first order theory. Then for any $M^I/\de$, the following are equivalent:
\begin{enumerate}
\item $M^I/\de$ is $\lambda^+$-saturated.
\item $M^I/\de$ realizes all $\vp$-types over sets of size $\leq \lambda$, for all formulas $\vp$ in the language of $T$.
\end{enumerate}
\end{fact}

To understand classes other than the Keisler-maximal class, as in the present paper, it is therefore 
necessary to realize some types while omitting others, that is,
to understand how certain model-theoretically meaningful families of functions may have multiplicative refinements while others do not. 
A point of leverage on this problem was built in \cite{MiSh:999} and applied there to obtain the first ZFC dividing line among the unstable theories. 
It translates the problem just described into a problem about patterns in some quotient Boolean algebra, as we now explain. 
For \ref{d:built}, note that the notion of an $\lambda^+$-excellent filter is defined in \cite{MiSh:999}. It is proved in Theorem 12.3 of that paper that a filter is 
$\lambda^+$-excellent if and only if it is $\lambda^+$-good and in the present paper, $\lambda^+$-excellent and 
$\lambda^+$-good are used interchangeably.\footnote{More precisely, one can define a notion of ``excellent for a theory $T$'' 
and likewise of ``good for a theory $T$.''
What is proved in \cite{MiSh:999} is that ``excellent'' i.e. ``excellent for all countable $T$'' coincides with ``good'' i.e. ``good for all countable $T$'' 
i.e. every monotonic function from finite subsets of $I$ into the filter has a multiplicative refinement. 
This is the property we need for Theorem \ref{t:separation}, so the reader may 
substitute good for excellent in that theorem. 
However, it is important to mention that for specific values of $T$, ``excellent for $T$'' and ``good for $T$'' need not coincide.
If one wanted to work with more precise versions of Theorem \ref{t:separation} where $\de_0$ 
is excellent only for certain theories, the situation might be different.} 

\begin{defn}[Regular ultrafilters built from tuples, from \cite{MiSh:999}] \label{d:built}
Suppose $\de$ is a regular ultrafilter on $I$, $|I| = \lambda$. We say that $\de$ is built from 
$(\de_0, \ba, \de_*)$ when the following hold. Note that $\lambda$ is given by $\de_0$, and
 if not mentioned otherwise, we will assume the index set of $\de$ is $\lambda$. 

\begin{enumerate}
\item {$\de_0$ is a regular, $|I|^+$-excellent filter on $I$} 
\\ {$($for the purposes of this paper, it is sufficient to use regular and good$)$}
\item {$\ba$ is a complete Boolean algebra of cardinality $2^\lambda$ and $\leq \lambda^+$-c.c.}
\item {$\de_*$ is an ultrafilter on $\ba$}
\item {there exists a surjective homomorphism $\jj : \mcp(I) \rightarrow \ba$ such that:}
\begin{enumerate}
\item $\de_0 = \jj^{-1}(\{ 1_\ba \})$ 
\item $\de = \{ A \subseteq I : \jj(A) \in \de_* \}$.
\end{enumerate}
\end{enumerate}
We may make $\jj$ explicit and write ``built from $(\de_0, \ba, \de_*, \jj)$''. 
\end{defn}

It was verified in \cite{MiSh:999} Theorem 8.1 that whenever $\mu \leq \lambda$ and $\ba = \ba^1_{2^\lambda, \mu}$, 
Definition \ref{d:ind-fns} below,  
there exists a regular good $\de_0$ on $\lambda$ and a surjective homorphism $\jj: \mcp(I) \rightarrow \ba$ 
such that $\de_0 = \jj^{-1}(1)$. Thus, Definition \ref{d:built} is meaningful, 
and this opens up many possibilities for ultrafilter construction.  

We now state Theorem \ref{t:separation}, used throughout the present paper, and then define ``morality'' in \ref{d:moral}. 

\br
\begin{thm-lit} \emph{(Separation of variables, Malliaris and Shelah \cite{MiSh:999} Theorem 6.13; see Observation \ref{o:upgrade} below)} \label{t:separation}
Let $\kappa \leq \lambda$. Suppose that $\de$ is built from $(\de_0, \ba, \de_*, \jj)$, and $\de_0$ is excellent.\footnote{The 
requirement ``$\de_0$ is $\lambda^+$-excellent'' is assumed in the definition of ``built from'' but we repeat it here for emphasis.} 
Then the following are equivalent:
\begin{itemize}
\item[(A)] $\de_*$ is $(\kappa, \ba, T)$-moral, i.e. $\kappa$-moral for each formula $\vp$ of $T$.
\item[(B)] For any $M \models T$, $M^\lambda/\de_1$ is $\kappa^+$-saturated.
\end{itemize}
\end{thm-lit}

\br 

The practical consequence of Theorem \ref{t:separation} is that one can construct regular ultrafilters in a two-step process. 
First, one constructs a $\lambda^+$-excellent
\emph{filter} $\de_0$ admitting the desired homomorphism $\jj$ to a specified Boolean algebra $\ba$. One may ensure 
the non-saturation half of the argument at this stage by clever choice of $\ba$: for example, a key move of \cite{MiSh:999} is 
to show that if $\ba$ has the $\mu^+$-c.c. for $\mu < \lambda$, then $\de$ cannot be good for non-low or non-simple theories, 
regardless of the choice of $\de_*$.
See Section \ref{s:why-compact} below.\mpx\footnote{Informally, there is too little room in the 
Boolean algebra to account for the ``wideness'' of theories with significant forking: see \cite{MiSh:999} \S 9, specifically Conclusion 9.10. 
That proof uses in an essential way that (in our notation) $\theta = \aleph_0$, as will be explained in due course.} See \S \ref{s:why-compact} below for the present analogue. 
Second, one builds an appropriate ultrafilter $\de_*$ on the Boolean algebra $\ba$, usually focused on the positive (saturation) side of the argument. 
Theorem \ref{t:separation} ensures that $\de_*$ controls the resulting saturation properties of $\de$.  
This will be our strategy below.

We now explain the condition of ``morality'' on $\de_*$. 

\begin{defn} \label{d:poss} \emph{(Possibility patterns, c.f. \cite{MiSh:999})}
Let $\ba$ be a Boolean algebra and $\overline{\vp} = \langle \vp_\alpha : \alpha < \lambda \rangle$ a sequence of formulas of $\ml$. 
Say that $\overline{\mb}$ is a $(\lambda, \ba, T, \overline{\vp})$-possibility when:
\begin{enumerate}
\item $\overline{\mb} = \langle \mb_u : u \in \lao \rangle$ 
\item $u \in \lao$ implies $\mb_u \in \ba^+$
\item if $v \subseteq u \in \lao$ then $\mb_u \subseteq \mb_v$ (monotonicity) and $\mb_u = 1_\ba$
\item if $u_* \in \lao$ and $\mc \in \ba^+$ satisfies
\[ \left( u \subseteq u_* \implies \left( ( \mc \leq \mb_u )  ~\lor~ ( \mc \leq 1 - \mb_u ) \right) \right) 
\]
then we can find a model $M \models T$ and $a_\alpha \in M$ for $\alpha \in u_*$ such that for every $u \subseteq u_*$,
\[ M \models (\exists x)\bigwedge_{\alpha \in u} \vp_\alpha(x;a_\alpha) ~~ \mbox{iff} ~~ \mc \leq \mb_u  .\] 
\end{enumerate}
When the sequence $\overline{\vp}$ is constant with each $\vp_\alpha = \vp$, say $\overline{\mb}$ is a $(\lambda, \ba, T, \vp)$-possibility. 
\end{defn}

\ref{d:poss} ensures that $\overline{\ma}$ could have plausibly arisen as the image under $\jj$ of 
the distribution of a $\vp$-type by asking that the Venn diagram of the elements $\ma$  
accurately reflects the complexity of $\vp$: that is, whenever some nonzero element $\mb$ of $\ba$ induces an ultrafilter on  
some $\{ \ma_{v} : v \subseteq u \}$,  
we can find a set of instances $\{ \vp_i : i \in u \}$ in a monster model of $T$ whose pattern of intersection 
corresponds exactly to that dictated by $\mb$. 

\begin{expl} \label{e:los}
Let $\de$ be built from $(\de_1, \ba, \de_0, \jj)$. 
Let $p \in S(A), A \subseteq M^\lambda/\de_1$ be a small $\vp$-type and, identifying $p$ with $\lambda$, let $f: \lao \rightarrow \de_1$ be the \los map of $p$. 
Let $\overline{\ma} = \langle \ma_u : u \in [\lambda]^{<\aleph_0} \rangle$ be given by $\ma_u = \jj(f(u))$, so $\ma_u \in \ba^+$.  
Then $\overline{\ma}$ is a $(\lambda, \ba, T, \vp)$-possibility. 
\end{expl}

Then morality, \ref{d:moral}, is simply the Boolean algebra equivalent to a regular ultrafilter being good for a theory, see $\ref{f:mult}$ and $\ref{conv:good}$ above. 

\begin{defn} \emph{(Moral ultrafilters on Boolean algebras, \cite{MiSh:999})} \label{d:moral}
We say that an ultrafilter $\de_*$ on the Boolean algebra $\ba$ is $(\lambda, \ba, T,  \overline{\vp})$-moral when
for every $(\lambda, \ba, T, \overline{\vp})$-possibility $\overline{\mb} = \langle \mb_u : u \in \lao \rangle$ 
such that $\mb_u \in \de_*$ for each $u \in [\lambda]^{<\aleph_0}$,
there is a multiplicative $\de_*$-refinement $\overline{{{\mb}^\prime}} = \langle {{\mb}^\prime}_u : u \in \lao \rangle$, i.e.
\begin{enumerate}
\item $u_1, u_2 \in \lao \implies {{\mb}^\prime}_{u_1} \cap {{\mb}^\prime}_{u_2} = {{\mb}^\prime}_{u_1 \cup u_2}$
\item $u \in \lao \implies {{\mb}^\prime}_u \subseteq \mb_u$
\item $u \in \lao \implies {{\mb}^\prime}_u \in \de_*$.
\end{enumerate}
We write $(\lambda, \ba, T, \Delta)$-moral to mean $(\lambda, \ba, T, \vp)$-moral for all $\vp \in \Delta$. We write 
$(\lambda, \ba, T)$-moral to mean $(\lambda, \ba, T, \vp)$-moral for all formulas $\vp$. 
\end{defn}

In the last sentence of the definition of ``moral'', we could equivalently have said: 
we write $(\lambda, \ba, T)$-moral to mean $(\lambda, \ba, T, \overline{\vp})$-moral 
for all $(\lambda, \ba, T, \overline{\vp})$-possibilities, because our theories $T$ are countable. 
(On uncountable $T$, see \cite{Sh:14}.)
The equivalence is by Fact \ref{phi-types} and Theorem \ref{t:separation}. Since in the present paper 
it is not usually necessary to restrict to $\vp$-types, we will often use this second formulation.  

The statement of \cite{MiSh:999} Theorem 6.13 was stated just for the case $\kappa = \lambda$. For completeness, 
we justify the use of $\kappa \leq \lambda$ in Theorem \ref{t:separation} above.

\begin{obs} \label{o:upgrade} 
Fix $|I| = \lambda$ and $\kappa \leq \lambda$. Then (1) iff (2). 
\begin{enumerate}
\item[(1)] $\de$ is a regular ultrafilter on $I$, built from $(\de_0, \ba, \de_*)$, and $\de_*$ is $(\kappa, \ba, T)$-moral 
for all formulas $\vp$ of $T$. 
\item[(2)] For any model $M \models T$, and any type $p \in \ts(N)$ where $N \subseteq M$, $||N|| = \kappa$, 
$p$ is realized in $M$.
\end{enumerate}
\end{obs}

\begin{proof}
(1) implies (2): This is the direction we use in the present paper.  Recall that ``built from'' implies $\de_0$ is $\lambda^+$-excellent. 
By regularity of $\de$, we may choose any model $M$ of $T$, in particular we may choose $M$ $\lambda^+$-saturated. 
By Fact \ref{phi-types}, we may assume $p$ is a $\vp$-type. Let 
$\langle \vp(x,\bar{a}^*_\alpha) : \alpha < \kappa \rangle$ be an enumeration of $\vp$.  
Fix some lifting of the parameters so that we may write ``$a[t]$'' for $a \in N$ and $t \in I$. 
For each $u \in [\lambda]^{<\aleph_0}$, define 
\[ B_u = \{ t \in I : (\exists x)\bigwedge_{\alpha \in u} \vp_\alpha(x,\bar{a}_{v_\alpha}[t]) \}.\]
Without loss of generality, $B_\emptyset = I$.
Define $\langle A_u : u \in [\lambda]^{<\aleph_0} \rangle$ by: $A_u = B_{u \cap \kappa}$. 
Let $\ma_u = \jj(A_u) \in \de_*$, which gives us the sequence $\bar{\ma} = \langle \ma_u : u \in [\lambda]^{<\aleph_0} \rangle$, 
which is a possibility pattern, by \lost theorem, c.f. \ref{e:los}. By hypothesis (1), there exists 
$\bar{\ma}^\prime = \langle \ma^\prime_u : u \in [\lambda]^{<\aleph_0} \rangle$ such that $\bar{\ma}^\prime$ is a 
sequence of elements of $\de_*$ which form a 
multiplicative refinement of $\bar{\ma}$. 
For each $u \in [\lambda]^{<\aleph_0}$, choose $A^\prime_u$ such that $\jj(A^\prime_u) = \ma^\prime_u$. Let 
$A^{\prime\prime}_u = A^\prime_u \cap A_u$. Then the sequence $\langle A^{\prime\prime}_u : u \in [\lambda]^{<\aleph_0} \rangle$
refines $\langle A_u : u \in [\lambda]^{<\aleph_0} \rangle$ and is multiplicative mod $\de_0$. 
Now we use the definition of excellent, specifically Claim 4.9(1) of \cite{MiSh:999}, using $\de_0$ and $\bar{A}^{\prime\prime}$ here for 
$\de$ and $\bar{A}$ there. By that Claim, there is a sequence $\bar{B}^\prime = \langle B^\prime_u : u \in [\lambda]^{<\aleph_0} \rangle$ 
such that first, $\bar{B}^\prime$ refines $\bar{A}^{\prime\prime}$ so a fortiori $\bar{B}^\prime$ refines $\bar{A}$, and second, 
$\bar{B}^\prime$ is actually multiplicative, not just multiplicative mod $\de_0$.  The map
$f: [\kappa]^{<\aleph_0} \rightarrow \de$ given by $u \mapsto B^\prime_u$ is therefore a multiplicative map, which means that 
for each $t \in I$, the set 
\[ \{~ \vp(x, \bar{a}^*_{\alpha}[t]) : t \in f(\{ \alpha \}) ~\} \]
is a partial type in $M$. Since $M$ is $\lambda^+$-saturated, we may choose some $b_*[t]$ realizing this type. Let $b_* = \prod_{t \in I} b_*[t]/\de$. 
Then $b_*$ realizes $p$ as desired.

(2) implies (1):  This is immediate from Lemma 6.12 of \cite{MiSh:999} replacing $\lambda$ by $\kappa$ in conclusions (A) and (B) of that lemma 
and in the corresponding proof. 
\end{proof}

\subsection{Why a large cardinal?} \label{s:why-compact}

The cardinal $\sigma$ is supercompact iff on every set there exists a normal $\sigma$-complete ultrafilter 
(see \ref{f:sct}). This implies that $\sigma$ is compact, i.e. that every $\sigma$-complete filter can be 
extended to a $\sigma$-complete ultrafilter. 

Where do complete filters appear, given that all ultrafilters in Keisler's order are regular? 
The idea is that Theorem \ref{t:separation} allows us to build regular ultrafilters from complete ones: 
$\de_*$ may be $\sigma$-complete for some uncountable $\sigma$, assuming the existence of $\sigma > \aleph_0$ compact. 

Why is this useful? In the main theorem of \cite{MiSh:999}, we proved existence of a ZFC dividing line in Keisler's order among the unstable theories, 
by separating the minimum unstable theory, the random graph, from all non-simple and simple non-low theories. 
The non-saturation half of that argument proved, in the context of Theorem \ref{t:separation}, that when 
the quotient Boolean algebra is $\ba = \ba^1_{2^\lambda, \mu, \aleph_0}$ and $CC(\ba) = \mu^+ \leq \lambda$, 
i.e. the maximal size of an antichain in $\ba$ is $\mu$, 
then the resulting $\de_1$ was not good for any non-low or non-simple theory, see also \ref{c:uf-ba2} below. 
It is crucial there that $\ba$ is the completion of a free Boolean algebra and 
that the last of the three cardinal subscripts for $\ba$ is $\aleph_0$, so in our notation, $\sigma = \theta = \aleph_0$.  
[The saturation part of the proof showed that the lack of global inconsistency in the random graph 
meant that its types could still be realized when $\mu$ was small.]  
To the extent that the random graph is typical of simple theories, one can ask whether higher octaves of those arguments  
would work to separate all simple theories from all non-simple ones. 
Our strategy here is, therefore, to continue working with completions of free Boolean algebras, and to 
continue to concentrate on the case where $CC(\ba) = \mu^+ \leq \lambda$.  
(We will consider other Boolean algebras in the paper \cite{MiSh:F1484} in preparation.) 
However, the ultrafilter $\de_*$ we construct is $\sigma$-complete for some uncountable $\sigma$,  
so in our present notation $\theta \geq \sigma > \aleph_0$, in order 
to have a chance at saturating simple theories which are non-low. The large cardinal assumption 
gives us enough room in the construction to deal with the extra amount of forking in simple non-low theories, 
while still allowing us to ensure non-saturation of any non-simple theory.  
The remarkable fact is that, after taking care of this one possible problem at lowness, we are able to leverage a new analysis of amalgamation in 
simple theories to build ultrafilters which precisely characterize the dividing line at simplicity.  

We consider both $\sigma = \aleph_0$ and also $\sigma$ uncountable and supercompact in our various ultrafilter existence proofs. 
We use the second case in this paper to characterize simplicity, and will use the first [which necessarily does not saturate non-low simple theories, but is in ZFC] in \cite{MiSh:1050}.

\subsection{Structure of the paper}
The structure of the paper is as follows.  
We assume throughout that our tuples $(\lambda, \mu, \theta, \sigma)$ of cardinals are suitable in the sense of 
\ref{d:suitable} above.   
In \S \ref{s:e-simple}, we develop the model-theoretic amalgamation condition called ``$(\lambda, \mu, \theta, \sigma)$-explicitly simple.''
In \S \ref{s:s-e-s}, we characterize simple theories as explicitly simple using $\mu^+ = \lambda$. As discussed there and carried further in \cite{MiSh:1050}, varying the distance of $\mu$ and $\lambda$ outlines a new approach to  classifying the simple unstable theories. 
In \S \ref{s:optimal}, we define the new property of ultrafilters on certain Boolean algebras, called ``$(\lambda, \mu, \theta, \sigma)$-{optimal},'' 
and prove an existence theorem assuming $\sigma$ is uncountable and supercompact.  
If $\de$ is a regular ultrafilter on $\lambda$ built from $(\de_0, \ba, \de_*)$ where $\de_*$ is 
$(\lambda, \mu, \theta, \sigma)$-{optimal}, we will call $\de$ $(\lambda, \mu, \theta, \sigma)$-\emph{optimized}. 
Assuming $\mu < \lambda$, we then show how to ensure {optimized} ultrafilters do not saturate non-simple theories. 
\S \ref{s:skolem} proves a technical lemma about arranging presentations to interact well with liftings in ultrapowers. 
In \S \ref{s:ultrapower}, assuming $\mu^+ = \lambda$ (as well as $\sigma$ uncountable and supercompact to quote the 
ultrafilter existence theorem of \S \ref{s:optimal}), we prove that optimized ultrafilters saturate simple theories. 
\S \ref{s:main-theorems} contains the paper's main theorems, characterizing simple theories via saturation of ultrapowers.  
\S \ref{s:perfect} states and proves existence of so-called perfect ultrafilters on certain Boolean algebras, 
which will be useful for $\sigma = \aleph_0$ in future papers.
\S \ref{s:questions} contains a list of open problems. 

\br

\subsection{Basic definitions}
For history on simple unstable theories, and for statements of theorems from the literature, we refer to the survey article \cite{GIL}. We will use: 

\begin{defn}[Simple theories] \label{d:simple} Given a background theory $T$,
\begin{enumerate}
 \item A formula $\vp = \vp(x,y)$ has the \emph{$k$-tree property}, for $k<\omega$, 
when there exist parameters $\{ a_\eta : \eta \in {^{\omega>}\omega} \}$, $\ell(a_\eta) = \ell(y)$, so that:
 \begin{enumerate}
 \item for each $\eta \in {^{\omega>}\omega}$, the set $\{ \vp(x,a_{\eta^\smallfrown i}) : i < \omega \}$ is $k$-inconsistent 
 \item for each $\eta \in {^{\omega}\omega}$, the set $\{ \vp(x,a_{\eta|n}) : n < \omega \}$ is consistent.
 \end{enumerate}
 \item A formula $\vp$ is \emph{simple} if it does not have the tree property, i.e. it does not have the $k$-tree property for any $k$. 
 \end{enumerate}
A theory is called simple if all of its formulas are.
\end{defn}

\begin{defn}[D-rank, lowness] \label{d:rank} Again fix $T$. 
\begin{enumerate}

\item For each formula $\vp(\bar{x}, \bar{y})$, an integer $k<\omega$, and a formula $\theta(\bar{x})$,
all possibly with parameters, we define $D(\theta, \vp, k)$ 
to be $\geq 0$ if $\theta(\bar{x})$ is consistent, and $\geq \alpha + 1$ if there exists $\bar{a}_\alpha$ 
which forks over the parameters of $\theta$ 
such that $D( \theta(\bar{x}) \land \vp(x,\bar{a}_\alpha), \vp, k) \geq \alpha$. 
Equivalently, $T$ is simple if and only if for all formulas $\vp$ and $\theta$ and all $k<\omega$, 
the rank $D(\theta, \vp, k)$ is finite. 

\item We say $T$ is \emph{low} if for each formula $\vp(\bar{x};\bar{y})$ there is $k<\omega$
such that for any indiscernible sequence $\langle \bar{a}_n : n < \omega \rangle$, with $\ell(\bar{a}_n) = \ell(y)$, 
we have that $\{ \vp(\bar{x}; \bar{a}_n) : n < \omega \}$ is consistent iff it is $k$-consistent. 

\end{enumerate}
\end{defn}

\begin{thm-lit}[Independence theorem, version of \cite{GIL} Theorem 2.11]  \label{t:ind-thm}
Let $T$ be simple and $M \models T$.  Let $A$, $B$ be sets such that 
$tp(A/MB)$ does not fork over $M$. Let $p \in S(M)$. Let $q$ be a nonforking extension 
of $p$ over $MA$ and $r$ be a nonforking extension of $p$ over $MB$. Then 
$q \cup r$ is consistent, moreover $q \cup r$ is a nonforking extension of $p$ over $MAB$. 
\end{thm-lit}

\begin{defn}\emph{(Partitions)} \label{d:prtn}
\begin{enumerate}
\item A partition of a Boolean algebra is a maximal set of pairwise disjoint nonzero elements.  
We may also apply this to sequences with no repetitions. 
\item $CC(\ba) = \sup \{ \mu^+ : \ba \mbox{ has a partition of size $\mu$ }\}$.
\item When $\ma \in \ba$ and $\bar{\mc} = \langle \mc_\epsilon : \epsilon < \mu \rangle$ is a partition of $\ba$, we say that $\bar{\mc}$ \emph{supports} $\ma$
when $\epsilon < \mu \implies (\mc_\epsilon \leq \ma) \lor (\mc_\epsilon \leq 1-\ma)$ ~ $($in $\ba$$)$.    

\item When $\bar{\mc} = \langle \mc_\zeta : \zeta < \mu \rangle$, $\bar{\xd} = \langle \xd_\epsilon : \epsilon < \mu \rangle$ are partitions of $\ba$, 
say that $\bar{\md}$ \emph{refines} $\bar{\mc}$ if for each $\epsilon < \mu$, there is $\zeta < \mu$ such that 
$\xd_\epsilon \leq \mc_\zeta$.
\end{enumerate}
\end{defn}

We focus on completions of free Boolean algebras, mainly $\ba = \ba^1_{2^\lambda, \mu, \theta}$, the completion of the 
Boolean algebra generated freely by $2^\lambda$ independent partitions of size $\mu$, where intersections of fewer than $\theta$ nonzero elements are nonzero 
precisely when no two are from the same partition. It will be convenient to describe such objects as follows. 

\begin{defn}\emph{(Boolean algebra notation\footnote{Following \cite{Sh:c} VI \S 3 or \cite{MiSh:997}.  
``$\fin$'' recalls the simplest case $\theta = \aleph_0$, i.e. ``finite intersection.''})} \label{d:ind-fns}
Let $\alpha$ be an ordinal, $\mu \geq \theta$ cardinals; the existence statement is $\ref{fact-iff}$.
\begin{enumerate}
\item Let $\fin_{\mu, \theta}(\alpha) = $
\[ \{ h : h\mbox{ is a function, $\dom(h) \subseteq \alpha$, $|\dom(h)| < \theta$ and $\rn(h) \subseteq \mu$} \} \]
\item $\ba^0 = \ba^0_{\alpha, \mu, \theta}$ is the Boolean algebra generated by:
\\ $\{ \mx_f : f \in \fin_{\mu,\theta}(\alpha) \}$ freely subject to the conditions that
\begin{enumerate}
\item $\mx_{f_1} \leq \mx_{f_2}$ when $f_1 \subseteq f_2 \in \fin_{\mu, \theta}(\alpha)$. 
\item $\mx_f \cap \mx_{f^\prime} = 0$ if $f, f^\prime$ are incompatible 
functions.\footnote{Note that `iff' follows. It also follows that when $j < \theta$, $g = {\bigcup_{i<j}{f_i}}$ implies $\mx_g = \bigcap_{i<j} \mx_{f_i}$ in $\ba^0$ and in $\ba^1$.} 
\end{enumerate}
\item $\ba^1_{\alpha, \mu, \theta}$ is the completion of $\ba^0_{\alpha, \mu, \theta}$. 
\end{enumerate}
\end{defn}

\begin{conv}
We will assume that giving $\ba$ determines a set of generators 
$\langle \mx_f : f \in \fin_{\mu, \theta}(\alpha_*)\rangle$, so also $\alpha_*$, $\mu$, $\theta$. 
\end{conv}

\begin{fact} \label{fact-iff}
Assuming $\lambda = \lambda^{<\theta}$, 
$\ba^0_{2^\lambda, \mu, \theta}$ and thus its completion exists. 
\end{fact}

\begin{proof}
See Engelking-Karlowicz \cite{ek}, Fichtenholz and Kantorovich\cite{f-k}, Hausdorff \cite{hausdorff}, 
or Shelah \cite{Sh:c} Appendix, Theorem 1.5. 
\end{proof}

\br

\begin{conv}[Conventions on notation] 
\label{conv-not}
Some effort has been made to standardize notation as follows $($these objects will be defined below, and will be subject to further hypotheses$)$.  
The reader can quickly scan the following list at this point, and refer back to it later on as needed. 
\begin{itemize}
\item The letters $\de, \ee$ always indicate a filter.  

\item When occurring together, the symbols $\de, \de_0, \de_*, \jj$, $I$ 
are used in compliance with Theorem \ref{t:separation} p. \pageref{t:separation}. 

\item $\ba$ is a Boolean algebra; in the proofs, it is always a completion of a free Boolean algebra, so of the form 
$\ba^1_{2^\lambda, \mu, \theta}$, as defined in \ref{d:ind-fns}.  

\item $\ba^+$ is $\ba \setminus \{ 0 \}$. 

\item When $\de$ is a filter on $\ba$, $\de^+ = \{ \ma \in \ba : \ma \neq 0 \mod \de \}$. 
\item $\lambda \geq \mu \geq \theta \geq \sigma$ are suitable infinite cardinals $(\ref{d:suitable})$, where:
\begin{itemize}
\item $\lambda$ is the size of the index set for our background regular ultrafilter $\de_1$, thus, we are interested in realizing types 
in simple theories over sets of cardinality $\leq \lambda$.
\item $\mu \leq \lambda$, in the interesting case $\mu < \lambda$: this is the range of the coloring function we build on fragments of types in simple 
theories, and also the size of a maximal antichain in our Boolean algebra $\ba$. 
\item $\theta$ $($note $\sigma \leq \theta \leq \mu$$)$ is the last parameter for the underlying Boolean algebra $\ba = \ba^1_{2^\lambda, \mu, \theta}$, 
see $\ref{d:ind-fns}$.
\item $\sigma$ is $\aleph_0$ or an uncountable supercompact cardinal, and we build $\de_*$ to be $\sigma$-complete. 
\end{itemize}
We have kept $\theta$ and $\sigma$ separate due to their different roles and requirements, but the casual  
reader will not lose much by assuming they are equal. 
\item Boldface letters $\mc, \mx, \mb \dots$ are elements of $\ba$. 
\item Fraktur letters are generally used for objects of interest having multiple parts, 
e.g. $\xm$ for presentations, $\xr$ for elements of the set of type fragments $\mcr_{\xm}$ associated to a  
presentation $\xm$. 
\item $f, f_1, f_2 \dots$ are elements of $\fin_{\mu, \theta}(\alpha_*)$, noting that
$\mx_f \in \ba$ is an element corresponding to the function $f$ as in $\ref{d:ind-fns}$.
\item $\Omega \subseteq [\lambda]^{<\sigma}$ is stationary, which means cofinal if $\sigma = \aleph_0$. 
\item $u, v, w$ are subsets of $\lambda$; generally $u \in \Omega$, so $|u| < \sigma$, whereas $w$, $v$ may be larger. 
\item $\epsilon, \zeta, \xi$ are elements of $\mu$, i.e. ordinals $< \mu$.
\item $\delta_{\xm}$ is an ordinal $\leq |T|$, usually $\clm(\emptyset)$ in the context of a presentation $\xm$. 
\end{itemize}
\end{conv}

\newpage
\setcounter{theoremcounter}{0} \section{Definition of ``explicitly simple''} \label{s:e-simple}

In this section and the next we develop a new perspective on simplicity. 

This section gives the first main definition of the paper: ``the theory $T$ is $(\lambda, \mu, \theta, \sigma)$-explicitly simple.''  
The definition makes sense for any suitable four-tuple of infinite cardinals $\lambda \geq \mu \geq \theta \geq \sigma$ 
in the sense of \ref{d:suitable}, and so varying these cardinals will give information about the theory.  
The parameter we are mainly interested in varying is $\mu$. 
As mentioned, it will follow from the definitions in this section that    
$(\lambda, \mu, \theta, \sigma)$-explicitly simple becomes weaker as $\mu$ increases, and that every 
$(\lambda, \mu, \theta, \sigma)$-explicitly simple theory is simple, even when $\mu = \lambda$.  

Recall from \ref{d:dist} that when analyzing saturation of ultrapowers, 
\lost theorem guarantees that while projections of finite pieces of a type to a given index model 
may each be consistent, their `relative position' is a priori not preserved, so there is no guarantee 
that the union of these pieces is consistent. 

An informal model-theoretic description of this problem is the following. 
Suppose, for clarity, that $T$ is a theory whose only forking comes from equality, and $p$ is a type over a set of size $\lambda$.  
Suppose that finitely many finite 
pieces of the type are moved by piecewise automorphisms of the monster model 
agreeing on common intersections and introducing no new forking. 
Is the union of these automorphic images consistent? 
Not necessarily: consider the effect of piecewise automorphisms 
$f, g, h$ on three formulas $\{ R(x,a,b) \}$, $\{ R(x,b,c) \}$, $\{ R(x,a,c) \}$ in the generic tetrahedron-free three-hypergraph 
where despite $f(a)=h(a), f(b) = g(b), g(c) = h(c)$ we may have   
$R(f(a), g(b), h(c))$. So instead we may try to gauge the complexity of the `amalgamation problems' arising under such partial automorphisms 
by asking: 
can we color the pieces $[p]^{<\aleph_0}$ with no more than $\mu$ colors in such a way that within each color class, 
after piecewise automorphism, 
the union is always consistent? 
Note that when $\mu = \lambda$ there is trivially a coloring, as each piece gets its own color.  
To make the question precise, one will want to add some clarifying hypotheses, such as closure conditions on the finite pieces, 
and in the general case, some natural conditions on forking.   
After doing so, however, the question is whether a non-trivial coloring exists ($\mu < \lambda$).

Our picture 
is that all simple theories 
are in some sense close to what we see in these generic hypergraphs: 
the noise arising from forking may be muted so that the basic amalgamation problems 
controlling consistency rise to the surface. 
Enumerating each $p$ in such a way that an algebra defined on its indices captures 
this additional noise, a precise general formulation of this partial-automorphism 
condition ``$T$ is $(\lambda, \mu, \theta, \sigma)$-explicitly simple'' may be given. 
The first main theorem of the paper, which we prepare for here and prove in the next section, will prove that we may 
essentially always find such a coloring when $T$ is simple and $\mu^+ = \lambda$ (so using the first nontrivial number 
of colors), and moreover that this characterizes simplicity of $T$. 

\begin{cont} \label{cont:es} In this section we make the following assumptions.
\begin{enumerate}
\item $\sigma \leq \theta \leq \mu \leq \lambda$ are 
suitable in the sense of $\ref{d:suitable}$. 
The reader may wish to assume $\sigma = \theta$. 
\item $T$ is a 
complete first order theory, with infinite models, and $|T| < \sigma$. The definition of `explicitly simple' will entail that 
$T$ is simple, i.e. $\kappa(T)$ exists. 
\item $\mathfrak{C} = \mathfrak{C}_T$ is the monster model for $T$.
\item For transparency, $T$ eliminates imaginaries, i.e. $T = T_*^{eq}$ for some complete theory $T_*$. In particular, we assume that whenever 
$M\models T$, every finite sequence of elements of $M$ is coded by some $a \in \dom(M)$. Otherwise, write $T^{eq}$ and $M^{eq}$ throughout.\npx\footnote{For 
general $T$ this assumption indeed makes things clearer, although for certain specific $T$ it may be more transparent to stick to elements. 
The use of imaginaries poses no problems in ultrapowers, since ultrapowers commute with reducts, so there is no issue in passing to a larger theory and proving realization of types there. We use imaginaries below in the definition of the algebra, as $\langle a^*_\alpha : \alpha < \lambda \rangle$ is allowed to be a sequence of imaginaries. However, by use of a more complex indexing scheme and 
algebra, this assumption could straightforwardly be avoided, as is done for the notationally simpler case of certain 
hypergraphs in \cite{MiSh:1050}.} 
\item ``Independent'' means nonforking and ``dnf'' means does not fork. 
\end{enumerate}
\end{cont}

We begin by stating the organizing definition. We will define the key items ``$\xm$ is a presentation,'' 
``$\xn$ refines $\xm$,''
the set of type fragments ``$\mcr_\xm$'' associated to $\xm$, and ``$G: \mcr_\xm \rightarrow \mu$ is an intrinsic coloring'' over the course of the section, 
in \ref{d:pres}, \ref{d:extend}, 
\ref{d:es3}, \ref{d:es4} respectively.  
\ref{d:c3} makes sense because we will prove the existence of presentations $\xm$ for all 
simple theories in Section \ref{s:s-e-s}.

\begin{defn}[Explicitly simple] \label{d:c3} 
Assume $(\lambda, \mu, \theta, \sigma)$ are suitable. We say 
$T$ is $(\lambda, \mu, \theta, \sigma)$-\emph{explicitly simple} if $T$ is simple and 
for every $N \models T$, $||N|| = \lambda$, $p \in \ts(N)$ nonalgebraic, 
\begin{enumerate}
\item[(a)] there exists a presentation $\xm$ of $p$.

\item[(b)] for every presentation $\xm$ of $p$, there is a presentation $\xn$ of $p$ refining $\xm$ and a function $G: \mcr_{\xn} \rightarrow \mu$ 
such that $G$ is an intrinsic coloring of $\mcr_\xn$. 
\end{enumerate}
\end{defn}

\br
\noindent Next we define a presentation of a type. 
This will essentially be the data of a certain enumeration of that type along with 
an algebra to capture nonforking and amalgamation bases. 
By `algebra' on $\lambda$ we mean a first order structure with functions and no relations whose domain is $\lambda$.  
The closure of a set $u \subseteq \lambda$ in such an algebra $\zm$, denoted $\clm(u)$, is the substructure generated by $u$, so 
$u \subseteq \clm(u) = \clm(\clm(u))$.  We also give a value to $\clm(\emptyset)$.  

\begin{defn} \label{d:pres}
Suppose we are given $N \models T$, $||N|| = \lambda$, and $p \in \ts(N)$. 
A \emph{$(\lambda, \theta, \sigma)$-presentation} for $p$ is the data of an enumeration and an algebra, 
\[ \xm = (\langle \vp_\alpha(x, a^*_\alpha) : \alpha < \lambda \rangle, \zm ) \]
where these objects satisfy: 

\begin{enumerate}
\item $p = \langle \vp_\alpha(x;a^*_\alpha) : \alpha < \lambda \rangle$ is an enumeration of $p$, 
which induces an enumeration $\langle a^*_\alpha : \alpha < \lambda \rangle$ of $\dom(N)$, possibly with repetitions, and with the $a^*_\alpha$ 
possibly imaginary. 

\item $\zm$ is an algebra on $\lambda$ with $<\theta$ functions. 

\item For any finite $u \subseteq \lambda$, $|\clm(u)| < \sigma$. Thus, for any $u \subseteq \lambda$, 
if $|u| < \sigma$ then $|\clm(u)| < \sigma$, and if $|u| < \theta$ then $|\clm(u)| < \theta$.

\item $\clm(\emptyset)$ is  
an infinite cardinal $\leq |T|$, so an initial segment of $\lambda$. 
\\ $M_* := N \rstr \{ a^*_\alpha : \alpha < \clm(\emptyset) \}$ is a distinguished elementary submodel of $N$, 
and we require that $p$ does not fork over $M_*$. 

\item Moreover, for each $u \in [\lambda]^{<\sigma}$, $N_u := N \rstr \{ a^*_\alpha : \alpha \in \clm(u) \}$ is an elementary submodel of $N$, 
and $\{ \vp_\alpha(x,a^*_\alpha) : \alpha \in \clm(u) \}$ is a complete type over this submodel which dnf over $M_*$. $($In particular, 
$\{ \vp_\alpha(x, a^*_\alpha) : \alpha \in \clm(\emptyset) \}$ is a complete type over $M_*$.$)$

\item If $\alpha \in \clm(u)$, $\beta \leq \alpha$, writing $A_\beta = \{ a^*_\gamma : \gamma < \beta \}$, we have that  
\\ $\tp(a^*_\alpha, A_\beta \cup M_*)$ does not fork over $\{a^*_\gamma : \gamma \in \clm(u) \cap \beta \} \cup M_*$.

\end{enumerate}
\end{defn}

\begin{rmk} \label{d:es2}
From a presentation $\xm$, 
the following were unambiguously defined: 
$M_*$ in item 5, $\clm(\emptyset)$ in item 5, 
$N_u$ in item 4 for any $u \subseteq \lambda$, $A_\alpha$ in item 6 for any $\alpha < \lambda$.\footnote{Although it is already a global assumption for the section, 
note that together items 3, 4, and 5 and the fact that $\sigma$ is strongly inaccessible imply that $|T| < \sigma$ essentially, i.e. identifying two 
non-logical symbols under the relation of equivalence modulo $T$. 
That is, if $T$ is a complete first order theory, $T$ has a model $M = M_*$ of cardinality 
$\delta < \sigma$, and $E$ is the equivalence relation on $\tau(T)$ which identifies functions, resp. predicates, iff they have the same interpretation in $M$, 
then we may conclude $E$ has $\leq 2^\delta < \sigma$ classes.} 
\end{rmk}

Although we don't pursue this approach in the present paper, it is worth noting that in 
Definition \ref{d:pres}, for certain less complicated theories (e.g. $T$ with no function symbols and 
trivial forking, as is the case in \cite{MiSh:1050}) 
we might prefer to allow $\clm(u)$ to be a set, rather than requiring it to be a submodel, and in particular to 
only require of $\clm(\emptyset)$ that the following observation holds. 

\begin{obs}[$($Independence theorem over $\clm(\emptyset)$$)$, see \cite{GIL}, 2.13, p. 17]
By the definition of presentation and the simplicity of $T$, the following will be true for any presentation $\xm$. 
If $\ell = 1, 2$ are such that:  
\begin{enumerate} 
\item $\clm(\emptyset) \subseteq u_\ell \subseteq \lambda$ 
\item $u_\ell = \clm(u_\ell)$, \emph{thus} $u_1 \cap u_2 = \clm(u_1 \cap u_2)$ 
\item $A_\ell = \{ a^*_\alpha : \alpha \in u_\ell \}$ and $A_1$ is independent from $A_2$ over $A_1 \cap A_2$
\item $p_\ell \in \ts(A_\ell)$ dnf over $\{ a^*_\alpha : \alpha \in \clm(\emptyset) \}$ and 
$p_\ell \supseteq p \rstr \{ a^*_\alpha : \alpha \in A_\ell \}$ 
\end{enumerate}
then $p_1 \cup p_2$ is a consistent type which does not fork over $\{ a^*_\alpha : \alpha \in \clm(\emptyset) \}$.
\end{obs}

\begin{defn}[Refinements of presentations] \label{d:extend} 
Suppose we are given $N \models T$, $||N|| = \lambda$, and $p \in \ts(N)$. Let $\xm = (\bar{\vp}_\xm, \zm_\xm)$, 
$\xn = (\bar{\vp}_\xn, \zm_\xn)$ be presentations of $p$. 
We say that \emph{$\xn$ refines $\xm$} when: 
\begin{enumerate}
\item[(a)] $\bar{\vp}_\xm = \bar{\vp}_\xn$. 
\item[(b)] $\operatorname{cl}_{\zm_\xm}(\emptyset) = \operatorname{cl}_{\zm_\xn}(\emptyset)$. 
\item[(c)] $\zm_\xm \subseteq \zm_{\xn}$. 
\end{enumerate}
\end{defn}

Since we allow the sequence $\langle a^*_\alpha : \alpha < \lambda \rangle$ to contain repetitions, some care was needed in the definition 
of the models $N_u$: the set $\{ \alpha : a_\alpha \in |N_u| \}$ could have size $\lambda$, although $\dom(N_u)$ has cardinality $<\theta$.  
Note also that the cardinal $\theta$ has two roles: first, the size of $\clm(u)$ thus $||N_u||$ is $<\theta$, call this $\theta_1$, and second, 
in $\fin_{\mu,\theta}(\alpha_*)$ in \S \ref{s:ultrapower}, call this $\theta_2$.
We don't separate them here, but what we use is $\theta_1 \leq \theta_2$, and we could have used $\sigma = \theta_1 < \theta_2$. 

Let us emphasize that we have included simplicity of $T$ in definition of explicitly simple, to avoid trivial satisfaction of the hypotheses: 

\begin{obs} \label{c:look} 
Let $T$ be a theory and let $(\lambda, \mu, \theta, \sigma)$ be suitable infinite cardinals. 
Suppose that for every $N \models T$ of size $\lambda$ every nonalgebraic types $p \in \ts(N)$ has a $(\lambda, \theta, \sigma)$-presentation. 
Then $T$ is simple. 

A fortiori, if $T$ is $(\lambda, \mu, \theta, \sigma)$-explicitly simple, 
then $T$ is simple. 
\end{obs}

\begin{proof} By the definition. 
\end{proof}

Nonetheless, the assumption of simplicity is natural because we assume that $p$ does not fork over a small set. Recall that: 

\begin{fact} \label{kappa-of-T}
Let $T$ be a complete theory. Then $T$ is simple iff $\kappa(T)$ exists iff $\kappa(T) \leq |T|^+$, where  
\[ \kappa(T) = \min \{ \kappa : \mbox{ if } A \subseteq \mathfrak{C}, q \in \ts(A) \mbox{ then } q \mbox{ dnf over some } B \subseteq A, |B| < \kappa \}.\] 
 \end{fact}

\begin{proof}
See Theorems 3.4 and 3.6 of \cite{GIL}. 
\end{proof} 


We now arrive to the right general analogue of a `fragment of a type.' Its ingredients are a set of indices $u$, 
a closed set $w \supseteq u$ (containing e.g. forking of $u$), the type of a model in the variables $\bar{x}_w$, 
and a type over that model in the variables $x,\bar{x}_w$, satisfying some additional conditions suitable 
to automorphic images of pieces of $p$. 

\begin{defn} \label{d:es3} \emph{(The set of quadruples $\mcr_{\xm}$)}
Let $\xm$ be a presentation of a given type $p = p_{\xm}$.  
Then $\mcr = \mcr_{\xm}$ is the set of $\xr = (u, w, q, r)$ such that:
\begin{enumerate}
\item $u \in [\lambda]^{<\sigma}$, $w \in [\lambda]^{<\theta}$ and $w = \clm(w)$.
\item $u \subseteq \clm(u) \subseteq w$. 
\item $q = q(\overline{x}_w)$ is a complete type in the variables $\overline{x}_{w}$ such that:
\begin{enumerate}
\item for any finite $v \subseteq \clm(\emptyset)$, if $M_* \models \psi(\overline{a}^*_v)$ then $\psi(\overline{x}_v) \in q$.
\item for any finite $\{ \alpha_0, \dots, \alpha_n \} \subseteq u$, $\exists x \bigwedge_{i\leq n} \vp_\alpha(x,a^*_\alpha) ~\in q$.
\end{enumerate} 
\item $r = r(x,\overline{x}_w)$ is a {complete type} in the variables $x, \overline{x}_{w}$, 
extending 
\[  q(\overline{x}_w) \cup \{ \vp_\alpha(x,x_\alpha) : \alpha \in u \}. \] 
\item \label{here} \underline{if} $\overline{b}^*_{w}$ realizes $q(\overline{x}_{w})$ in $\mathfrak{C}_T$ and $\alpha < \clm(\emptyset) \implies b^*_\alpha = a^*_\alpha$,
\underline{then} 
\begin{enumerate}
\item $r(x,\overline{b}^*_{w})$ is a type which does not fork over $M_*$ and extends $p\rstr M_*$.
\item if $w^\prime \subseteq w$ is $\mlx$-closed, $\mathfrak{C}_T \rstr \{ b^*_\alpha : \alpha \in w^\prime \} \preceq \mathfrak{C}_T$ and 
$r(x,\overline{b}^*_{w}) \rstr \overline{b}^*_{w^\prime}$ is a complete type over this elementary submodel.
\item if $w^\prime \subseteq w$ is $\mlx$-closed and $\alpha \in w^\prime$ then $tp(b^*_\alpha, \{ b^*_\beta : \beta \in w \cap \alpha\})$ 
dnf over $\{ b^*_\beta : \beta \in w^\prime \cap \alpha \}$. 
\end{enumerate}
\end{enumerate}
\end{defn}

Note that for tuples in $\mcr$, the type $r$ is like $p$ in the sense of being a nonforking extension of $p \rstr M_*$ to a set including the domain of $N_u$, however 
this type is not guaranteed to be ``correct'' on all of $w$. 
Since the definition $\mcr_{\xm}$ is fairly unconstrained, 
in comparing elements of this set we will be most interested in cases which avoid trivial inconsistency.  

\begin{defn} \label{d:good-inst}
Suppose we are given $\overline{\xr} = \langle \xr_t = (u_t, w_t, q_t, r_t) : t < t_* < \sigma \rangle$ from $\mcr_{\xm}$. 
Say that $\overline{b}^* = \langle b^*_\alpha : \alpha \in \bigcup_{t} w_{t} \rangle$, with each $b^*_\alpha \in \mathfrak{C}$ $($possibly imaginary$)$, 
is a \emph{good instantiation} of $\overline{\xr}$ when the following conditions hold. 
\begin{enumerate}
\item $\alpha \in \clm(\emptyset) \implies b^*_\alpha = a^*_\alpha$. 
\item for each $t < t_*$, $\overline{b}^*\rstr_{w_{t}}$ realizes $q_t(\overline{x}_{w_{t}})$.
\item for each $t < t^\prime < t_*$, if $v \subseteq w_t \cap w_{t^\prime}$ is finite, then:
\begin{enumerate}
\item for each formula $\psi(\overline{x}_v)$, 
$\psi(\overline{b}^*_v) \in q_t ~\iff~ \psi(\overline{b}^*_v) \in q_{t^\prime}$.
\item for each formula $\psi(x,\overline{x}_v)$, 
$\psi(x,\overline{b}^*_v) \in r_t ~\iff~ \psi(x,\overline{b}^*_v) \in r_{t^\prime}$. 
\end{enumerate}
\item if $\beta \in w_t$ for some $t < t_*$ then
\[ \tp(b^*_\beta, \{ b^*_\gamma : \gamma \in \bigcup_{ s\leq t} w_s ~\mbox{and}~ \gamma < \beta \} ) ~\mbox{dnf over}~
\{ b^*_\gamma : \gamma \in w_t \cap \beta \}. \] 
\item for each $t < t_*$, if $w^\prime \subseteq w$ and $\clm(w^\prime) = w^\prime$ then 
$\mathfrak{C}_T \rstr \{ b^*_\alpha : \alpha \in w^\prime \} \preceq \mathfrak{C}_T$ and 
$r_t(x,\overline{b}^*_{w^\prime})$ is a complete type over this elementary submodel which does not fork over $M_*$ 
$($noting that the domain of $M_*$ is $\{ b^*_\alpha : \alpha \in \clm(\emptyset) \}$ by the first item$)$. 
\end{enumerate}
\end{defn}

Now we arrive at the key point, coloring $\mcr_{\xm}$ with few $(\mu < \lambda)$ colors to capture consistency. 

\begin{defn} \label{d:es4}
Let $\xm$ be a $(\lambda, \theta, \sigma)$-presentation and $\mcr = \mcr_{\xm}$ be from $\ref{d:es3}$. 
Call $G : \mcr_{\xm} \rightarrow \mu$ \emph{an intrinsic coloring of $\mcr_{\xm}$} if: 
whenever 
\[ \overline{\xr} = \langle \xr_t = (u_t, w_t, q_t, r_t) : t < t_* < \sigma \rangle \]
is a sequence of elements of $\mcr_{\xm}$ and 
$\overline{b}^* = \langle b^*_\alpha : \alpha \in \bigcup_{t<t_*} w_t \rangle$ is a good instantiation of $\overline{\xr}$, 

\emph{if} $G \rstr \{ \xr_t : t < t_* \}$ is constant, 

\emph{then} the set of formulas 
\[ \{ \vp_\alpha(x, b^*_\alpha) \colon ~\alpha \in u_t, ~\vp_\alpha \in r_t, ~ t < t_*  \}  \]
is a consistent partial type which does not fork over $M_*$. 
\end{defn}

\noindent Note that in \ref{d:es4}, we ask for $\langle b^*_\alpha : \alpha \in \bigcup_t w_t \rangle$ 
when we only aim for consistency of $\{ \vp_\alpha(x, b^*_\alpha) \colon ~\alpha \in u_t, ~\vp_\alpha \in r_t, ~ t < t_* \}$, 
however meeting the requirements of the larger type will affect the choice of $\bar{b}^*$ thus of $\bar{b}^*\rstr \bigcup_t u_t$.  


We have now defined all terms necessary for `explictly simple,' so the reader may wish to re-read Definition \ref{d:c3}.  
In the next section, we will use this definition to characterize simplicity. 
\footnote{
Our proof in the next section will also work for the following slightly different definition, by \ref{p:e}.2, 
which we include for interest. 
Note it entails that 
any reasonable enumeration may be extended to a presentation, making explicit what is proved in \ref{p:e}, but 
does not say that every presentation may be refined to one which works. 
\begin{defn}
Assume $(\lambda, \mu, \theta, \sigma)$ are suitable. We might alternatively have said that 
$T$ is $(\lambda, \mu, \theta, \sigma)$-{explicitly simple} if $T$ is simple and whenever we are given:  
\begin{enumerate}
\item[(i)] $N \models T$, $||N|| = \lambda$, $p \in \ts(N)$ nonalgebraic,
\item[(ii)] an enumeration $\langle \vp_\alpha(x,a^*_\alpha) : \alpha < \lambda \rangle$ of $p$, 
where each $a^*_\alpha$ is a singleton, possibly imaginary, $\{ a^*_\alpha : \alpha < \lambda \} = \dom(N)$, and 
\item[(iii)] for some cardinal $\delta_{\xm} \leq |T|$, 
 $\{ a^*_\alpha : \alpha < \delta_{\xm} \}$ is the domain of an elementary submodel of $N$ over which $p$ does not fork, and 
$\{ \vp_\alpha(x, a^*_\alpha) : \alpha < \delta_{\xm} \}$ is a complete type over this submodel 
\item[(iv)] $\zs$ is an algebra on $\lambda$ with $<\sigma$ functions, with $\{ \alpha : \alpha < \delta_{\xm} \}$ closed under $\zs$, 
\end{enumerate} 
there exist
\begin{enumerate}
\item[$(a)$] an algebra $\zm \supseteq \zs$ of functions on $\lambda$ such that $\clm(\emptyset) = \cls(\emptyset)$ and 
$(\langle \vp_\alpha : \alpha < \lambda \rangle, \zm)$ form a 
 $(\lambda, \theta, \sigma)$-presentation $\xm$ of $p$ $($\emph{thus} the set of type fragments 
$\mcr_{\xm}$ associated to $\xm$ is well defined$)$ 
\item[$(b)$] and a function $G : \mcr_{\xm} \rightarrow \mu$ 
\end{enumerate}
such that $G$ is an intrinsic coloring of $\mcr_\xm$. 
\end{defn}} 

\begin{disc}[A classification program]  \label{d:classification}
\emph{In the case of the random graph, any function $G$ will work. 
In the case of an arbitrary simple theory, it will be shown in the next section 
that a suitable algebra and coloring can always be found assuming $\mu^+ = \lambda$. This outlines a program of stratifying the simple theories by 
determining which intermediate classes exist: that is, determine model-theoretic 
conditions which will characterize explicit simplicity for arbitrary $\mu$, or just e.g. $\mu = \aleph_0$ or $\mu^+ < \lambda$. 
This work begins in \cite{MiSh:1050}.} 
\end{disc}

\newpage
\setcounter{theoremcounter}{0} \section{Proof that simple theories are explicitly simple} \label{s:s-e-s} 
\setcounter{equation}{0} 

In this section we will prove that \emph{all} simple theories $T$ with $|T| < \sigma$ are 
explicitly simple for suitable $(\lambda, \mu, \theta, \sigma)$, when $\lambda = \mu^+$, 
by judicious use of Skolem functions, $\kappa(T)$, and the independence theorem. 
Recall that the $a^*_\alpha$ may be imaginaries.  
In cases where we have more information about the theory $T$, it is to be expected that more direct arguments 
may be given. Is there a simple theory for which $\mu^+ = \lambda$ is necessary? 
See Section \ref{s:questions}.

\begin{cont} \label{7:cont} In this section we assume: 
\begin{enumerate}
\item $T$ is a simple theory with infinite models, $T = T^{eq}$, in the signature $\tau$. 
\item $|T| < \sigma \leq \theta$.
\item $\mu^+ = \lambda$.
\item $N \models T$, $|N| = \lambda$. 
\item $p \in \ts(N)$ is nonalgebraic, and $p=p(x)$.  
\end{enumerate}
\end{cont}

The main ingredient in the proof that every simple theory is explicitly simple is the next Lemma \ref{p:e}.  
By essentially the same proof, we will show that under our present hypotheses: presentations exist, 
moreover any reasonable enumeration of 
a type given with some basic algebra may be extended to a presentation, and moreover that any presentation may be extended to 
one whose set of type fragments admits an intrinsic coloring. 

\begin{lemma} \label{p:e}
Let $(\lambda, \mu, \theta, \sigma)$, $T$, $N$, $p$ be as in $\ref{7:cont}$, so $\mu^+ = \lambda$ and $T$ is simple. 
\begin{enumerate}
\item Whenever $\bar{\vp} = \langle \vp_\alpha(x,a^*_\alpha) : \alpha < \lambda \rangle$ is an enumeration of $p$ satisfying 
\begin{enumerate}
\item[(a)] each $a^*_\alpha$ is a singleton, possibly imaginary; 
\item[(b)] $\{ a^*_\alpha : \alpha < \lambda \} = \dom(N)$; 
\item[(c)] for some cardinal $\delta \leq |T|$, 
 $\{ a^*_\alpha : \alpha < \delta \}$ is the domain of an elementary submodel $M_*$ of $N$ over which $p$ does not fork, and 
$\{ \vp_\alpha(x, a^*_\alpha) : \alpha < \delta \}$ is a complete type over this submodel. 
\end{enumerate} 
there exist an algebra $\zm$ on $\lambda$ and a function $G$ such that 
$\xm = ( \bar{\vp}, \zm)$ is a presentation, and $G: \mcr_{\xm} \rightarrow \mu$ 
is an intrinsic coloring. 
\item Suppose that in addition to $\bar{\vp}$ from (1) we are given an algebra $\zs$ on $\lambda$ with $<\theta$ functions, 
such that $\{ \alpha : \alpha < \delta \}$ is closed under the functions of $\zs$ and $u \in [\lambda]^{<\sigma}$ implies 
$\cls(u) \in [\lambda]^{<\sigma}$. 
Then there exist an algebra $\zm \supseteq \zs$ on $\lambda$ and a function $G$ such that 
$\xm = ( \bar{\vp}, \zm)$ is a presentation, $\clm(\emptyset) = \delta$, and $G: \mcr_{\xm} \rightarrow \mu$ 
is an intrinsic coloring.
\item Suppose that in addition to $\bar{\vp}$ from (1) we are given an algebra $\zs$ on $\lambda$ such that 
$(\bar{\vp}, \zs)$ is a $(\lambda, \theta, \sigma)$-presentation, and $\delta = \cls(\emptyset)$. 
Then there exist an algebra $\zm \supseteq \zs$ on $\lambda$ and a function $G$ such that 
$\xm = ( \bar{\vp}, \zm)$ is a presentation, $\clm(\emptyset) = \cls(\emptyset) = \delta$, and $G: \mcr_{\xm} \rightarrow \mu$ 
is an intrinsic coloring.
\end{enumerate}
\end{lemma}

\begin{proof}
It will suffice to prove (2). 
Since $T$ is simple, recall that $\kappa(T)$ exists and $\kappa(T) \leq |T|^+$, Fact \ref{kappa-of-T}. 
For clarity, rename $\delta$ as $\delta_\xm$ for the entirety of this proof. 
(Our construction will ensure that $\clm(\emptyset)  = \delta = \delta_\xm$.)  

The construction will take several steps and include some intermediate 
definitions and claims. First, we build a presentation by specifying an algebra $\zm$.
Let 
\begin{equation}
 X_0 = \{ ~\delta~ < \lambda ~:~ \delta \geq \delta_{\xm} \mbox{ and } N\rstr_{\{ a^*_\alpha : \alpha < \delta \}} \preceq N \}.  
\end{equation}
As $\lambda$ is regular, $X_0$ is a club of $\lambda$.  
We will construct $\zm$ to satisfy the following additional properties:

\begin{enumerate}
\item $\clm(\emptyset) = \{ \alpha : \alpha < \clm(\emptyset) \} = \{ \alpha : \alpha < \delta_{\xm} \}$, so in particular $\clm(\emptyset)$ 
is a cardinal $\leq |T|$. 
\item if $u \in [\lambda]^{<\aleph_0}$ then $|\clm(u)| < \sigma$, and if $u \in [\lambda]^{<\sigma}$ then also $|\clm(u)| < \sigma$. 
\item for each $v \in [\lambda]^{<\theta}$,~ $\clm(v) \in Y_0$ where 
\begin{align*}
Y_0 = & \{ ~ w  \subseteq \lambda ~:~  |w| < \theta, \{ \alpha : \alpha < \clm(\emptyset) \} \subseteq w,  \\
                 & (\forall \delta \in X_0)\left( N \rstr_{\{ a^*_\alpha ~:~ \alpha \in w \cap \delta \}} \preceq N \right), \\
                 & (\alpha \in w) \implies \tp(a^*_\alpha, \{ a^*_\beta : \beta < \alpha \}, N) ~\mbox{dnf over}~ \{ a^*_\beta : \beta \in w \cap \alpha \}~\}. 
\end{align*}
\item if $w \in [\lambda]^{<\theta}$ and $w = \cla(w)$ then, writing $Z_0 = \{ \alpha : \alpha < \clm(\emptyset) \} \cup X_0$:
\begin{enumerate}
\item[(i)] $\alpha \in w \iff \alpha+1 \in w$
\item[(ii)] $\alpha \in w \implies \min(Z_0 \setminus \alpha) \in w \land \sup(Z_0 \cap (\alpha + 1)) \in w$
\item[(iii)] $\alpha \in w \implies \alpha \cap w \subseteq \cla(\{\alpha \} \cup (w \cap \mu) )$ 
\item[(iv)] if $v \in [w]^{<\aleph_0}$ then $(\exists \alpha \in w)( \{ a^*_\beta : \beta \in v\} \subseteq a^*_\alpha)$. 
\end{enumerate}
\item $\zm$ contains functions $F_i : \lambda \times \lambda \rightarrow \lambda$, $i=1,2$ such that:  
\begin{enumerate}
\item[(i)] if $\alpha \geq \mu$ then $\langle F_1(\alpha, \beta) : \beta < \alpha \rangle$ lists $\mu$ without repetition; otherwise, it lists $|\alpha|$.
\item[(ii)] $\langle F_2(\alpha, \beta) : \beta < |\alpha| \leq \mu \rangle$ lists $\{ \gamma : \gamma < \alpha \}$ without repetition.
\item[(iii)] $\beta < \alpha \implies F_2(\alpha, F_1(\alpha, \beta)) = \beta$. 
\end{enumerate}

\item $\zm \supseteq \zs$.
\end{enumerate}

\noindent In our construction of such an $\zm$ we will do a little more than is necessary. 
The first collection of functions we add to the algebra are: 

\begin{enumerate}
\item[(A)] Let $\delta_{\xm} = \delta$ be as in (1)(c) of the Lemma. 
For each $\epsilon < \delta_{\xm}$, let $F_\epsilon : \lambda \rightarrow \lambda$ be the constant function equal to $\epsilon$. 
This will ensure that\footnote{The reader may prefer to also add a single constant, $0$.} 
$\{ \alpha : \alpha < \delta_{\xm} \} \subseteq \clm(\emptyset)$. 

\item[(B)] Recalling $\kappa(T) \leq |T|^+$, 
for each $\epsilon \leq |T|$ add functions 
$F^0_\epsilon : \lambda \rightarrow \lambda$ such that 
$F^0_\epsilon(\alpha) \leq \alpha$ and $tp(a^*_\alpha, A_\alpha)$ dnf over $\{ a^*_{F^0_\epsilon(\alpha)} : \epsilon < |T| \}$.   

\item[(C)] For each $\epsilon < |T|$ choose functions $F^1_\epsilon: \lambda \rightarrow \lambda$ such that 
$\langle F^1_\epsilon(\alpha)  : \epsilon < |T| \rangle$ enumerates the elements of $\acl^{eq}(a^*_\alpha)$, possibly with repetition.  

\item[(D)] We include all functions from $\zs$, which, without loss of generality, are 
denoted by symbols distinct from all other functions we add to $\zm$. 

\item[(E)] It will be useful to add a family of Skolem functions which are guaranteed to choose the minimum witness with respect to our given 
enumeration $\langle a^*_\alpha : \alpha < \lambda \rangle$. 
Fix some enumeration $\langle \vp_\epsilon : \epsilon < |T| \rangle$ of the $\ml$-formulas. 
For each $\epsilon < |T|$ choose functions $F^2_\epsilon: \lambda \rightarrow \lambda$ such that 
$F^2_\epsilon(\alpha)$ is the minimal $\beta$ such that $a^*_\beta$ is a witness to $\vp_\epsilon(x,a^*_\alpha)$, 
if this is well defined and nonempty, and $0$ otherwise. (Alternately, we could consider 
$F^2_\epsilon$ as a $k_\epsilon$-place function, but this is not necessary as we have coded finite sets in condition 4.(iv).)

\item[(F)] For each $\epsilon < |T|$, define $F^3_\epsilon : \lambda \rightarrow \lambda$ so that $F^3_\epsilon(\alpha) = \beta$ if
$\vp_\epsilon=\vp(x,y,a^*_\alpha)$ is an equivalence relation with finitely many classes and $\vp(x,a^*_\beta, a^*_\alpha) \in p$, and $F^3_\epsilon(\alpha) = 0$ 
otherwise.  

\item[(G)] We'd like to ensure that the type restricted to closed sets is complete.\footnote{Up to this point,  
if $\alpha \in w$, our restricted type will include  $\vp_\alpha(x,a^*_\alpha)$
 but won't necessarily decide the value of some other formula 
$\psi(x,a^*_\alpha)$ with the same parameter 
unless there is $\beta \in w$ such that $\vp_\beta = \psi$ or $\neg \psi$ and $a^*_\beta = a^*_\alpha$.} 
For each formula $\psi(x,\bar{y})$ in the signature $\tau$, let $G_\psi$ be a new function of arity $\ell(\bar{y})$ 
defined so that: $G_\psi(\alpha_{i_0}, \dots, \alpha_{i_{k-1}}) = \beta$ if 
$\beta < \lambda$ is the least ordinal such that $\vp_\beta(x,a^*_\beta)$ is equivalent mod $T$   
to either $\psi(x,a^*_{\alpha_{i_0}}, \dots, a^*_{\alpha_{i_{k-1}}})$ or its negation. Since $N$ is a model and $p$ is complete, 
this is well defined. 
\end{enumerate}

Next we add some families of functions which will help   
in partitioning our eventual $\mcr_{\xm}$ into equivalence classes which are sufficiently tree-like 
to allow inductive amalgamation, since we have made no particular assumptions 
about the theory (beyond simplicity) which would otherwise guarantee such a coherence of patterns. 
The functions in family (I) use $\lambda = \mu^+$ in an essential way. 

\br
\begin{itemize}
\item[(H)] $J_i: \lambda \rightarrow \lambda$, $i=1,2,3$ where
$J_1(\alpha) = \alpha+1$ and  $J_2(\alpha) = \operatorname{min}(Z_0 \setminus \alpha)$ and 
$J_3(\alpha) = \sup(Z_0 \cap (\alpha +1))$. 

\br
\item[(I)] $($recalling $\lambda=\mu^+$$)$ 
$F_i : \lambda \times \lambda$, $i=1,2$ such that:  
\begin{enumerate}
\item if $\alpha \geq \mu$ then $\langle F_1(\alpha, \beta) : \beta < \alpha \rangle$ lists $\mu$ without repetition; otherwise, let it list $|\alpha|$.
\item $\langle F_2(\alpha, \beta) : \beta < |\alpha| \leq \mu \rangle$ lists $\{ \gamma : \gamma < \alpha \}$ without repetition.
\item $\beta < \alpha \implies F_2(\alpha, F_1(\alpha, \beta)) = \beta$.   
\end{enumerate}
So for each $w \in [\lambda]^{<\theta}$, $\alpha \in w$ and $\beta \in \alpha \cap w$ implies
$F_1(\alpha, \beta) \in w\cap \mu$ and $\beta = F_2(\alpha, F_1(\alpha, \beta))$.
\end{itemize}
\br
\noindent Note that $\langle F_1(\alpha, \beta) : \beta < \alpha \rangle$ maps $\alpha$ into $\mu$, whereas 
$\langle F_2(\alpha, \beta) : \beta < \alpha \rangle$ maps a subset of $\mu$ to $\alpha$, so the third condition is natural. 
Note also that by choice of $J_1$, $J_2$ and unions of elementary chains we have that $w \in {Y_0}$ implies $\sup(w) \in X_0$. 

\br

Let $\zm$ be the algebra including the functions from (A) through (I) above.  
Let us check that the numbered conditions of the claim are satisfied:

\begin{enumerate}
\item[1.] We ensured with family (A) that 
$\{ \alpha : \alpha < \delta_{\xm} \} \subseteq \clm(\emptyset)$. 
Let us check that equality holds 
by examining the functions of $\zm$. 
The functions in family (B) are nonincreasing so will not change an initial segment. 
The functions in families (C) and (F) will map tuples from $\{ \alpha : \alpha < \delta_{\xm} \}$ back to this set as $M_* \preceq N$ and $T = T^{eq}$. 
As for family (D), we required that $\{ \alpha : \alpha < \delta_{\xm} \}$ be closed under the functions of $\zs$.  
The new Skolem functions in (E) were chosen to give the least witness 
in the ordering inherited from the enumeration, and $M_* \preceq N$. As for the functions of (G), 
we assumed to begin with that $p \rstr M_*$ is complete.    
Recalling the definition of $Z_0$ in item (2) of the claim, the functions $J_1$ and $J_2$ are equal on 
$\clm(\emptyset)$ and act by $\alpha \mapsto \alpha+1$. Since $\clm(\emptyset)$ is a limit ordinal, indeed a cardinal this poses no problem. 
Finally, as $|\clm(\emptyset)| < \sigma \leq \mu$, the functions $F_i$ on $\clm(\emptyset) \times \clm(\emptyset)$ 
will just list ordinals less than $\clm(\emptyset)$. 
This proves that $\clm(\emptyset) = \delta_{\xm}$. 

\item[2.] Immediate, as $u \in [\lambda]^{<\sigma}$ implies $\cls(u) \in [\lambda]^{<\sigma}$ and 
$|\zm \setminus \zs| \leq |T| < \sigma$. 

\item[3.] Given $v \in [\lambda]^{<\theta}$, we need to check that $w = \clm(v) \in Y_0$. 
We know $|w| < \theta$, and $\{ \alpha : \alpha < \clm(\emptyset) \} \subseteq w$ by the addition of 
constant functions in item (A).  The nonforking condition 
\[ (\alpha \in w) \implies \tp(a^*_\alpha, \{ a^*_\beta : \beta < \alpha \}, N) ~\mbox{dnf over}~ \{ a^*_\beta : \beta \in w \cap \alpha \} \]
is guaranteed by the functions in item (B). Finally, why should 
\[ \delta \in X_0 ~\implies ~ \left( N \rstr_{\{ a^*_\alpha ~:~ \alpha \in w \cap \delta \}} \preceq N \right) \mbox{ ? } \]
Suppose we don't get an elementary submodel, i.e. there is a formula with parameters in $\{ a^*_\alpha ~:~ \alpha \in w \cap \delta \}$ 
which has a solution in $N$ but not in this submodel. However, we chose the Skolem functions in (E) to select the 
minimum possible witness. Since $\delta \in X_0$, the minimal witness must be of index $\beta < \delta$, contradiction. 

\item[4.] Conditions (i)-(ii) are ensured by the functions in family (H).
For item (iii), recall from family (I) that 
for each $w \in [\lambda]^{<\theta}$, $\alpha \in w$ and $\beta \in \alpha \cap w$ implies
$F_1(\alpha, \beta) \in w\cap \mu$ and $\beta = F_2(\alpha, F_1(\alpha, \beta))$.
For (iv) remember that we assumed that $T$ eliminates imaginaries; in fact, in item (F) we have coded all finite sets. 

\item[5.] Ensured by the functions of family (I). 

\item[6.] We assumed $\zm \supseteq \zs$.  
\end{enumerate}

Let us check that $(\bar{\vp}, \zm)$ is a presentation, by checking the requirements of Definition 
\ref{d:pres}. \ref{d:pres}.1 was assumed in the claim. \ref{d:pres}.2-3 
were just verified in item 2. For \ref{d:pres}.4, the first line was verified in item 1. $M_*$ follows by the 
hypothesis $(iii)$ of the Lemma and the fact that $\clm(\emptyset) = \delta_{\xm}$.  
\ref{d:pres}.5 was ensured by the functions of families (E), (G) respectively. 
Finally, \ref{d:pres}.6 follows by the functions in (B). This shows that $\xm = (\bar{\vp}, \zm)$ is indeed a 
presentation. 
\begin{equation} 
\mbox{ Fix this presentation $\xm$ for the remainder of the proof. } 
\end{equation}
Let $\mcr_\xm$ be the associated set of quadruples given by \ref{d:es3}.  
We now work towards the definition of the coloring $G$. 
In particular, we look for underlying trees. This will require several definitions. 
Since we are assuming $|T| < \sigma \leq \theta$, we have $\theta \geq \aleph_1$ always. We could avoid referencing the order topology by assuming $\theta > \aleph_1$. 

\begin{defn}
For $\alpha$ an ordinal and $w \subseteq \alpha$, write $\overline{w}$ for its topological closure, i.e. closure with respect to the order topology. 
\end{defn}

\begin{defn} 
Recalling $Y_0$ from condition 1. on the algebra above\footnote{and the fact that the $a^*_\alpha$ may be imaginary, so code finite sets.}
define ${\Upsilon_1} \subseteq {Y_0}$ to be the set 
\[\{ w \in [\lambda]^{<\theta} ~:~ w = \cla(w) 
~\mbox{and if $v \in [w]^{<\aleph_0}$ then $(\exists \alpha \in w)( \{ a^*_\beta : \beta \in v\} \subseteq a^*_\alpha)$} \}. 
\]
\end{defn}

By regularity of $\lambda > \aleph_1$, ${\Upsilon_1}$ is cofinal in the natural partial order, and moreover is stationary. 
Since we constructed $\zm$ so that the closure of a set $\clm(u) \in Y_0$, 
and in addition we have coding of finite sets, $\clm(u) \in \Upsilon_1$ for each $u \in \Omega$.

\begin{defn} \label{d:eqr}
Define an equivalence relation $E$ on ${\Upsilon_1}$ by $w_1 ~E~ w_2$ {when}:
\begin{enumerate}
\item $\{ w_1, w_2 \} \subseteq {\Upsilon_1}$. 
\item $\otp(w_1) = \otp(w_2)$ and $\otp(\overline{w}_1) = \otp(\overline{w}_2)$.
\item $w_1 \cap \mu = w_2 \cap \mu$. 
\item \begin{enumerate}
\item if $h: w_1 \rightarrow w_2$ is order preserving onto then the map 
$a^*_\alpha \mapsto a^*_{h(\alpha)}$ is elementary. 
\item if $h: \overline{w}_1 \rightarrow \overline{w}_2$ is order preserving onto then the map 
$a^*_\alpha \mapsto a^*_{h(\alpha)}$ is elementary.
\end{enumerate}
\item If $\alpha \in \overline{w}_1 \setminus \mu$ and $ \beta \in \overline{w}_1 \cap \mu$ then 
$F_1(\alpha, \beta) = F_1(h(\alpha), h(\beta))$. 
\end{enumerate}
\end{defn}

\begin{obs} \label{o:counting} 
$E$ is an equivalence relation with $\mu$ classes. 
\end{obs}

\begin{proof} 
There are $\leq \theta \leq \mu$ choices of order type in both clauses of (2). For (3), recall by \ref{d:suitable}(b) that 
$\mu^{<\theta} = \mu$. 
Then for (4), let $\rho = |w|$ or $=|\overline{w}|$, so $\rho < \theta$ since $\theta$ is regular. Let $v$ be $\otp(w)$  
and fix an order preserving bijection $\pi: v \rightarrow w$. 
As an upper bound, we count the $v$-indexed sequences of types $\langle p_i : i \in v \rangle$ where $p_i = p_i(x_{\pi(i)}, \{ a^*_j : j \in w \cap \pi(i) \} )$. 
For each $i \in v$ (without loss of generality $i \geq |T|$), suitability \ref{d:suitable}(c) implies that $i < \theta \implies 2^{|i|} < \mu$, 
so there are $<\mu$ choices for $p_i$ thus no more than $\theta \cdot \mu \leq \mu$ such sequences. 
The case where $v = \otp(\overline{w})$ is analogous. 
Finally, for (5): if $\alpha \in \overline{w}_1 \setminus \mu$ then $\langle F_1(\alpha, \beta) : \beta < \alpha \rangle$ lists $\mu$ without 
repetition. So for each $\alpha \in \overline{w}_1$ (of which there are $<\theta$) and each $\beta \in \overline{w}_1 \cap \mu$ (of which there are $<\theta$)
we need to know the value of $F_1(\alpha, \beta) \in \mu$. Since $\mu^{<\theta} = \mu$, (5) requires no more than $\mu$ classes. 
This completes the proof. 
\end{proof}

\begin{claim}[``Treeness''] \label{c:treeness} 
Whenever $wEv$ $($so $w, v$ are $\clm$-closed$)$, we have: 
\begin{enumerate}
\item $w \cap v$ is $\clm$-closed. 
\item $($``treeness''$)$ $w \cap v \tlf w$ (and $w \cap v \tlf v$).
\end{enumerate}
\end{claim}

\begin{proof} 
By definition of $E$, $w \cap \mu = v \cap \mu$.
Suppose we have $\mu \leq \beta < \alpha$ with $\alpha \in w \cap v$ and $\beta \in w$. Recalling the functions from family (I), 
$\beta^\prime := F_1(\alpha, \beta) \in w \cap \mu$ and so $\beta^\prime \in v$ since $w$ and $v$ agree on $\mu$. 
As $v = \clm(v)$, $F_2(\alpha, \beta^\prime) \in v$ and $F_2(\alpha, \beta^\prime) = F_2(\alpha, F_1(\alpha, \beta)) = \beta$ by definition of $F_1, F_2$. 
This shows (2.): $w \cap v \tlf w$, and $\tlf v$. Finally, closure (1.) is immediate because the algebra consists 
of functions and so 
the nonempty intersection of two $\clm$-closed sets will be closed under these functions. 
\end{proof}

We will refer to \ref{c:treeness}(2) as ``by treeness'' in the rest of the proof. 
We now define the coloring $G: \mcr_{\xm} \rightarrow \mu$. 

\begin{claim} \label{x:c17}   
There is $G: \mcr_{\xm} \rightarrow \mu$ so that $G(u,w,q,r) = G(u^\prime, w^\prime, q^\prime, r^\prime)$ implies:
\begin{enumerate}
\item[(i)] 		$w/E = w^\prime/E$, where $E$ is the equivalence relation from \ref{d:eqr}. 
\item[(ii)] 	$\beta \in w \cap w^\prime \implies \otp(\beta \cap w) = \otp(\beta \cap w^\prime)$. 
\item[(iii)]	if $h: w \rightarrow {w^\prime}$ is order preserving onto then $h$ maps $u$ to $u^\prime$. 
\item[(iv)] 	if $h: w \rightarrow {w^\prime}$ is order preserving onto then $h$ maps $q$ to $q^\prime$ and  
$r$ to $r^\prime$ in the obvious way. 
\item[(v)] 		$\rn(G) \subseteq \mu$.
\end{enumerate}
\end{claim}

\begin{proof}
For (i), recall that $E$ has $\mu$ equivalence classes.  Now (ii) will follow by the ``treeness'' condition \ref{c:treeness}. 
Recall $\mu = \mu^{<\theta}$ by \ref{d:suitable}(b). 
So for (iii), there are indeed no more than $\mu$ ways to choose a sequence of ordinal length $\alpha < \sigma \leq \theta$ 
from a sequence of length $<\theta$. For (iv), as in \ref{o:counting}, there are likewise (relatively) few equivalence classes of types recalling 
$\alpha < \theta \implies 2^{|\alpha|} \leq \mu$ by \ref{d:suitable}(c). 
(Note that $q$ need not be the type of $\{ a^*_\alpha : \alpha \in w \}$.) 
So each condition requires no more than $\mu$ classes, and then (v) follows. 
\end{proof}

For the rest of the proof, fix $G$ satisfying \ref{x:c17}. 
Fix $\bar{\xr} = \langle \xr_t = (u_t, w_t, p_t, r_t) : t < t_* < \sigma \rangle$ from $\mcr_\xm$. 
Let $w = \bigcup_t w_t$. Suppose that $G \rstr \bar{\xr}$ is constant and that $\bar{\mb}^*_w$ is a good instantiation of 
$\bar{\xr}$ in the sense of $\ref{d:good-inst}$. 
Recall that this definition ensures that $\{ b^*_\alpha : \alpha < \clm(\emptyset) \} = \dom(M_*)$, and that if $\beta \in w_t$ then 
\begin{equation}
\label{e-nfk}
\mbox{$\tp(b^*_\gamma, \{ a_\beta: \beta \in \bigcup_\xi w_\xi \cap \gamma \} ) $ dnf over $\{ b^*_\beta : \beta \in w_t \cap \gamma \}$. } 
\end{equation}

Prior to the main amalgamation, let us record a case of good behavior. 

\begin{claim} \label{c:obs}
Suppose $\gamma \in \bigcap_{t<t_*} u_t$. Then the set\[
\{ \vp(x,\overline{b}^*_v) : ~\mbox{there is $t < t_*$ s.t.}~ v \in [w_t \cap \beta]^{<\aleph_0}
~\mbox{and}~ \vp(x,\overline{x}_v) \in r_t(x,\overline{x}_{w_t}) \} \]
is a partial type which does not fork over $M_*$.
\end{claim}

\begin{proof} 
This is simply because the treeness condition \ref{c:treeness} along with conditions (i)-(iv) from the definition of $G$ in \ref{x:c17} guarantee 
that for any $t < t^\prime < t_*$, there is an order preserving map $h: w_t \rightarrow w_{t^\prime}$, 
and this map is constant on the common initial
segment of the $w_t$'s. Thus for some, equivalently every, $t < t_*$ the type in the statement of the claim is simply the partial type
\[\{ \vp(x,\overline{b}^*_v) : ~v \in [w_t \cap \beta]^{<\aleph_0} ~\mbox{and}~ \vp(x,\overline{x}_v) \in r_t(x,\overline{x}_{w_t}) \} \] 
which is consistent and does not fork over $M_*$ by definition of $\mcr$. 
\end{proof} 

\br
\noindent It remains to show that the set of formulas  
\begin{equation} \label{e:first}
   \{ \vp_\alpha(x,b^*_\alpha) : t < t_*, \alpha \in u_t \} 
\end{equation}
is consistent.
Let $u = \bigcup_t u_t$ and recall $w = \bigcup_t w_t$, so $\{ \alpha : \alpha < {\clm(\emptyset)} \} \subseteq w$.
For each $\gamma \leq \lambda$, define:
\begin{equation} \label{e:second}
 r^*_\gamma = \{ \vp(x,\overline{b}^*_v) : ~\mbox{there is $t < t_*$ s.t.}~ v \in [w_t \cap \gamma]^{<\aleph_0} 
~\mbox{and}~ \vp(x,\overline{x}_v) \in r_t(x,\overline{x}_{w_t}) \} 
\end{equation} 
We prove by induction on $\gamma \leq \lambda$ that $r^*_\gamma$ is a consistent partial type which does not fork 
over $\{ a^*_\alpha : \alpha < \clm(\emptyset) \} = \{ b^*_\alpha : \alpha < \clm(\emptyset) \}$. 
Clearly this will imply that equation (\ref{e:first}) is consistent.\footnote{We are actually proving something stronger than explicit simplicity 
as we will have consistency over all the $w$'s, not just the $u$'s.}

\br
\underline{$\gamma = 0$}: Trivial.

\underline{$\gamma$ limit}: Consistency is by compactness, and nonforking is by the finite character of nonforking in simple theories.

\underline{$\gamma = \beta+1$}: If $\beta \notin \bigcup_t w_t$, there is nothing to show. 

Suppose then that $\beta \in w_t$ for at least one $t$. We can write $t_*$ as the disjoint union of two sets
\[ Z_0 := \{ t < t_* : \beta \notin w_t \},  Z_1 := \{ t < t_* : \beta \in w_t \} \]
where, by assumption, $Z_1$ is nonempty. For $i \in \{ 0,1\}$ we define:
\[ r^i_\gamma := \{ \vp(x,\overline{b}^*_v) : ~\mbox{there is $t \in Z_i$ s.t.}~ v \in [w_t \cap \gamma]^{<\aleph_0} 
~\mbox{and}~ \vp(x,\overline{x}_v) \in r_t(x,\overline{x}_{w_t}) \} \] 
Now both $r^0_\gamma$, $r^1_\gamma$ are consistent partial types which moreover do not fork over $M_*$: 
the case of $r^0_\gamma$ is by inductive hypothesis (it is contained in $r^*_\beta$ by definition of $Z_0$), 
and the case of $r^1_\gamma$ is \ref{c:obs}. If $Z_0$ is empty, there is nothing to amalgamate, so we are finished. 
If not, define $W \subseteq \beta$ to be:  
\[ W:= W_0 \cap W_1~ \mbox{ where } W_0 := \bigcup \{ w_t : t \in Z_0 \}\cap \gamma, ~ W_1 := \bigcup \{ w_t : t \in Z_1 \} \cap \gamma .\]
Let us show that 
\begin{align}
\label{int-model}
\mathfrak{C} \rstr \{ b^*_\gamma : \gamma \in W \} \preceq \mathfrak{C}.
\end{align}
By \ref{d:good-inst}.5, to prove equation (\ref{int-model}) it will suffice to prove 
\begin{align}
\label{int-closed}
W = \clm(W). 
\end{align}
Suppose equation (\ref{int-closed}) did not hold. 
Then for some finite tuple of elements $\beta_0, \dots, \beta_n \in W$ (of an appropriate length) 
and one of the functions of the algebra, call it $X$, 
$X(\beta_0, \dots, \beta_n) = \beta_* \notin W$. As $t<t_* \implies w_t = \clm(w_t)$,  
$\beta_* \in w_t$ for each $t < t_*$. 
Now for any $w_s$, $w_t$ with $s \in Z_0 \neq \emptyset$, $t \in Z_1 \neq \emptyset$,
we have that $\beta_* \in w_s \cap w_t$ while by assumption, $\beta \in w_t$ and $\beta \notin w_s \cap w_t$. If $\beta_* > \beta$, this contradicts treeness, \ref{c:treeness}. 
However, if $\beta_* \leq \beta$ and $\beta_* \in \bigcup_{t < t_*} w_t$ then necessarily $\beta_* \in W$, also a contradiction.  
This completes the verification of equation (\ref{int-closed}) and so also of equation (\ref{int-model}).  

We now check nonforking. 
By \ref{c:treeness}, $W_1 = w_t \cap \gamma$ for some, equivalently every, $t \in Z_1 \neq \emptyset$. So by 
equation (\ref{e-nfk}),
$B_1 = \{ b^*_\alpha :  \alpha \in W_1 \}$ is independent from $B_0 = \{ b^*_\alpha : \alpha \in W_0 \}$ over $B = \{ b^*_\alpha : \alpha \in W \}$.  
$B$ is the domain of an elementary submodel of $\mathfrak{C}$ by equation (\ref{int-model}). 

Let $\delta = \sup(W)$. From equations (\ref{int-closed}) and (\ref{int-model}) and 
Definition \ref{d:good-inst}(5), $p_x := r^*_\delta \rstr B$ is a complete type over 
a model, which does not fork over $M_*$. 
By inductive hypothesis, $r^0_\gamma$ and $r^1_\gamma$ are consistent extensions of $p_x$ to $B_0$, $B_1$ respectively, which do not fork over $M_*$. 
If necessary, we can complete $r^0_\gamma, r^1_\gamma$.
Apply the independence theorem (Theorem \ref{t:ind-thm} p. \pageref{t:ind-thm}) to complete the induction. 

This proves that $\xm$, $G$ satisfy \ref{d:es4} and completes the proof of Lemma \ref{p:e}.
\end{proof}

We arrive at the first main theorem of the paper:

\begin{theorem}[Simple is explicitly simple] \label{c:s-es}
Suppose that $T$ is a complete, simple first order theory. Suppose that $(\lambda, \mu, \theta, \sigma)$ 
are suitable cardinals in the sense of $\ref{d:suitable}$ and in addition: 
\begin{enumerate}
\item[(a)] $|T| < \sigma$ 
\item[(b)] $\mu^+ = \lambda$.  
\end{enumerate} 
Then $T$ is simple if and only if $T$ is $(\lambda, \mu, \theta, \sigma)$-explicitly simple.
\end{theorem}

\begin{proof}
By Lemma \ref{p:e} and Claim \ref{c:look}. 
\end{proof}

What did we use in this characterization? 

First, note that to apply the independence theorem in the inductive step in Lemma \ref{p:e} 
we needed that $\mathfrak{C} \rstr \{ b_\gamma : \gamma \in X \} \preceq \mathfrak{C}$;
the addition of $\cla$ ensured that this would happen. 
More precisely, we used that
given $\{ \xr_t : t < t_* < \sigma \}$ and $\beta + 1 = \gamma < \lambda$,  
\begin{enumerate}
\item for each $\xr_t = (u_t, w_t, q_t, r_t)$, $w_t = \clm(w_t)$. 
\item for each $\beta < \lambda$, really $\beta \in \bigcup_t w_t$, 
$\mathfrak{C} \rstr \{ b_\gamma : \gamma \in W \} \preceq \mathfrak{C}$ where $W$ is the intersection of 
$W_0 := \bigcup \{ w_t : t < t_*, \beta \notin w_t \}\cap \gamma, ~ W_1 := \bigcup \{ w_t : t < t_*, \beta \in w_t \} \cap \gamma$. 
\end{enumerate}
For some theories no problem would arise, e.g. when every type is stationary, so no algebra would be necessary. 

Second, it is natural to ask about Theorem \ref{c:s-es} for arbitrary simple $T$ in the case $\mu = \lambda$. 
If $|T| < \sigma$, one can satisfy the definition of presentation by simply choosing the algebra to contain Skolem functions 
and functions to cover nonforking (a small part of what was done in the argument above).  
Provided the cardinals $\lambda, \mu, \theta, \sigma$ are such that 
the corresponding $\mcr_{\xm}$ has cardinality $\lambda$, the existence of the required $G$ with range 
$\lambda = \mu$ may then be trivially satisfied by assigning each element of $\mcr$ its own color.   

The choices for how to present the parameters $\langle a^*_\alpha : \alpha < \lambda \rangle$ have different advantages. 
Listing imaginaries, as we currently do, makes the presentation of formulas much more compact but introduces a lot of 
redundancy: each singleton $c \in \dom(N)$ appears in cofinally many tuples. This doesn't interfere with the nonforking condition in 
\ref{d:es4}, as we assumed $\mu^+ = \lambda$. To show $(\lambda, \mu, \theta, \sigma)$-explicitly simple for $\mu^+ < \lambda$, however, 
would require addressing this, perhaps by listing the domain of $N$ in terms of actual singletons. 

\begin{disc}
We have stated Theorem $\ref{c:s-es}$ as ``$T$ simple iff $T$ explictly simple,'' not ``$T$ simple iff $T^{eq}$ explicitly simple,'' 
reflecting our claim that using imaginary elements is purely presentational. To see this, notice that throughout the entire proof of this theorem 
we could have considered each $a^*_\alpha$ (or $b^*_\alpha$) as a finite set with respect to some fixed background enumeration of the 
singletons of $\dom(N)$, and considered the corresponding variable $x_\alpha$ as the corresponding finite sequence of variables. 
In this setup the only essential addition would be an additional coding function, i.e. an enumeration, translating between the finite set coded by index $\alpha$ 
and vice versa, so as to be able to define and apply the functions of the algebra $\zm$. 
Such a coding function could alternately be subsumed into the basic algebra $\zs$, much as we subsume Skolem functions in $\ref{c:lifting}$. 
\end{disc}

Finally, as this discussion reflects, Theorem \ref{c:s-es} raises the question of 
the nature of an intrinsic coloring in certain examples, since we have given only a proof 
of its existence. One central example, the generic $k$-ary hypergraphs omitting complete sub-hypergraphs on $n$ vertices \cite{hrushovski1},   
will be developed in our paper \cite{MiSh:1050}. 

\br
\br

\newpage
\setcounter{theoremcounter}{0} \section{Existence of optimal ultrafilters} \label{s:optimal}
\setcounter{equation}{0} 

\begin{conv} \label{c:exist} In this section we assume:

\begin{itemize}
\item $\lambda, \mu, \theta, \sigma$ are suitable cardinals. 
\item $\ba = \ba^1_{2^\lambda, \mu, \theta}$. 
\item $\sigma$ is uncountable and supercompact, see $\ref{f:sct}$. 
\item All ultrafilters $\de_*$ on $\ba$ mentioned in this section are $\sigma$-complete; we may repeat this for emphasis. 
\item When $\vv \subseteq 2^\lambda$, we write $\ba \rstr \vv$ to mean the subalgebra generated by $\{ \mx_f : f \in \fin_{\mu,\theta}(\vv) \}$. 
\end{itemize}
\end{conv} 

\noindent Readers familiar with normal ultrafilters may wish to skip \ref{f:sct}, \ref{f:sct2} and \ref{obs:elem}, starting again with \ref{d:cont}.  

\begin{defn} \emph{(see e.g. Jech p. 137)} \label{f:sct}
\begin{enumerate}
\item Call the uncountable cardinal $\sigma$ \emph{supercompact} if for any $A$, $|A| \geq \sigma$ there exists an ultrafilter $\ee$ on $I = [A]^{<\sigma}$
which is:
\begin{itemize}
\item[(a)] $\sigma$-complete. 
\item[(b)] fine, meaning that in addition for any $a \in A$, $\{ X \in I : a \in X \} \in \ee$.
\item[(c)] normal, meaning that in addition $\ee$ is closed under diagonal intersections: if $\{ X_a : a \in A \} \subseteq \ee$
then $\{ X \in I :  X \in \bigcap_{a \in X} X_a \} \in \ee$.
\end{itemize}
\item If in 1. $A = \sigma$ and $\ee$ is normal on $I = [\sigma]^{<\sigma}$, the set $\{ X \in I : X $ is an ordinal $<\sigma \} \in \ee$, 
so we may say ``$\ee$ is an ultrafilter on $\sigma$.'' 
\end{enumerate}
\end{defn}

\begin{fact} \label{f:sct2}
If $\ee$ is a normal ultrafilter on $\kappa$, then $\ee$ contains every closed unbounded subset of $\kappa$. Moreover any $f: \kappa \rightarrow \kappa$ 
which is regressive on a set in $\ee$ must be constant on a set in $\ee$. 
\end{fact}

\begin{obs} \label{obs:elem}
Let $\chi$ be large enough. Fix $\mathfrak{A} = (\mch(\chi), \epsilon)$ 
and let $A = \mch(\chi)$. 
Let $J$ be the set $[A]^{<\sigma}$ and let $\ee$ be a normal ultrafilter on $J$. Then the set 
$I := \{ X \in J : X \mbox{ is an elementary submodel of }\mathfrak{A} \} \in \ee$. 
\end{obs}

\begin{proof}
 Let $\langle a_\alpha : \alpha < \kappa \rangle$ 
be an enumeration of $A$ and expand $\mathfrak{A}$ to $\mathfrak{A}^*$ by Skolem functions which choose the least witness according to this 
enumeration.  $Th(\mathfrak{A}^*)$ eliminates imaginaries and has Skolem functions.
For each of the countably many Skolem 
functions $g$, define $f_g : A \rightarrow A$ by: $f_g(X)$ is the least $\beta$ 
such that $b_\beta \in X$ but $g(b_\beta) \notin X$, if this exists, and otherwise $f_g(X) = X$. 
If $f_g(X) \neq X$ on an $\ee$-large set, then it is regressive on an $\ee$-large set, and so by normality constant and equal to 
some $b \in A$ on an $\ee$-large set, contradicting the fact that $\ee$ is fine. So it must be that 
$f_g(X) = X$ on an $\ee$-large set $Y_g$ for each of the countably many Skolem functions $g$, and 
since $\ee$ is $\sigma$-complete, $\bigcap_g Y_g \in \ee$ is the desired set of elementary submodels. 
\end{proof}

\begin{defn}[Continuous sequence] \label{d:cont}
Let $\bar{\mb} = \langle \mb_u : u \in [\lambda]^{<\sigma} \rangle$ be a sequence of elements of $\ba^+$. We call $\bar{\mb}$ 
\emph{continuous} when it is monotonic, meaning $u \subseteq v$ implies $\mb_u \geq \mb_v$, and in addition for all infinite $u \in [\lambda]^{<\sigma}$, 
\[ \mb_u = \bigcap \{ \mb_v : v \subseteq u, |v| < \aleph_0 \}. \]
\end{defn}

\noindent So if $\sigma = \aleph_0$ then ``continuous'' is just ``monotonic.''

\begin{defn}[Support of a sequence] \label{d:support}
Let $\overline{\mb} = \langle \mb_u : u \in [\lambda]^{<\sigma} \rangle$ be a sequence of elements of $\ba = \ba^1_{2^\lambda, \mu, \theta}$. 
\begin{enumerate}
\item We say $X$ is a support of $\overline{\mb}$ in $\ba$ when $X \subseteq \{ \mx_f : f \in \fin_{\mu}(\alpha) \}$ and 
for each $u \in [\lambda]^{<\aleph_0}$ there is a maximal antichain of $\ba$
consisting of elements of $X$ all of which are either $\leq \mb_u$ or $\leq 1 -\mb_u$. 
Though there is no canonical choice of support we will write $\supp(\bar{\mb})$ to mean \emph{some} support. 
\item When a support $\supp(\bar{\mb})$ is given, 
write 
\[ \ba^+_{\supp(\bar{\mb}), \mu} \mbox{ to mean } \ba^+_{\alpha_*, \mu} \] 
where $\alpha_* < 2^\lambda$ is minimal such that $\bigcup \{ \dom(f) : \mx_f \in \supp(\overline{\mb}) \} \subseteq \alpha_*$. 
\item When $\vv \subseteq 2^\lambda$, we say ``$\bar{\mb}$ is supported by $\ba \rstr \vv$'' to mean that 
there is a support for $\bar{\mb}$ contained in $\ba \rstr \vv$, recalling the notation from \ref{c:exist}. 
\end{enumerate}
\end{defn}

\noindent We emphasize that the support need not be unique. 
In the next definition, $\lambda, \mu, \theta$ come from the Boolean algebra $\ba$ and $\sigma$ comes 
from the sequence $\bar{\mb}$. 

\begin{defn}[Key Property] \label{kp}
Let $(\lambda, \mu, \theta, \sigma)$ be suitable, and $\ba = \ba^1_{2^\lambda, \mu, \theta}$. 
Let $\bar{\mb} = \langle \mb_u : u \in [\lambda]^{<\sigma} \rangle$ be a continuous sequence of elements of $\ba^+$. 
We say $\bar{\mb}$ has the $(\lambda, \mu, \theta, \sigma)$-Key Property when there exist
\begin{itemize}
\item[(a)] $\vv \subseteq 2^\lambda$, $|\vv| \leq \lambda$, such that a support of $\bar{\mb}$ is contained in $\ba \rstr \vv$  
\item[(b)] a closed unbounded $\Omega_* \subseteq [\lambda]^{<\sigma}$  
\end{itemize}
such that for every $\alpha < 2^\lambda$ with $\vv \subseteq \alpha$, there is a sequence 
\[ \bar{\mb}^\prime = \bar{\mb}^\prime({\alpha}) = \langle \mb^\prime_{\alpha, \{i\}} : i < \lambda \rangle \] 
of elements of $\ba^+$ which generates a multiplicative refinement 
$\langle \mb^\prime_{\alpha, u} : u \in [\lambda]^{<\sigma}\rangle$ of $\bar{\mb}$ 
such that for each $f \in \fin_{\mu,\theta}(\alpha)$, 
and each $u \in \Omega_*$, if $\mx_f \leq \mb_u$ then we may extend $f \subseteq f^\prime \in \fin_{\mu,\theta}(2^\lambda)$ so that 
$\mx_{f^\prime} \leq \mb^\prime_u$.  
\end{defn}

\begin{defn} \label{d:optimal} 
Assume $\lambda, \mu, \theta, \sigma$ are suitable. 
$\de_*$ is $(\lambda, \mu, \theta, \sigma)$-optimal if:
\begin{itemize}
\item $\de_*$ is a $\sigma$-complete ultrafilter on $\ba = \ba^1_{2^\lambda, \mu, \theta}$, and 
\item whenever $\bar{\mb} = \langle \mb_u : u \in [\lambda]^{<\sigma} \rangle$ is a continuous sequence of elements of $\de_*$ 
with the $(\lambda, \mu, \theta, \sigma)$-Key Property 
there is a multiplicative sequence $\bar{\mb}^\prime = \langle \mb^\prime_u : u \in [\lambda]^{<\sigma} \rangle$ of elements of $\de_*$ 
which refines $\bar{\mb}$.  
\end{itemize}
\end{defn}

\begin{theorem} \label{t:optimal} 
Suppose $(\lambda, \mu, \theta, \sigma)$ are suitable, $\sigma > \aleph_0$ is 
supercompact, and $\ba = \ba^1_{2^\lambda, \mu, \theta}$.
\begin{enumerate}
\item There exists a $(\lambda, \mu, \theta, \sigma)$-optimal ultrafilter $\de_*$ on $\ba$. 
\item Let $\de^*_0$ be a $\sigma$-complete filter on $\ba$ generated by $< 2^\lambda$ sets, or\footnote{Note that if 
$\mu < \lambda$ then $\mu^+ < 2^\lambda$, a case we will use.} just generated 
by a set supported by $\ba \rstr \vv$, $\vv \subseteq 2^\lambda$, $|\vv| < 2^\lambda$. 
Then there exists a $(\lambda, \mu, \theta, \sigma)$-optimal ultrafilter $\de_*$ on $\ba$ which extends $\de^*_0$. 
\end{enumerate}
\end{theorem}

\begin{proof}
Clearly it suffices to prove the second. 
Let $\de^*_0$ be given. Fix a set $X_* \subseteq \de^*_0$ of generators for this filter, with $|X_*| < 2^\lambda$. 
Choose $\vv_* \subseteq 2^\lambda$, $|\vv_*| < 2^\lambda$ which contains a support for $X_*$. 
Without loss of generality, $\vv_* = \alpha_* < 2^\lambda$ (i.e. if necessary, use a permutation of $2^\lambda$ 
mapping the relevant support into an ordinal $< 2^\lambda$). We will use $X_*$, $\alpha_*$ in our inductive construction. 
If no $\de^*_0$ is given, let $X_* = \emptyset$ and let $\alpha_* = 0$. 

There are at most $2^\lambda = |{^\lambda \ba}|$ sequences $\bar{\mb}$ with the Key Property. 
For each such $\bar{\mb}$, fix $\vv \in [2^\lambda]^{\leq \lambda}$ and $\Omega_* \subseteq [\lambda]^{<\sigma}$ 
satisfying \ref{kp}(a)-(b).  Choose an enumeration of these tuples  
\[ \bar{\xs} = \langle (\bar{\mb}_\delta, \vv_\delta, \Omega_\delta) : \delta \in S \rangle \]
which satisfies 
\begin{itemize}
\item $ S \subseteq \{ \delta < 2^\lambda : \delta > \alpha_* \mbox{ and } \delta \mbox{ is divisible by } \lambda  \}. $
\item $ \vv_\delta \subseteq \delta$ for $\delta \in S$. 
\end{itemize}
For each $\delta \in S$, let $\bar{\mb}^\prime_\delta$ be a multiplicative refinement of $\bar{\mb}_\delta$ as guaranteed by 
Definition \ref{kp} in the case where $\alpha$ of \ref{kp} is replaced by $\delta$. Without loss of generality\footnote{We may always appeal to 
an automorphism of $\ba$ which is the identity on $\ba \rstr \delta$ to find a sequence of appropriate support, as e.g. in Observation \ref{o:support} below.} 
$\bar{\mb}^\prime_\delta$ is supported by $\ba \rstr (\delta + \lambda)$.

Let $\mathfrak{A}, I_0, \ee$ be given by \ref{obs:elem}, so each $N \in I_0$ is an elementary submodel of $\mathfrak{A}$ of size $<\sigma$. 
Since $\ee$ is fine, we may assume that on an $\ee$-large set $I \subseteq I_0$ the elements 
$\lambda, \mu, \theta, \sigma, \ba, \bar{\xs}$ belong to $N$.  
When $N \in I$ and $\lambda, \sigma, \delta, \Omega_\delta \in N$, $N \models$ ``$\Omega_\delta$ is closed 
and unbounded in $[\lambda]^{<\sigma}$'', so from our external point of view, 
$N \in I$ and $\delta \in S \cap N$ implies $\lambda \cap N \in \Omega_\delta$. 

Since each $N$ is small it will (from an external point of view) only contain a small part of each of these objects. 
Both $\lambda \cap N$ and $S \cap N$ are of size $<\sigma$. 
Given $N$ and $\delta \in S \cap N$, $\mb_{\delta, \lambda \cap N}$ is (from an external point of view) 
an element of $\langle \mb_{\delta, u} : u \in [\lambda]^{<\sigma} \rangle$, which we may call the 
``canonical element'' for the sequence $\bar{\mb}_\delta$ as seen by $N$.  

Fix for awhile $N \in I$. Let $\md_N = \cap \{ \md : \md \in N \cap X_* \}$, so $\md_N \in \de^*_0$ is supported by $\ba \rstr \alpha_*$. 

Enumerate ${S \cap N} = \langle \delta_\epsilon : \epsilon < \epsilon_N \rangle$ in increasing order, for some limit 
ordinal $\epsilon_N < \sigma$.  Working in the large background model, by induction on $\epsilon \leq \epsilon_N$ we will build an increasing continuous sequence 
$\langle f_\epsilon = f_{N,\delta_\epsilon} : \epsilon \leq \epsilon_N \rangle$ such that:
\begin{enumerate}
\item[(a)] each $f_{\epsilon} \in \fin_{\mu,\theta}(\delta_\epsilon)$, so necessarily $\mx_{f_{\epsilon}} \in \ba^+$. 
\item[(b)] if $\gamma < \epsilon \leq \epsilon_N$ then $f_{\gamma} \subseteq f_{\epsilon}$, and if $\epsilon \leq \epsilon_N$ is 
a nonzero limit ordinal, then $f_{\epsilon} = \bigcup \{ f_{\gamma} : \gamma < \epsilon \}$, i.e. the sequence is continuous 
and increasing. 
\item[(c)] if $\epsilon = \gamma + 1 < \epsilon_N$ then either 
\[ \mx_{f_{\epsilon}} \leq \mb_{{\delta_\gamma}, {\lambda \cap N}} \]
or else $\mx_{f_\epsilon}$ is disjoint to some $\mb_{{\delta_\gamma}, v}$ where 
$v \in [{\lambda \cap N}]^{<\aleph_0}$.  
\item[(d)] if $\epsilon = \gamma + 1 < \epsilon_N$ and $\mx_{f_{\epsilon}} \leq \mb_{{\delta_\gamma}, {\lambda \cap N}}$, then 
in addition $\mx_{f_{\epsilon}} \leq \mb^\prime_{{\delta_\gamma}, {\lambda \cap N}}$.
\item[(e)] $\mx_{f_0} \leq \md_N$. 
\end{enumerate}
\noindent If $\epsilon = 0$, choose $f_0$ to satisfy (a) and (e), recalling that all elements of $S$ are $\geq \alpha_*$. Otherwise, 
arriving to $\epsilon$, let $g = \bigcup_{\gamma < \epsilon} f_\epsilon$.
If $\epsilon = \epsilon_N$ or a nonzero limit ordinal, this suffices, so let $f_\epsilon = g$.  
Otherwise, suppose $\epsilon = \gamma + 1$, so $g = f_\gamma \in \fin_{\mu,\theta}(\delta_\gamma)$ by hypothesis (b).  
As observed above, working in $\ba$ we see that the canonical object for the sequence $\mb_{\delta_\gamma}$ 
in $N$, $\mb_{{\delta_\gamma}, {\lambda \cap N}}$, is an element $\mc$ of $\ba$ which is, by choice of $S$, supported by $\ba \rstr \delta_\gamma$. 
So we may extend $g$ to $g^\prime \in \fin_{\mu,\theta}(\delta_\gamma)$ such that either 
$\mx_{g^\prime} \leq \mc$ or $\mx_{g^\prime} \cap \mc = 0$. (Either $\mx_g \cap \mc = 0$ already, 
or not, and if not we can find $g^\prime$ so that $\mx_{g^\prime} \leq \mc$ using the fact that $\bar{\mb}_{\delta_\gamma}$ is continuous; so 
clause (c) holds.)
Since we had chosen $\bar{\mb}^\prime_{\delta_\gamma}$ to satisfy the Key Property and to be supported by $\ba \rstr (\delta_\gamma + \lambda)$, 
and $\delta_\gamma + \lambda \leq \delta_\epsilon$ by choice of $S$, the Key Property ensures that {if} 
$\mx_{g^\prime} \leq \mc =  \mb_{{\delta_\epsilon}, {\lambda \cap N}}$, {then} 
we may extend $g^\prime$ to $g^{\prime\prime} \in \fin_{\mu, \theta}(\delta_\epsilon)$ so that 
\[ \mx_{g^{\prime\prime}} \leq \mb^\prime_{{\delta_\epsilon}, {\lambda \cap N}}. \]  
Let $f_\epsilon = g^{\prime\prime}$. This completes the 
inductive successor step and therefore the construction. Let $f_N = f_{\epsilon_N}$. 

Having constructed $f_N$ for each $N \in I$, we now consider the Boolean algebra $\ba$. Each $f_N$ corresponds to the nonzero element 
$\mx_{f_N} \in \ba^+$. Observe that for each $N \in I$ we may choose a $\sigma$-complete ultrafilter $D_N$ on $\ba$
which contains $\mx_{f_N}$. Moreover for each $N$, the choice of $\mx_{f_0}$ in condition (e) above ensures that 
$D_N$ contains $\{ \md : \md \in N \cap X_*\}$. 

In the remainder of the proof, we will build the optimal ultrafilter $\de$ on $\ba$ by averaging these ultrafilters. 
Recalling Fact \ref{f:sct}, let $\ee$ be an ultrafilter on $I$ which is $\sigma$-complete and normal. Define  
\begin{equation}
\de_* = \operatorname{Av}_{\ee} (\langle D_N : N \in I \rangle) = \{ \ma \in \ba : \{ N \in I : \ma \in D_N \} \in \ee \}.
\end{equation}

Let us verify that $\de_*$ is $\sigma$-complete. Given a sequence $\langle \ma_i : i < i_* <\sigma \rangle$ 
of elements of $\de_*$, by definition, for each $i$ the set 
$X_i := \{ N \in I : \ma_i \in \de_N \} \in \ee$. 
Since $\ee$ is $\sigma$-complete, $X := \bigcap \{ X_i  : i < i_* \} \in \ee$. $N \in X$ means that  
$\{ \ma_i : i < i_* \} \subseteq D_N$, so by the $\sigma$-completeness of $D_N$ we have that 
$\ma := \bigcap \{ \ma_i : i < i_* \} \in D_N$. Then 
$\{ N \in I : \ma \in D_N \} \supseteq X \in \ee$, so $\ma \in \de_*$ by definition of $\de_*$. 
Clearly $\de_* \supseteq \de^*_0$ because $\md \in X_*$ implies that 
$\{ N \in I : \md \in N \} \in \ee$ which implies that 
$\{ N \in I : \md_N \leq \md \} \in \ee$ which implies that 
$\{ N : \mx_{f_N} \leq \md \} \in X_*$ which implies that $\md \in \de_*$. 
This suffices recalling that $X_*$ generates $\de^*_0$.  

Now suppose $\bar{\mb}$ has the $(\lambda, \mu, \theta, \sigma)$-Key Property. 
Let $\delta \in S$ be such that $\bar{\mb} = \bar{\mb}_\delta$ in the enumeration above. Let $\bar{\mb}^\prime$ 
be its canonical multiplicative refinement $\bar{\mb}^\prime_\delta$, chosen at the beginning of the proof. 
It will suffice to show that if $\mb_{\delta, u} \in \de_*$ for each $u \in [\lambda]^{<\aleph_0}$, then for each 
$i < \lambda$, 
\[ \mb^\prime_{\{i\}} \in \de_*. \] 
Fix for awhile such an $i$. 
Consider any $u \in [\lambda]^{<\sigma}$. By our present assumptions, $\mb_u \in \de_*$. 
By definition of $\de_*$, this means that 
\[ \{ N \in I : \mb_u \in D_N \} \in \ee \]
and moreover that $\mb_u$ occurs $\ee$-a.a. in models $N$ such that $\{ \delta, i \} \subseteq N$, i.e. 
models which consider $\mb_u$ to be part of the correct problem and contain the index $i$: 
\[ \{ N \in I : \{ \delta, i \} \in N \mbox{ and if } u \in ([\lambda]^{<\sigma}) \cap N \mbox{ then } \mb_u \in D_N \} \in \ee. \]
For any $N$ in this set, 
$\delta \in S \cap N$ and so by construction (see clause (d) )
\[ \mx_{f_{N,\delta}} \leq \bigcap \{ \mb_u : u \in [\lambda]^{<\sigma} \cap N \}. \]
So by clause (d) above, $\mx_{f_N} \leq \mb^\prime_{\{i\}}$. This shows that $\mb^\prime_{\{i\}} \in \de_*$, and as $i < \lambda$ was arbitrary, this completes the proof. 
\end{proof}

Our next task is to build a useful choice for $\de^*_0$ mentioned in the previous proof, with the aim that any 
$\de$ built from $(\de_0, \ba, \de_*)$ where $\de_0$ is regular and excellent and 
$\de_*$ is from Theorem \ref{t:optimal} and this given $\de^*_0$ will not saturate any non-simple theory. 
The following fact is well known; we include a proof for completeness. 

\begin{fact} \label{f-seq}
Assume $\mu = \mu^{<\sigma}$.
Then there is $\langle u^\alpha_\epsilon : \epsilon < \mu, \alpha < \mu^+ \rangle$ such that:
\begin{enumerate}
\item $u^\alpha_\epsilon \in [\alpha]^{<\sigma}$
\item $\beta \in u^\alpha_\epsilon \implies u^\beta_\epsilon = u^\alpha_\epsilon \cap \beta$
\item if $u \in [\mu^+]^{<\sigma}$ then for some $\epsilon < \mu$ we have that
$(\forall \beta \in u) ( u \cap \beta \subseteq u^\beta_\epsilon)$.
\end{enumerate}
\end{fact}

\begin{proof} To begin, fix a sequence 
$\langle f_\alpha : \alpha < \mu^+ \rangle$ where each $f_\alpha : \alpha \rightarrow \mu$  
is injective and in addition is surjective whenever $|\alpha| \geq \mu$. 
Define a symmetric binary relation $E$ on $u_1, u_2 \in [\mu^+]^{<\sigma}$ by: $E(u_1, u_2)$ if 
(a) $\otp(u_1) = \otp(u_2)$, and (b) if $h : u_1 \rightarrow u_2$ is order-preserving 
and onto, and $\beta < \alpha$ are from $u_1$, then $f_\alpha(\beta) = f_{h(\alpha)}(h(\beta))$.  
Then $E$ is an equivalence relation with $\leq \mu$ classes. [To see this, note 
the number of its classes is bounded by the following count: 
for each set $u$, we first choose an order type ($\leq \sigma$ options), and then for each $\alpha \in u$ (of which there are $<\sigma$), 
there are $\leq \mu^{<\sigma}$ possible choices for the values of $f_\alpha$ on $u \cap \alpha$. As we assumed $\mu^{<\sigma} = \mu$, 
$E$ is therefore an equivalence relation with $\sigma \cdot (\mu^{<\sigma})^{<\sigma} = \mu$ equivalence classes.] 
Let $\langle E_\epsilon : \epsilon < \mu \rangle$ list these classes. 

Now for each $u \in [\alpha]^{<\sigma}$, 
let $u^\alpha_\epsilon = u$ iff for some $w \in E_\epsilon$ and $\gamma \in w$  
\[ (w \cap (\gamma + 1)) E (u \cup \{ \alpha \} ) \] 
Otherwise, $u^\alpha_\epsilon$ is empty.  
This sequence will satisfy (1) by construction. 
To see that (2) is satisfied, suppose we are given $u = u^\alpha_\epsilon$, $w$, $\gamma$ satisfying the previous equation, 
and $\beta \in u^\alpha_\epsilon$. 
Then letting $w^\prime = u^\alpha_\epsilon \cap \beta$ and $\gamma^\prime = h^{-1}(\beta)$ suffices. Finally, condition (3) follows 
from condition (2) and the fact that for each $\alpha$ and at least one $\epsilon$, $u^\alpha_\epsilon$ is well defined (non-empty). 
\end{proof}

\begin{claim} \label{cl:non-sat}
Let $(\lambda, \mu, \theta, \sigma)$ be suitable, $\mu = \mu^{<\theta}$, $\sigma$ uncountable and compact. 
There exists a $\sigma$-complete filter $\de^*_0$ on $\ba^1_{2^\lambda, \mu, \theta}$ generated by $\mu^+$ sets 
such that:
\begin{quotation} 
\noindent if $\de$ is a regular ultrafilter on $\lambda$ built from $(\de_0, \ba, \de_*, \jj)$ where $\ba = \ba^1_{2^\lambda, \mu, \theta}$ 
and $\de_*$ is a $\sigma$-complete ultrafilter on $\ba$ extending $\de^*_0$, then $\de$ is not $\mu^{++}$-good for any 
non-simple theory.
\end{quotation}
\end{claim}

\begin{proof}
Let $\langle u^\alpha_\epsilon : \epsilon < \mu, \alpha < \mu^+ \rangle$ be given by Fact \ref{f-seq}.
Let $\langle \mc_\gamma : \gamma < \mu \rangle$ be a maximal antichain of $\ba$. 
Define
\[ \de^*_0 = \{ \bigcup_{\gamma \in A} \mc_\gamma : ~\mbox{for some $u \in [\mu^+]^{<\sigma}$,
$A \supseteq \{ \epsilon: (\forall \alpha \in u) (u \cap \alpha \subseteq u^\alpha_\epsilon)$} \} \}. \]
$\de^*_0$ is a $\sigma$-complete filter on $\mu$ by Fact \ref{f-seq}.  It is supported by 
$\ba \rstr \vv$ when it supports all the $\mc_\gamma$'s, and so satisfies the requirement in \ref{t:optimal}.2. 

We now look for an omitted type. 
Let $\vp(x;\overline{y})$ have the tree property in $\mathfrak{C}$ (so without loss of generality it will have $TP_1$ or $TP_2$; 
we observe this distinction, but ultimately don't really use it).
We choose $\overline{a}_\eta \in {^{\ell(\overline{y})}\mathfrak{C}}$ for $\eta \in {^{\sigma >} {\mu^+}}$
so that: 
\begin{enumerate}
\item if $\eta \in {^{\sigma}(\mu^+)}$ then $\{ \vp(x;\overline{a}_{\eta|_{i+1}}) : i < \sigma \}$ is a $1$-type.
\item if $\vp$ has $TP_1$ then for $\eta, \nu$ incomparable elements of
${^{\sigma > }(\mu^+)}$, the set 
\[ \{ \vp(x;\overline{a}_\eta), \vp(x;\overline{a}_\nu) \} \] 
is inconsistent,
\item if $\vp$ has $TP_2$ then for $i \in \sigma$ and $\eta, \nu$ incomparable (or equivalently, not equal) elements of
${^{i}(\mu^+)}$, the set $\{ \vp(x;\overline{a}_\eta), \vp(x;\overline{a}_\nu) \}$ is inconsistent.
\end{enumerate}
Now we use \ref{f-seq} to pick a proposed path through the tree. 
For $\epsilon < \mu$, $\alpha < \mu^+$ let
$\eta_{\epsilon, \alpha}$ list
$u^\alpha_\epsilon \cup \{ \alpha \}$ in increasing order,
so $\eta_{\epsilon, \alpha} \in {^{\sigma>}\mu^+}$.
Fix a partition $\langle C_\gamma : \gamma < \mu \rangle$ of $I$ such that $\jj(C_\gamma) = \mc_\gamma$ for each $\gamma < \mu$. 
For each $\alpha < \lambda$, define the function $f_\alpha$
from $I$ to ${^{\ell(\overline{y})}\mathfrak{C}}$ by: if $t \in C_\epsilon$ then $f_\alpha(t) = \overline{a}_{\eta_{\epsilon, \alpha}}$.

Let $p  = \{ \vp(x, f_\alpha/\de) : \alpha < \mu^+ \}$.
Then $p$ is a set of formulas of the language of $T$ with
parameters from $\mathfrak{C}^I/\de$.
Let us check that $p$ is a consistent partial $1$-type. Since consistency depends on comparability in the partial ordering, it suffices to check all pairs.  
If $\beta < \alpha < \mu^+$, let $A_1 = \{ \epsilon < \mu : \beta \in u^\alpha_\epsilon \}$. By construction, 
$\bigcup_{\gamma \in A_1} \mc_\gamma \in \de^*_0$, so we may choose $X_1 \in \de$ in the $\jj$-preimage of this set. 
Unraveling the construction, this shows that  
$p$ is indeed a consistent partial $\vp$-type.

Assume for a contradiction that we have $f \in {^I\mathfrak{C}}$ such that $f/\de$ realizes $p$ in $\mathfrak{C}^I/\de$.
For each $\alpha < \mu^+$, let
\[ A_\alpha = \{ t \in I : \mathfrak{C} \models \vp(f(t), f_\alpha(t)) \}  \in \de\]
Let $\ma_\alpha = A_\alpha/\de$ be the corresponding member of $\ba$. 

As $\langle \mc_\gamma : \gamma < \mu \rangle$ is an antichain
of $\ba$, for each $\alpha < \mu^+$ there is
$\gamma(\alpha) < \mu$ such that
$\mb_\alpha = \ma_\alpha \cap \mc_{\gamma(\alpha)} \in \ba^+$. By the pigeonhole principle, there is
$\zeta < \mu$ such that $|\uu_1| = \mu^+$, where
\[   \uu_1 = \{ \alpha < \mu^+ : \gamma(\alpha) = \zeta \}.    \]
As $\{ \operatorname{otp}(u_\alpha) : \alpha \in \uu_1 \}$ has cardinality
$\leq \sigma < \mu^+$, there is $\rho < \sigma$ such that $|\uu_2| = \mu^+$, where
\[ \uu_2 = \{ \alpha \in \uu_1 : \operatorname{otp}(u_\alpha) = \rho \} \]
However, $\ba$ has the $\mu^+$-c.c. So for some $\alpha \neq \beta
\in \uu_2$, we have that $\mb := \mb_\alpha \cap \mb_\beta$ is positive.
Let $B \subseteq I$ be such that $\jj(B) = \mb$, so $B \neq \emptyset \mod \de_0$.
Since $B$ is contained in each of $A_\alpha$, $A_\beta$, and $C_\zeta$ $\mod \de_0$,
we may choose $t \in B \cap (A_\alpha \cap A_\beta \cap C_\zeta)$.

Recall that $\eta_{\epsilon, \alpha}$ lists $u^\alpha_\epsilon \cup \{ \alpha \}$
in increasing order, and $A_\alpha = \{ t \in I : \mathfrak{C} \models \vp(f(t), f_\alpha(t)) \}$.
Thus by our choice of $t$ and $\zeta$, $f(t)$ realizes 
\[ \{ \vp(x,f_\alpha(t)), \vp(x,f_\beta(t)) \} = \{ \vp(x, \overline{a}_{\eta_{\zeta, \alpha}}), \vp(x, \overline{a}_{\eta_{\zeta, \beta}}) \} \]
But $\eta_{\zeta, \alpha}, \eta_{\zeta, \beta}$ are distinct
members of ${^{\sigma >} (\mu^+)}$ of the same length. Thus by our choice of parameters in the tree, 
the set $\{ \vp(x, \overline{a}_{\eta_{\zeta, \alpha}}), \vp(x, \overline{a}_{\eta_{\zeta, \beta}}) \}$ is inconsistent.
[Note that this does not depend on whether we are in the case of $TP_1$ or of $TP_2$.]
This contradiction shows that $p$ is not realized, and so completes the proof. 
\end{proof}

\begin{concl} \label{c:uf-ba}
Let $(\lambda, \mu, \theta, \sigma)$ be suitable, $\mu < \lambda$, 
and suppose $\sigma$ is uncountable and supercompact.  
Let $\ba = \ba^1_{2^\lambda, \mu, \theta}$. 
Then there is an ultrafilter $\de_*$ on $\ba$ such that:
\begin{enumerate}
\item[(a)] $\de_*$ is $(\lambda, \mu, \theta, \sigma)$-optimal. 
\item[(b)] whenever $\de$ is a regular ultrafilter built from $(\de_0, \ba, \de_*)$ 
then $\de$ is not $\mu^{++}$-good for any non-simple theory, \emph{thus} not $\mu^{++}$-good.  
\end{enumerate}
\end{concl}

\begin{proof}
Use the filter $\de^*_0$ from Claim \ref{cl:non-sat} in Theorem \ref{t:optimal}(2). 
\end{proof}

\begin{defn} \label{d:optimized}
Let $(\lambda, \mu, \theta, \sigma)$ be suitable and suppose that $\sigma$ is uncountable and that a 
$(\lambda, \mu, \theta, \sigma)$-optimal ultrafilter exists.  
We say the ultrafilter $\de$ on $I$, $|I| = \lambda$ is \emph{$(\lambda, \mu, \theta, \sigma)$-optimized} when there exists a regular 
excellent filter $\de_0$ on $I$ and a $(\lambda, \mu, \theta, \sigma)$-optimal ultrafilter $\de_*$ on 
$\ba^1_{2^\lambda, \mu,\theta}$ such that 
$\de$ is built from $(\de_0, \ba, \de_*)$.  Note that any such $\de$ will be regular. 
\end{defn}

We now record the following connection. On the relevance of this property, 
see section \ref{s:why-compact}, ``Why a large cardinal?'' 

\begin{defn} \emph{(Flexible filters, Malliaris \cite{mm4})}
\label{flexible}
We say that the filter $\de$ on a set $I$ is $\lambda$-flexible if for any $f \in {^I \mathbb{N}}$ with
$n \in \mathbb{N} \implies n <_{\de} f$, we can find $X_\alpha \in \de$ for $\alpha < \lambda$ such that 
for all $t \in I$
\[ f(t) \geq | \{ \alpha : t \in X_\alpha \}|  \]
Informally, we can find a $\lambda$-regularizing family below any given nonstandard integer. 
\end{defn}

\begin{obs} \label{c:optimal-is-flexible} 
Let $(\lambda, \mu, \theta, \sigma)$ be suitable and $\sigma > \aleph_0$ supercompact. 
Let $\de$ be a $(\lambda, \mu, \theta, \sigma)$-optimized ultrafilter. 
Then $\de$ is flexible. 
\end{obs}

\begin{proof} 
It was proved in \cite{mm4} Section 8 that any regular ultrafilter which is good for some non-low simple theory must be flexible. 
So this Observation will be an immediate corollary of the theorem, proved in the next section, 
that any such optimized ultrafilter is good for any countable simple theory. 

One can also give a direct proof, which we only sketch as it is not central for our arguments. 
We know $\de$ is built from some $(\de_0, \ba, \de_*)$ where $\de_0$ is excellent (therefore good, therefore 
flexible) and $\de_*$ is $\sigma$-complete for $\sigma > \aleph_0$. 
Then the argument is exactly that worked out in Malliaris-Shelah \cite{MiSh:997} Claim 7.8. In particular, nothing about optimality of $\de_*$ 
is used, only its $\sigma$-completeness on a completion of a free Boolean algebra. 
\end{proof}

This gives a new solution to an old question of Dow \cite{dow}, which we had also answered 
in an earlier paper \cite{MiSh:996} assuming a measurable cardinal. 

\begin{qst}[Dow 1985, in our language]
Does there exist a regular ultrafilter which is $\kappa^+$-flexible and not $\kappa^+$-good? 
\end{qst}

\begin{concl} \label{c:dow}
Let $(\lambda, \mu, \theta, \sigma)$ be suitable and $\sigma > \aleph_0$ supercompact. Then there is a regular ultrafilter $\de$ on $\lambda$ 
which is $\lambda$-flexible and not $\mu^{++}$-good. 
\end{concl}

\begin{proof}
An optimized ultrafilter will fit the bill by \ref{c:uf-ba} and \ref{c:optimal-is-flexible}.
\end{proof}

%

Our final claim shows that ``not $\mu^{++}$-good'' in \ref{c:uf-ba}(b) is best possible.  

\begin{claim} \label{c:sharp}
Let $(\lambda, \mu, \theta, \sigma)$ be suitable. 
If $\de_*$ is a $(\lambda, \mu, \theta, \sigma)$-optimal ultrafilter on $\ba = \ba^1_{2^\lambda, \mu, \theta}$ then:
\begin{enumerate}
\item $\de_*$ is $\mu^+$-good.
\item if $\de$ is a $(\lambda, \mu, \theta, \sigma)$-optimized ultrafilter built from some regular excellent $\de_0$ 
along with $\de_*$ and $\ba$, then $\de$ is $\mu^+$-good. 
\end{enumerate}
\end{claim}

\begin{proof}
For (1) it will suffice to show that if $\langle \mb^1_u : u \in \Omega = [\mu]^{<\sigma} \rangle$ is a continuous sequence of 
elements of $\de$ then it has a multiplicative refinement. Define $\bar{\mb}^2 = \langle \mb^2_u : u \in [\lambda]^{<\sigma} \rangle$ 
by: $\mb^2_u = \mb^1_{u \cap \mu}$. It suffices to show that 
$\bar{\mb}^2$ has the Key Property \ref{kp}, and therefore has a multiplicative refinement by optimality. 
Choose $\vv \subseteq 2^\lambda$, $|\vv| \leq \mu$ as in \ref{kp}(a).  Let $\alpha < 2^\lambda$ 
be any ordinal such that $\alpha \supseteq \vv$. Define a new maximal antichain of $\ba$ by $\langle \mc_\epsilon : \epsilon < \mu \rangle$
where $\mc_\epsilon = \mx_{f_\epsilon}$ and $f_\epsilon$ is the function with domain $\{\alpha\}$ and range $\epsilon$. 
Let $\langle u_\epsilon : \epsilon < \mu \rangle$ list $\Omega = [\mu]^{<\sigma}$. Define $\mb^\prime_{\{i\}}$ for $i < \lambda$ by: 
if $i \geq \mu$, then $\mb^\prime_{\{i\}} = 1_\ba$, and if $i<\mu$ then 
$\mb^\prime_{\{i\}} = \bigcap \{ \mb^2_{u_\epsilon} \cap \mc_\epsilon : i \in u_\epsilon \}$. Let $\langle \mb^\prime_u : u \in [\lambda]^{<\sigma} \rangle$
 be the multiplicative sequence generated by $\langle \mb^\prime_{\{i\}} : i < \lambda \rangle$. It remains to check that \ref{kp}(b) holds.  
Fix $u \in [\lambda]^{<\sigma}$ and let $\epsilon_*$ be such that $u \cap \mu = u_{\epsilon_*}$. 
If $\mx_f \leq \mb^2_u$, then $\mx_f \cap \mc_{\epsilon_*} \cap \mb^2_{u_{\epsilon_*}} > 0$, and $\mc_{\epsilon_*} \cap \mb^2_{u_{\epsilon_*}} \leq \mb^\prime_{\{i\}}$ 
for each $i \in u = u_{\epsilon_*}$. So as $\bar{\mb}^\prime$ is multiplicative, 
$\mx_f \cap \mc_{\epsilon_*} \cap \mb^2_{u_{\epsilon_*}} \leq \mb^\prime_u$ as desired. This completes the proof that $\bar{\mb}^2$ has the Key Property, 
and so the proof of (1). 

(2) follows from (1) by Theorem \ref{t:separation} page \pageref{t:separation}. 
\end{proof}

In Claim \ref{c:sharp}, it would be natural to consider adding 
a monotonicity clause of the form ``$\de_*$ is $(\lambda, \mu, \theta, \sigma)$-optimal and 
$\lambda^\prime \in [\mu, \lambda)$ implies $\de_*$ is $(\lambda^\prime, \mu, \theta, \sigma)$-optimal.'' This 
would require a slight change in the definition, 
since we have tied $\lambda$ to the size of the underlying Boolean algebra $2^\lambda$, so we omit it. If we 
were to add an additional parameter so as to separate these 
two uses of $\lambda$, the size of the Boolean algebra and the length of the sequence $\bar{\mb}$, then we have monotonicity 
in the second.


\section{Presentations in ultrapowers} \label{s:skolem}
\setcounter{equation}{0} 
\setcounter{theoremcounter}{0}

\br

In this section we prove a lemma saying that presentations for types in ultrapowers can be arranged to interact with a 
choice of lifting of the parameters in a nice way. This lemma will proceed by building an algebra $\zs$ which is the optional input 
to Lemma \ref{p:e}.2 above.

\begin{lemma} \label{c:lifting}
Suppose $(\lambda, \mu, \theta, \sigma)$ are suitable, $T$ is $(\lambda, \mu, \theta, \sigma)$-explicitly simple, 
$T$ eliminates imaginaries, $T$ is complete and simple with infinite models, 
and $|T| < \sigma$.  
Suppose we are given:
\begin{enumerate}
\item $\de$ a regular ultrafilter on $I$, $|I| = \lambda$.
\item $M \models T$ is $\lambda^+$-saturated, and admits an expansion $M^+$ by new Skolem functions for formulas of $T$. Let $T^+ = Th(M^+)$.  
\item $N \preceq M^I/\de$, $||N|| = \lambda$, $N$ admits an expansion to $N^+ \models T^+$ such that $N^+ \subseteq (M^+)^I/\de$, and 
$p \in \ts(N)$. 
\end{enumerate} 
Then there exist 
\begin{itemize}
\item a $(\lambda, \theta, \sigma)$-presentation $\xm = ( \langle \vp_\alpha(x,a^*_\alpha) : \alpha < \lambda \rangle, \zm )$ of $p$ 
\item an intrinsic coloring $G : \mcr_\xm \rightarrow \mu$, which we may assume has range $ = \mu$
\item and a choice of lifting $\{ a^*_\alpha : \alpha < \lambda \} \rightarrow M^I$
\end{itemize}
such that identifying each $a^*_\alpha$ with its image under this lifting, we have that 
for any $w = \clm(w) \subseteq \lambda$, 
\begin{enumerate}
\item[(a)] for every $t \in I$,  
\[ \mathfrak{C}_T \rstr \{ a^*_\alpha[t] : \alpha \in w \} \preceq \mathfrak{C}_T. \]
\item[(b)] for each finite sequence $\langle \alpha_{i_0}, \dots, \alpha_{i_{k-1}} \rangle$ of 
elements of $w$, and each formula $\vp(x,\bar{y})$ of $T$ with $\ell(\bar{y}) = k$, there is $\beta \in w$ 
such that for \emph{all} $t \in I$,
\end{enumerate}
\[ M  \models (\exists x)\vp(x, a^*_{\alpha_{i_0}}[t], \dots, a^*_{i_{k-1}}[t]) \implies 
\vp(a^*_\beta[t], a^*_{\alpha_{i_0}}[t], \dots, a^*_{i_{k-1}}[t]). \]
\end{lemma}

\noindent Before proving Lemma \ref{c:lifting}, we record that such types are enough. 

\begin{obs} \label{o:enough}
Let $T$, $\de$ be as in Lemma \ref{c:lifting}. 
To prove that $\de$ is good for $T$ it would suffice to show that 
every $p$ arising in the form $\ref{c:lifting}.3$ is realized. 
\end{obs}

\begin{proof}
Since the ultrafilter $\de$ is regular, we are free to choose any infinite model $M \models T$ as the index model, in particular 
we may choose it to be sufficiently saturated. Fix any $A \subseteq M^I/\de$, $|A| \leq \lambda$, and $p_0$ a type over $A$. 
Let $M^+$ be any expansion of $M$ by Skolem functions and let $T^+ = Th(M)$. We may assume $|T^+| = |T| < \sigma$. 
Since ultrapowers commute with reducts, there is an expansion of the ultrapower $M^I/\de$ to a model of $T^+$. 
In particular there is an elementary submodel $N^+$ of the ultrapower (in the expanded language) such that $A \subseteq \dom(N^+)$ and $||N^+|| = \lambda$. 
Let $N$ be the reduct of $N^+$ to $\tau(T)$ and let $p \in \ts(N)$ be any type extending $p_0$. Clearly to realize $p_0$ it suffices to 
realize $p$. Finally, regarding the range of the function $G$: if $G$ is an intrinsic coloring of some presentation $\xm$, 
let $E_G$ be the equivalence relation on elements of $\mcr_{\xm}$ given by $E_G(\xr, \xr^\prime)$ iff $G(\xr) = G(\xr^\prime)$. 
Then any function $G^\prime : \mcr_{\xm} \rightarrow \mu$ such that $E_{G^\prime}$ refines $E_G$ will also be an intrinsic 
coloring of $\mcr_{\xm}$, so we may assume the range of $G$ is exactly $\mu$. 

This completes the proof. 
\end{proof}

\br
\begin{proof}[Proof of Lemma \ref{c:lifting}.]
Let $\tau^+$ denote the signature of $N^+$ and 
$\tau$ that of $N$, and let $\mathfrak{C}$ denote $\mathfrak{C}_T$, the monster model for $T$. 
Let $E \subseteq \lambda$ denote the set of even ordinals less than $\lambda$. 
To begin, let 
\begin{equation}
\langle \vp_\alpha(x,  a^*_\alpha) : \alpha \in E  \rangle 
\end{equation}
be an enumeration of $p$ which satisfies: 
\begin{enumerate}
\item[(a)] each $a^*_\beta$ is a singleton, possibly imaginary; 
\item[(b)] $\{ a^*_\beta : \alpha \in E \} = \dom(N)$; 
\item[(c)] $\{ a^*_\beta : \alpha \in |T| \cap E \}$ is the domain of an elementary submodel $M_*$ of $N$ over which $p$ does not fork, and 
$\{ \vp_\alpha(x, a^*_\beta) : \alpha \in |T| \cap E \}$ is a complete type over this submodel. 
\end{enumerate}
For each $\alpha \in E$, choose $g_\alpha \in {^I M}$ such that 
first, $a^*_\alpha = g_\alpha/\de$, and second, if $\alpha, \alpha^\prime \in E$ and $a^*_\alpha = a^*_{\alpha^\prime}$, then $g_\alpha = g_{\alpha^\prime}$. 

Let $\terms$ be the set of all terms built up inductively from function symbols of $\tau^+$ and the free variables 
$\{ x_\alpha : \alpha \in E \}$. 
Choose a map $\rho : \terms \rightarrow \lambda$ such that:
\begin{enumerate}
\item[(i)] $\rho$ is one-to-one and onto. 
\item[(ii)] for each $\alpha \in E$, $x_\alpha \mapsto \alpha \in E$.
\item[(iii)] $\rho^{-1}( \{ \alpha : \alpha < |T| \}$ consists precisely of the elements of $\terms$ whose free variables are among 
$\{ x_\alpha : \alpha < |T| \}$.  
\end{enumerate}

Now we define functions $\{ g_\alpha : \alpha \in \lambda \setminus E \} \subseteq {^I M}$, i.e. we need to define the value of $g_\alpha$ when $\alpha$ is odd. 
Fix for awhile $\alpha \in \lambda \setminus E$. Then $\rho^{-1}(\alpha)$ is a term, say $\xt = \xt(x_{i_0}, \dots, x_{i_{k-1}}) \in \terms$, where 
$k$ is finite and depends on $\xt$ and this notation means that the free variables of $\xt$ are precisely $x_{i_0}, \dots, x_{i_{k-1}}$ 
~(in particular $\xt$ is \emph{not} necessarily a term arising as a single function applied to a series of variables). 
Fix $t \in I$. Since $M^+ \models T^+$, there is a unique $a \in \dom(M^+)$ such that 
\[ M^+ \models \mbox{``} a = \xt[ g_{i_0}(t), \cdots, g_{i_{k-1}(t)} ] \mbox{''} \]
where the expression in quotations means that the term $\xt$ evaluates in $M^+$, on the given sequence of values, to $a$. 
Assign $g_\alpha(t) = a$. As $\alpha \in \lambda \setminus E$ and $t \in I$ were arbitrary, this completes the definition of 
$\{ g_\alpha : \alpha < \lambda \setminus E \}$. Note that this definition applied to $g_\alpha$ for $\alpha$ even would just return $g_\alpha$. 

Before continuing, let us prove that if $N$ and $\{ a^*_\alpha : \alpha < |T| \}$ are closed under the functions of 
$(M^+)^I/\de$ then $\{ g_\alpha/\de : \alpha \in X\} = \{ g_\alpha/\de : \alpha \in E \cap X\} = X$
where $X \in \{ |T|, \lambda \}$. In other words, we have not actually added new elements to either $\dom(M_*)$ or $\dom(N)$, 
but have simply repeated existing elements in a larger enumeration. Note that the second equality 
was ensured by (b)-(c) above and the right to left inclusions are obvious. 
We prove the remaining inclusion, $\{ g_\alpha/\de : \alpha \in X\} \subseteq \{ g_\alpha/\de : \alpha \in E \cap X\}$, 
by induction on the complexity of the term $\rho^{-1}(\alpha)$. 
If $\rho^{-1}(\alpha)$ is a single variable $x_\beta$ with $\beta \in X$, then by our construction we know that 
$2\beta = \alpha$ and so $\alpha \in E \cap X$.  Suppose then that $\rho^{-1}(\alpha)$ is a term of the form 
$f_\vp(\xt_{i_0}, \dots, \xt_{i_{\ell-1}})$, where in slight abuse of notation, we write this to mean 
that $f_\vp$ is an $\ell$-place function symbol from $\tau^+ \setminus \tau$ applied to the terms $\xt_{i_0}, \dots, \xt_{i_{\ell-1}}$.
By inductive hypothesis, for each $j < \ell$ there is $\alpha_j \in E \cap X$ such that 
\[ g_{\rho(\xt_{i_j})}/\de = g_{\alpha_{i_j}}/\de . \]
Then as both $M_*$ and $N$ were expanded to models of $T^+$, 
writing ``$\mathfrak{C}^+ \rstr X$'' 
for the appropriate expansion, 
there is some $\beta \in E \cap X$ such that 
\[ \mathfrak{C}^+ \rstr X \models f(g_{\alpha_{i_0}}/\de, \dots, g_{\alpha_{i_{\ell-1}}}/\de) = g_\beta/\de \]
and then unraveling the definition of $g_\alpha$ in the previous paragraph, 
clearly $g_\alpha/\de = g_\beta/\de$. 

For $\alpha \in \lambda \setminus E$, let $\vp_\alpha$ be the formula ``$x \neq y$'', recalling that $p$ is nonalgebraic 
(of course even simpler formulas would work e.g. ``$x =x$''). 
For each $\alpha < \lambda$, let $a^*_\alpha = g_\alpha/\de$. Then the enumeration 
\[ \bar{\vp} = \langle \vp_\alpha(x, a^*_\alpha) : \alpha < \lambda \rangle \]
will satisfy the hypotheses of Lemma \ref{p:e}. So we have our enumeration and our lifting $a^*_\alpha \mapsto g_\alpha$ (for $\alpha < \lambda$), 
and it remains to translate these Skolem functions in the natural way into  
an algebra $\zs$ on $\lambda$ and to prove this algebra has the desired properties. 

For each function symbol $f_\vp \in \tau^+ \setminus \tau$, of arity $k = k_\vp$, add a function $F_\vp$ of the same arity to the algebra 
defined as follows. Although similar to the argument just given, this definition will have an important additional uniqueness property. 
For each $\langle \alpha_{i_0}, \dots, \alpha_{i_{k-1}} \rangle \in {^k \lambda}$, in slight abuse of notation, define 
\[ F_\vp ( \alpha_{i_0}, \dots, \alpha_{i_{k-1}}  ) = \rho ( ~\mbox{``}f_\vp ( \rho^{-1}(\alpha_{i_0}), \cdots, \rho^{-1}(\alpha_{i_{k-1}}))\mbox{''} ~) \]
where the expression in quotation denotes the element of $\terms$ formed by applying the $k$-place function symbol $f_\vp$ to the sequence of terms 
$\rho^{-1}(\alpha_{i_0})$, $\cdots$, $\rho^{-1}(\alpha_{i_{k-1}})$. 
As $\rho$ was a bijection, this value is unique and well defined. 
Let $\zs$ be the algebra given by the functions $\{ F_\vp : f_\vp \in \tau^+ \setminus \tau \}$, so clearly $|\zs| = |\tau^+ \setminus \tau| \leq |T| < \sigma$. 
We now make several observations about how the algebra $\zs$ interacts with the enumeration $\bar{\vp}$ and the functions of $\tau^+ \setminus \tau$.  
First, our construction has guaranteed that: 
\begin{enumerate}
\item[(1)] For each Skolem function $f_\vp \in \tau^+ \setminus \tau$, say of arity $k$, 
and every distinct $\alpha_{i_0}, \dots, \alpha_{i_{k-1}}$ from $\lambda$, there exists $\beta < \lambda$ such that 
for \emph{all} $t \in I$, $M^+ \models$ ``$f_\vp(a^*[t]_{\alpha_{i_0}}, \dots, a^*_{\alpha_{i_{k-1}}}[t]) = a^*_\beta[t]$.''
\end{enumerate}
This is more than would be guaranteed a priori by \lost theorem: \los would say that if we fix an enumeration and a lifting, then 
for any such sequence of $\alpha$'s we may find a $\beta$ which works almost everywhere. Here we have a single $\beta$, 
namely the value of $F_\vp ( \alpha_{i_0}, \dots, \alpha_{i_{k-1}}  )$, which works everywhere.\footnote{Note that we accomplished 
this by ``padding'' our original enumeration so that $\langle a^*_\alpha : \alpha < \lambda \rangle$ 
may contain many repetitions.}

\begin{enumerate}
\item[(2)] 
For each $k<\omega$ and each $f_\vp \in \tau^+ \setminus \tau$ of arity $k$, there is a function $F_\vp \in \zs$ of arity $k$ such that 
whenever $\alpha_{i_0}, \dots, \alpha_{i_{k-1}}$, $\beta$ satisfy condition (a), 
\[ F^{\zs}_\vp(\alpha_{i_0}, \dots, \alpha_{i_{k-1}}) = \beta. \] 
\end{enumerate}
This says that the $F_\vp$ translate the action of the Skolem functions in the natural way.  
Finally, let us prove that:

\begin{enumerate}
\item[(3)] For any nonempty $w = \clmo(w) \subseteq \lambda$ and $t \in I$, 
$\mathfrak{C} \rstr \{ a^*_\alpha[t] : \alpha \in w \} \preceq \mathfrak{C}$. 
\end{enumerate}
Since we have fixed a lifting, $a^*_\alpha[t] = g_\alpha(t)$, so we will use these interchangeably. 
Fix some such $w = \clmo(w) \subseteq \lambda$ and some $t \in I$. 
Since this set is a subset of the index model $M$, it will suffice to prove that $M \rstr \{ g_\alpha(t) : \alpha \in w \} \preceq M$. 
Suppose for a contradiction that $\{ g_\alpha(t) : \alpha \in w \} = \{ a^*_\alpha[t] : \alpha \in w \}$ is not the domain of an elementary submodel of $M$. 
Then there are a formula $\vp = \vp(x,\bar{y})$ and $\alpha_{i_0}, \dots, \alpha_{i_{\ell(\bar{y})-1}} \in w$ such that 
\[ M \models (\exists x)\vp(x, a^*_{\alpha_{i_0}}[t], \dots, a^*_{i_{\ell(\bar{y})-1}}[t]) \] 
but there does not exist $\gamma \in w$ such that 
\[ M \models  \vp[a^*_\gamma[t], a^*_{\alpha_{i_0}}[t], \dots, a^*_{i_{\ell(\bar{y})-1}}[t]].  \]
Let $f_\vp \in \tau^+$ be the function symbol whose interpretation in $T^+$ corresponds to the Skolem function for $\vp$. 
Then 
\begin{align*} M^+  \models & (\exists x)\vp(x, a^*_{\alpha_{i_0}}[t], \dots, a^*_{i_{\ell(\bar{y})-1}}[t]) \implies \\  
& \vp(f_\vp( a^*_{\alpha_{i_0}}[t], \dots, a^*_{i_{\ell(\bar{y})-1}}[t] ), a^*_{\alpha_{i_0}}[t], \dots, a^*_{i_{\ell(\bar{y})-1}}[t]).
\end{align*}
Moreover, by observation (1), there is $\beta < \lambda$ such that for all elements of $I$, and in particular for the $t$ we have chosen, 
\[ M^+ \models f_\vp(a^*[t]_{\alpha_{i_0}}, \dots, a^*_{\alpha_{i_{k-1}}}[t]) = a^*_\beta[t]. \]
By observation (2), $w = \cls(w)$ means that necessarily $\beta \in w$. This contradiction completes the proof of (3).  
\br

Notice that if $\zm \supseteq \zs$ is any larger algebra and $w = \clm(w) \subseteq \lambda$, then a fortiori 
$w = \clmo(w)$ so (3) remains true.  

Then the enumeration $\bar{\vp}$ and the algebra $\zs$ satisfy the hypotheses of Definition \ref{d:c3}. 
As we have assumed that $T$ is $(\lambda, \mu, \theta, \sigma)$-explicitly simple, we may apply that Definition 
to obtain a presentation $\xm = (\bar{\vp}, \zm)$ and an intrinsic coloring $G: \mcr_{\xm} \rightarrow \mu$, with $\zm \supseteq \zs$. 
This presentation $\xm$, the coloring $G$, and the lifting $\langle g_\alpha : \alpha < \lambda \rangle$ are as desired.  
Note that by definition of presentation, any $\clm$-closed set is nonempty, so we no longer need the proviso ``nonempty'' when 
applying (3) in the context of a presentation, Definition \ref{d:pres}. 

This completes the proof of the claim. 
\end{proof}

%
%

\newpage

\section{Ultrapower types in simple theories} \label{s:ultrapower}
\setcounter{equation}{0} 
In this section we assume the following:

\begin{itemize}
\item $(\lambda, \mu, \theta, \sigma)$ are suitable, and in addition $\sigma > \aleph_0$ is compact.\footnote{What we use 
is that there is an optimal, thus $\sigma$-complete, ultrafilter on $\ba$. We also use that 
$\sigma$ is strongly inaccessible in Claim \ref{d:good-choice}.} 
\item $\ba = \ba^1_{2^\lambda, \mu, \theta}$.
\item $T$ is complete, countable, first-order, and $(\lambda, \mu, \theta, \sigma)$-explicitly simple. 
\end{itemize}

In the next Theorem, we have assumed existence of an optimal ultrafilter 
rather than ``$\sigma$ is uncountable and supercompact''.  This is because supercompactness was 
used for expediency to \emph{construct} an optimal ultrafilter but nothing about the definition of optimal or optimized 
(Definition $\ref{d:optimized}$)
seems to suggest its necessity, and nothing about supercompactness is otherwise used in the proof. 

\begin{theorem} \label{t:saturation}
Suppose $(\lambda, \mu, \theta, \sigma)$ are suitable.
Suppose a $(\lambda, \mu, \theta, \sigma)$-optimal ultrafilter exists. 
Let $T$ be a complete, countable theory which is $(\lambda, \mu, \theta, \sigma)$-explictly simple, let $M \models T$, and 
let $\de$ be a $(\lambda, \mu, \theta, \sigma)$-optimized ultrafilter on $I$, $|I| = \lambda$. Then 
$M^I/\de$ is $\lambda^+$-saturated. 
\end{theorem}

\begin{proof}
As ultraproducts commute with reducts, we may assume $T$ eliminates imaginaries (if not, work in $T^{eq}$ throughout). 
To begin, quoting Lemma \ref{c:lifting} and Observation \ref{o:enough}, let us fix:

\begin{enumerate}
\item $\de_0$, $\de_*$, $\jj$ witnessing that $\de$ is optimized. 
\item $N \preceq M^I/\de$, $||N|| = \lambda$, and $p \in \ts(N)$ nonalgebraic, for which there exist:  
\begin{itemize}
\item[(a)] a $(\lambda, \theta, \sigma)$-presentation $\xm = ( \langle \vp_\alpha(x,a^*_\alpha) : \alpha < \lambda \rangle, \zm )$ of $p$ 
\item[(b)] an intrinsic coloring $G : \mcr_\xm \rightarrow \mu$, with range exactly $\mu$
\item[(c)] and a choice of lifting $\{ a^*_\alpha : \alpha < \lambda \} \rightarrow M^I$ 
\\ ~[i.e. 
let $F: M^I/\de \rightarrow M^I$ be a choice function, meaning that $F(a/\de) \in a/\de$, and 
in the rest of the proof we allow ourselves to write $a^*_\alpha$ instead of $F(a^*_\alpha)$ when this is clear 
from the context, so for $\alpha < \lambda$ and $t \in I$ the value ``$a^*_\alpha[t]$'' is well defined]
\end{itemize}
we have that for any $w = \clm(w) \subseteq \lambda$, 
\begin{enumerate}
\item[(d)] for any $t \in I$,  
\[ \mathfrak{C}_T \rstr \{ a^*_\alpha[t] : \alpha \in w \} \preceq \mathfrak{C}_T. \]
\item[(e)] for each finite sequence $\langle \alpha_{i_0}, \dots, \alpha_{i_{k-1}} \rangle$ of 
elements of $w$ and formula $\vp(x,\bar{y})$ of $T$ with $\ell(\bar{y}) = k$, there is $\beta \in w$ 
s.t. for \emph{all} $t \in I$,
\end{enumerate}
\[ M  \models (\exists x)\vp(x, a^*_{\alpha_{i_0}}[t], \dots, a^*_{\alpha_{i_{k-1}}}[t]) \implies 
\vp(a^*_\beta[t], a^*_{\alpha_{i_0}}[t], \dots, a^*_{\alpha_{i_{k-1}}}[t]). \]

\item Recall that the properties ensured by the presentation $\xm$ include:  
\begin{enumerate}
\item $\langle a^*_\alpha : \alpha < \lambda \rangle$ is an enumeration of $N = N^{eq}$ (note that the next few conditions put some 
restrictions on this enumeration). We write $A_\alpha$ for the set $\{ a_\beta : \beta < \alpha \}$.
\item $\bar{\vp} = \langle \vp_\alpha(x,a^*_\alpha) : \alpha < \lambda \rangle$, corresponding enumeration of the type $p$.
\item $\zm$ is an algebra on $\lambda$ with $\leq \theta$ functions such that for each $u \in \lambda$, 
if $|u| < \theta$ then $|\clm(u)| < \theta$, and if $|u| < \sigma$ then $|\clm(u)| < \sigma$. 
\item $\{ \alpha : \alpha \in \clm(\emptyset) \}$ is a cardinal $\leq |T|$. 
\item $N \rstr \{ a^*_\alpha : \alpha \in \clm(\emptyset) \} = M_* \preceq N$, where $||M_*|| \leq |T| < \sigma$ and $p$ does not fork over $M_*$. 
\item for each $u \in [\lambda]^{<\sigma}$, $N \rstr \{ a^*_\alpha : \alpha \in \clm(u) \}$ is the domain of an elementary submodel $N_u \preceq N$, 
and if in addition $u = \clm(u)$ then the partial type $\{ \vp_\alpha(x,a^*_\alpha) : \alpha \in u \}$ implies the elementary diagram of $N_u$ 
(and is an element of $\ts(N_u)$ which does not fork over $M_*$).  
\item for each $u \in [\lambda]^{<\sigma}$ and $\beta \leq \alpha \in \clm(u)$, we have that $tp(a^*_\alpha, A_\beta \cup M_*)$ does not fork over 
$\{ a^*_\gamma : \gamma \in \clm(u) \cap \beta \} \cup M_*$. 
\end{enumerate}
\item $\Omega = [\lambda]^{<\sigma}$.  
\end{enumerate}
 
\noindent With the stage set, our first task is to build a continuous sequence $\langle \mb_u : u \in \Omega \rangle$ of elements of $\ba$ which 
form a possibility pattern corresponding to the type $p$. 
Towards this, we shall describe $N$ by a type in $\lambda$ variables in the natural way. Let $\langle x_\alpha : \alpha < \lambda \rangle$ 
be a sequence of singleton variables (possibly they will be filled by imaginaries) and as before, for $v$ a sequence of elements of $\lambda$, 
let $\bar{x}_v$ denote $\langle x_\alpha : \alpha \in v \rangle$. [If $v$ was defined to be a set, interpret $\bar{x}_v$ 
by considering $v$ as a sequence listing its elements in increasing order. 
The main set $\Gamma_1$ of formulas, in the language $\ml$ of $T$, is closed under permuting the variables.] 
Define 
\begin{equation}
\Gamma_1 = \{ \psi(\bar{x}_v) : v \in {^{\omega >} \lambda}, \mbox{ $\psi$ an $\ml$-formula in $|v|$ free variables } \}.
\end{equation} 
For each finite $u \subseteq \lambda$, define
\begin{equation}
\label{phi-u}
\vp_u = \vp_u(x,\bar{x}_u) = \bigwedge_{\alpha \in u} \vp_\alpha(x,x_\alpha).
\end{equation}
These are collected in the set 
\begin{equation} 
\label{gamma2}
\Gamma_2 = \{ \vp_u : u \in [\lambda]^{<\aleph_0} \}.
\end{equation} 
Note that since $p \in \ts(N)$, $(\exists x) \vp_u(x,\bar{x}_u) \in \Gamma_1$ for each $\vp_u \in \Gamma_2$. 
Now we invoke \lost map. For each $\psi(\bar{x}_v) \in \Gamma_1$, define 
\begin{equation} \label{a-notation}
\ma_{\psi(\bar{x}_v)} = \jj(A_{\psi(\bar{x}_v)}) \mbox{ where } 
A_{\psi(\bar{x}_v)} = \{ t \in I : M \models \psi[\bar{a}_v[t]] \}.
\end{equation}
It will be useful to name the element of $\ba$ recording that ``$M_*$ appears correctly'': 
\begin{equation} 
\label{m-appears-correctly}
\ma_{\clm(\emptyset)} := \bigcap \{ \ma_{\psi(\bar{x}_v)} : v \in [\clm(\emptyset)]^{<\omega} ~\land~  {\psi(\bar{x}_v)} \in \Gamma_1 ~\land ~\models \psi[\bar{a}_v] \} \in \de_*. 
\end{equation}
Note that in equation (\ref{m-appears-correctly}) we make essential use of $\sigma > |T|$. 
Likewise, for each $\vp_u \in \Gamma_2$, define
\begin{equation}
\mb_{\vp_u} = \jj( B_{\vp_u} ) \mbox{ where } 
B_{\vp_u} = \{ t \in I : M \models (\exists x) \vp_u(x, \bar{a}_u[t]) \}.
\end{equation}
By \lost theorem, each $\mb_{\vp_u(x,\bar{x}_u)}$ belongs to $\de_*$ and a fortiori to $\ba^+$. 
As $\de_*$ is $\sigma$-complete, we may define $\mb_u$ for all $u \in \Omega$ 
by setting $\mb_u = \mb_{\vp_u(x,\bar{x}_u)}$ when $u$ is finite and 
$\mb_u = \bigcap \{ \mb_v : v \in [u]^{<\aleph_0} \}$ when $u$ is infinite. Then   
\begin{equation}
\bar{\mb} = \langle \mb_u : u \in \Omega \rangle 
\end{equation}
is a continuous sequence of elements of $\de_*$ in the sense of \ref{d:cont} above, 
and is a possibility pattern \ref{d:poss} (for the formulas in the sequence $\bar{\vp}$) by \lost theorem. So by Theorem \ref{t:separation}, 
showing $\bar{\mb}$ has a multiplicative refinement $\langle \mb^\prime_u : u \in [\lambda]^{<\sigma} \rangle$ 
in $\de_*$ will suffice to realize the type $p$. 

Notice at this point that for any $u \subseteq \lambda$, 
\begin{equation} 
\label{bs-and-as}
\mb_{\clm(u)} \leq \ma_{\clm(\emptyset)}. 
\end{equation}
This follows from condition 3.(f) from the beginning of the proof: 
``for each $u \in [\lambda]^{<\sigma}$, $N \rstr \{ a^*_\alpha : \alpha \in \clm(u) \}$ is the domain of an elementary submodel $N_u \preceq N$, 
and if in addition $u = \clm(u)$ then the partial type $\{ \vp_\alpha(x,a^*_\alpha) : \alpha \in u \}$ implies the elementary diagram of $N_u$.''

Our next task is to define a family of supporting sequences for $\bar{\mb}$, recalling Definition \ref{d:support}.  
We will set the stage by defining several progressively more refined families 
\[ \mcf \subseteq \fin_{\mu,\theta}(2^\lambda) \times [\lambda]^{<\theta}. \]
In each case, $\leq_\mcf$ is the natural partial order on elements of $\mcf$ given by 
\[ (f,w) \leq_\mcf (f^\prime, w^\prime) \mbox{ when } f \subseteq f^\prime, w \subseteq w^\prime \]
where recall $f \subseteq f^\prime$ implies $\mx_{f^\prime} \leq \mx_f$, as befits having more information. 
In each case, the family $\mcf$ will be defined as $\bigcup \{ \mcf_u : u \in \Omega \}$. 

For each $u \in \Omega$, we define $\mcf^0_u$ to be the set of pairs $(f, w)$ such that 
\begin{enumerate}
\item[(i)] $f \in \fin_{\mu, \theta}(2^\lambda)$, $w \in [\lambda]^{<\theta}$.  
\item[(ii)] $u \subseteq w = \clm(w)$, so $w$ is closed and contains the closure of $u$. 
\end{enumerate}
For each $u \in \Omega$ we define $\mcf^1_u \subseteq \mcf^0_u$ to be the set of pairs $(f, w)$ which are, 
in addition, decisive ``on $w$'': 
\begin{enumerate}
\item[(iii)] $\mx_f \leq \mb_{\vp_v(x,\overline{x}_v)}$ or
$\mx_f \leq 1 - \mb_{\vp_v(x,\overline{x}_v)}$ 
when $v \in [w]^{<\aleph_0}$ and ${\vp_v \in \Gamma_2}$. 
\item[(iv)] $\mx_f \leq \ma_{\psi(\overline{x}_v)}$ or
$\mx_f \leq 1 - \ma_{\psi(\overline{x}_v)} $
when $v \in [w]^{<\aleph_0}$ and $\psi(\overline{x}_v) \in \Gamma_1$. 
\end{enumerate}

The family $\mcf^1_u$ is dense in $\mcf^0_u$, that is, for any $(f,w) \in \mcf^0_u$ there is $(f^\prime, w^\prime) \in \mcf^1_u$ with $(f,w) \leq_{\mcf^0_u} (f^\prime, w^\prime)$. This is because the generators $\langle \mx_f : f \in \fin_{\mu,\theta}(2^\lambda) \rangle$ are dense in $\ba$. 
Notice also that both families are closed under limits which are not too large, i.e. 
if $\alpha < \theta$ is a limit ordinal and $\langle (f_\beta, w_\beta) : \beta < \alpha \rangle$ is 
a strictly increasing sequence of elements of $\mcf^1_u$, 
then $(\bigcup_{\beta<\alpha} f_\beta, \bigcup_{\beta<\alpha} w_\beta) \in \mcf^1_u$. To prove this we need to check 
that the limit $\bigcup_{\beta < \alpha} w_\beta$ remains closed, which is true because each $w_\beta$ is closed. 

Following an idea from \cite{MiSh:1009}, we now settle collisions.
For $\alpha, \beta \in \lambda$, write 
\begin{equation}
 A_{a_\alpha = a_\beta} = \{ t \in I : M \models a_\alpha[t] = a_\beta[t] \} \mbox{ and }\ma_{a_\alpha = a_\beta} = \jj(A_{a_\alpha = a_\beta}). 
\end{equation}
Clearly for any $f \in \fin_{\mu,\theta}(2^\lambda)$ and $\alpha < \lambda$,
$\mx_f \leq \ma_{a_\alpha = a_\alpha}$. 
For each $u \in \Omega$, let $\mcf^2_u \subseteq \mcf^1_u$ be the set of pairs $(f,w)$ such that, in addition, 
\begin{enumerate}
\item[(v)] for each $\alpha \in w$, $f$ decides equality for $\alpha$, meaning that there is $\beta \leq \alpha$, $\beta \in w$  
such that: $\mx_f \leq \ma_{a_\alpha = a_\beta}$
and for no $f^\prime \supseteq f$ and $\gamma < \beta$ (not necessarily from $w$) do we have 
\[ \mx_{f^\prime} \leq \ma_{a_\alpha = a_\gamma}. \] 
\end{enumerate}
To prove that $\mcf^2_u$ is dense in $\mcf^1_u$, it suffices to show
that for each $(f,w) \in \mcf^1_u$ and $\alpha \in w$ there is $(f_\alpha, w_\alpha)$ 
with $(f,w) \leq_{\mcf^1_u} (f_\alpha, w_\alpha)$ in which this condition is met for $\alpha$. 
Suppose we are given $(f,w) \in \mcf^1_u$ and 
$\alpha \in w$.  Let $\beta_0 = \alpha$, $f_0 = f$. Arriving to $i$, $\mx_f \leq \ma_{a_\alpha = a_{\beta_i}}$ and if 
in addition $\mx_f \cap \ma_{a_\alpha = a_\beta} = 0$ 
for all $\beta < \beta_i$, then the condition 
is satisfied for $\alpha$. Otherwise, let $\beta_{i+1} < \beta_i$ be a counterexample 
and as the generators are dense, we may choose 
$f_{i+1} \supseteq f_i$ so that $\mx_{f_{i+1}} \leq \ma_{a_\alpha = a_{\beta_{i+1}}}$. Since the ordinals are well ordered, 
this process stops at some finite stage $n = n(\alpha)$.  Let $w_n$ be the closure of $w \cup \{ \beta_n \}$. 
Fix $(f_\alpha, w_\alpha) \in \mcf^1_u$ such that $(f_n, w_n) \leq_{\mcf^1_u} (f_\alpha, w_\alpha)$. 
Then $(f_\alpha, w_\alpha)$ is as required. 
As $\mcf^1_u$ is closed under limits of cofinality less than $\theta$, we may 
build the desired $(f_*,w_*) \in \mcf^2_u$ as a limit of the elements $(f_\alpha, w_\alpha)$, 
indexing suitably to handle the new $\beta$s added along the way.  

Note that whenever $(f,w) \in \mcf^2_u$ and $\alpha \in w$ we may define 
\begin{equation}
\rho_\alpha(f) = \min \{ \beta \leq \alpha : \mx_f \leq \ma_{a_\alpha = a_\beta} \mbox{ and for all $\gamma < \beta$}~ \mx_{f} \cap \ma_{a_\alpha = a_\gamma} = 0 \}. 
\end{equation} 
Moreover, $\rho_\alpha(f) \in w$ and also $\rho_\alpha(f^\prime) = \rho_\alpha(f)$ for any $f \subseteq f^\prime \in \fin_{\mu, \theta}(2^\lambda)$, 
so this value is robustly defined.\footnote{In some sense we have found a small region where the elements $a_{\alpha}, a_{\rho_\alpha(f)}$ move `in lockstep'. 
Note that we are simply tying them to each other, not to any particular values in the monster model.}

We now record how elements of $\mcf^2_u$ naturally induce types. The 
key families $\mcf^3_u$, $\mcf^4_u \subseteq \mcf^2_u$ will be defined in terms of conditions on these types. 
In what follows $q(\dots)$ is a type of parameters, and 
$r(x,\dots)$ may or may not correspond to a fragment of $p$; the notation is meant to invoke \ref{d:es3}.
For each $u \in \Omega$ and each $(f,w) \in \mcf^2_u$, 
\begin{equation}
\label{d:q}
\mbox{ define $\qfw$ to be the type in the variables $\bar{x}_w$ given by: }
\end{equation}
\begin{align*}
\qfw(\bar{x}_w) = ~&\{ \psi(\overline{x}_v) : \mx_f \leq \ma_{\psi(\overline{x}_v)}, 
~\psi(\overline{x}_v) \in \Gamma_1 \} \\
& \cup 
\{ \neg \psi(\overline{x}_v) : \mx_f \leq 1 - \ma_{\psi(\overline{x}_v)}, 
~\psi(\overline{x}_v) \in \Gamma_1 \}.
\end{align*}
Notice that if $\mx_f \leq \ma_{\clm(\emptyset)}$, then the restriction of $\qfw$ to the variables 
$\bar{x}_{\clm(\emptyset)} = \bar{x}_{\delta_{\xm}}$ is realized by $\langle a^*_\alpha : \alpha < \delta_{\xm} \rangle$. 
Next, for each $u \in \Omega$ and each $(f,w) \in \mcf^2_u$, 
\begin{equation} 
\label{d:r}
\mbox{ choose $\rfw = \rfw(x,\bar{x}_w) \supseteq \qfw$ to be a complete type such that: }
\end{equation}
\begin{itemize}
\item[(a)] if $\overline{b}_w$ realizes $q$ in $\mathfrak{C}$ 
then $r(x,\overline{b}_w)$ dnf over $\{ b_\alpha : \alpha < \clm(\emptyset) \}$. 
\item[(b)] if consistent, choose $r$ so that in addition   
\\ $ r \supseteq 
 \{ \vp_\alpha(x,x_\alpha) : \alpha \in u \}\cup\{ \vp_\alpha(x,x_\alpha) : \alpha \in \clm(\emptyset) \}. $ 
\end{itemize}
The choice of each $\rfw$ will be fixed for the rest of the proof. Note in particular that if 
$\bar{b}_w$ satisfies $\alpha < \clm(\emptyset) \implies b_\alpha = a^*_\alpha$  then 
$r(x,\overline{b}_w)$ dnf over $M_*$. 

\br
\noindent Let $\mcf^3_u \subseteq \mcf^2_u$ be the set of $(f,w)$ such that in addition: 
\begin{enumerate}
\item[(vi)] whenever $\bar{b}_w$ is a sequence of elements of the monster model realizing $\qfw$, 
then for any $v \subseteq w$, $v = \clm(v)$ we have that $\{ b_\alpha : \alpha \in v \}$ 
is the domain of an elementary submodel.
\end{enumerate}
\br

\noindent Let us prove that $\mcf^3_u = \mcf^2_u$ for all $u \in \Omega$. 
Suppose for a contradiction that (vi) fails for some $\qfw$ and $\bar{b}_w$. Then there is a finite $v \subseteq w$ and a 
formula $\vp(x,\bar{y})$ of $T$ with $\ell(\bar{y}) = |v|$ witnessing the failure, i.e. 
$\mathfrak{C} \models \exists x \vp(x, \bar{b}_v)$ but there does not exist $\gamma \in w$ such that 
$\mathfrak{C} \models \vp(b_\gamma, \bar{b}_v)$. Since $\qfw$ is a complete type, $\psi(\bar{x}_v) = \exists x \vp(x, \bar{b}_v) \in \qfw$. 
By definition of $\qfw$, it must be that
\[ \mx_f \leq \ma_{\psi(\bar{x}_v)} \mbox{ i.e.  } \mx_f \leq \jj( A_{\psi(\bar{x}_v)} = \{ t \in I : M \models \exists x \vp(x, \bar{a}^*_v[t]) \} ) . \] 
By assumption (2)(e) at the beginning of the proof, there exists $\beta \in w$ such that for all $t \in I$,  
\begin{equation} \label{eq15} M  \models (\exists x)\vp(x, \bar{a}^*_v[t]) \implies \vp(a^*_\beta[t], \bar{a}^*_v[t]). 
\end{equation} 
We will show $\mx_f \leq \ma_{\vp(x_\beta, \bar{x}_v)}$. If not, 
$\ma_{\psi(\bar{x}_v)} \cap ( 1- \ma_{\vp(x_\beta, \bar{x}_v)}) > \mx_f > 0$, thus 
\[ A_{\psi(\bar{x}_v)} \setminus A_{\vp(x_\beta, \bar{x}_v)} \neq \emptyset. \]
Let $t$ be any element of this supposedly nonempty set. Then 
$M  \models (\exists x)\vp(x, \bar{a}^*_v[t])$ but it is not the case that $M \models \vp(a^*_\beta[t], \bar{a}^*_v[t])$, 
contradicting (\ref{eq15}). This contradiction proves that $\mx_f \leq \ma_{\vp(x_\beta, \bar{x}_v)}$. 
So by the definition of $\qfw$, it must be that $\vp(x_\beta, \bar{x}_v) \in \qfw$. This contradiction proves that (vi) 
will always hold, i.e. Lemma \ref{c:lifting} has arranged that $\mcf^2_u = \mcf^3_u$ for each $u \in \Omega$. 

There is one more family to define, $\mcf^4_u$. 
First, let us record that the ``if consistent'' clause in the defintion of $\rfw$, (\ref{d:r})(b), 
is often activated. 

\begin{obs} \label{r-ok}
Suppose $u \in \Omega$, $(f,w) \in \mcf^3_u$, and $\mx_f \leq \mb_{\clm(u)}$. Then $\rfw$ will always contain 
$\{ \vp_\alpha(x,x_\alpha) : \alpha \in u \}\cup\{ \vp_\alpha(x,x_\alpha) : \alpha \in \clm(\emptyset) \}$. 
\end{obs}

\noindent\emph{Proof.}  
In other words, we will show the ``if consistent'' from the definition of $\rfw$ is consistent. 
Denote by $N_u$ the elementary submodel $\mathfrak{C} \rstr \{ a^*_\alpha : \alpha \in \clm(u) \}$. 
(By our conditions on the algebra this is an elementary submodel, which includes $M_*$.) 
The hypothesis on $\mx_f$ means that $q \rstr \bar{x}_{\clm(u)}$ is realized by $N_u$, 
under an appropriate enumeration. Since nonforking is invariant under automorphism and all the types 
in question are complete, suppose without loss of generality that we are given $\bar{b}_w$ realizing 
$\qfw$ such that $\bar{b}_{\clm(u)}$ enumerates $\dom(M)$, and $\alpha < \clm(\emptyset)$ 
implies $b_\alpha = a^*_\alpha$.  By the definition of presentation, $p \rstr N_u$ is a type which 
includes $\{ \vp_\alpha(x,a^*_\alpha) : \alpha \in u \} \cup \{ \vp_\alpha(x,a^*_\alpha) : \alpha < \clm(\emptyset) \}$ 
and does not fork over $M_*$. We may choose $r_*(x,\bar{b}_w)$ to be any nonforking extension of $p \rstr N_u$ 
to the elementary submodel $\mathfrak{C} \rstr \bar{b}_w$. Let $r(x,\bar{x}_w)$ be the translation of 
$r_*$ to a type over the empty set in the variables $x,\bar{x}_w$. This shows that ``if consistent'' 
indeed is, so we may assume $\rfw$ has the stated properties, though it need not have arisen in this way. 
\hfill{ \emph{Observation \ref{r-ok}}} $\qed$

\br
To motivate our remaining step in the construction of a support,  
consider for a moment a sequence of four-tuples 
induced by the same $f$, say 
\[ \langle (v_t, w_t, q_{f_t, w_t}, r_{f_t, w_t}) : t < t_* < \sigma \rangle \] 
where $(v_t, w_t, q_{f_t, w_t}, r_{f_t, w_t})$
arises from $(f_t, w_t) \in \mcf^3_{u_t}$ and $v_t \subseteq u_t$, and  
there is a single $f$ with  $f_t \subseteq f$ for all $t$. A priori, this four-tuple need not be from $\mcr_{\xm}$ (e.g. if 
$\mx_f \cap \ma_{\clm(\emptyset)} = 0$) nonetheless we may begin to analyze its properties. 
By definition of $\qfw$, existence of such an $f$ means 
the union $\bigcup_{t < t_*} q_{f_t, w_t}$ will always be a partial type in the variables $\bar{x}_w$, 
where $w = \bigcup \{ w_t : t < t_* \}$. Let $\bar{b}^*_w$ be any sequence realizing this type. Then 
the types $r_{f_t, w_t}(x,\bar{b}^*_{w_t})$ may be explicitly contradictory as their definition allowed for arbitrary choices. 
Going forwards, our strategy will be to handle 
the issue of explicit inconsistency in the $r$'s with the construction of $\bar{\mb}^\prime$; before that, 
we ensure the necessary nonforking with the following definition. 

\br
Recall $D$-rank from \ref{d:rank} above. In the following, we do not require that $D$-rank is definable, only that 
the value is constant in the sense described. 

Let $\mcf^4_u$ be the set of pairs $(f,w) \in \mcf^3_u$ such that in addition:
\begin{enumerate}
\item[(vii)] for every $(f^\prime, w^\prime) \in \mcf^3_u$ with $(f,w) \leq_{\mcf^0_u} (f^\prime, w^\prime)$, 
and every sequence $\langle b_\alpha : \alpha \in w^\prime \rangle$ of elements of the monster model 
which realizes $q_{f^\prime, w^\prime}(\overline{x}_{w^\prime})$, and every $\alpha \in w$,  
\begin{itemize}
\item $\tp(b_\alpha, \{ b_\gamma : \gamma \in w^\prime \cap \alpha \})$ dnf over $\{ b_\gamma : \gamma \in w \cap \alpha \}$. 
\item 
for any formula $\vp$ and $k<\omega$,
\[ D(\tp(b_\alpha, \{ b_\gamma : \gamma \in w^\prime \cap \alpha \} ), \vp, k) = 
D(\tp(b_\alpha, \{ b_\gamma : \gamma \in w \cap \alpha \} ), \vp, k). \]
\end{itemize}
\end{enumerate}

\begin{claim} \label{3-dense} 
As $T$ is simple and $\theta$ is regular, 
$\mcf^4_u$ is $\leq_{\mcf^0_u}$-dense in $(\mcf^1_u, \leq_{\mcf^1_u})$. 
\end{claim}

\noindent\emph{Proof.}  
It suffices to prove it is dense in $\mcf^3_u$. 
Suppose for a contradiction that $\mcf^4_u$ is not dense. By induction on $\alpha < \theta$ 
we choose $b_\alpha$, $q_\alpha$, $f_\alpha$, $w_\alpha$ such that:
\begin{itemize}
\item $(f_\alpha, w_\alpha) \in \mcf^3_u$
\item $\beta < \alpha$ implies $(f,w) \leq_{\mcf^1_u} (f_\beta, w_\beta) \leq_{\mcf^1_u} (f_\alpha, w_\alpha)$ 
\item if $\alpha$ is a limit then $(f_\alpha, w_\alpha) = (\bigcup_{\beta < \alpha} f_\beta, \bigcup_{\beta < \alpha} w_\beta)$ 
\item $q_\alpha = q_{f_\alpha, w_\alpha}$, so $\beta < \alpha \implies q_\beta \subseteq q_\alpha$
\item $\langle b_\beta : \beta < \alpha \rangle$ realizes $q_\alpha\rstr \{ x_\gamma : \gamma \in w_\alpha \cap \alpha \}$ 
\item if $\alpha = \beta + 1$ then for some $\mbi \in w_\beta$ 

\begin{enumerate}
\item[$(\alpha)$]
$\tp(b_\mbi, \{ b_\gamma : \gamma \in w_\alpha \cap \mbi \}) \mbox{ forks over } \{ b_\gamma : \gamma \in w_\beta \cap \mbi \}$
\item[$(\beta)$] for some $\vp \in \ml(\tau_T)$ and finite $k$, 
\[ D(\tp(b_\mbi, \{ b_\gamma : \gamma \in w_\alpha \cap \mbi \} ), \vp, k) <
D(\tp(b_\mbi, \{ b_\gamma : \gamma \in w_\beta \cap \mbi \} ), \vp, k). \]
\end{enumerate}
\end{itemize}
As $\theta = \cf(\theta) > |T|$, by Fodor's lemma this contradicts the assumption that $T$ is a simple theory. 
\hfill{\emph{Claim \ref{3-dense}}}$\qed$

We are ready to choose partitions supporting each $\mb_u$, $u \in \Omega$. The next claim states our requirements, and 
adds a coherence condition. 

\begin{claim} \label{d:good-choice} 
There exists 
\[ \overline{f} = \langle \overline{f}_u : u \in \Omega \rangle = \langle~ \langle (\fdp_{u,\zeta}, w_{u,\zeta}) : \zeta < \mu \rangle ~:~ u \in \Omega \rangle \] 
which is a \emph{good choice of partitions} for $\bar{\mb}$, where this means 
\begin{enumerate}
\item for each $u \in \Omega$, $\bar{f}_u = \langle (\fdp_{u,\zeta}, w_{u,\zeta}) : \zeta < \mu \rangle$ is 
a sequence of elements of $\mcf^4_u$ such that $\langle \mx_{\fdp_{u,\zeta}} : \zeta < \mu \rangle$ is a maximal antichain of $\ba$ 
and for each $\zeta < \mu$, either $\mx_{\fdp_{u,\zeta}}  \leq \mb_u$ or $\mx_{\fdp_{u,\zeta}} \leq 1 - \mb_u$. 
\item \emph{(coherence)} 
Writing $\vv = \bigcup \{ \dom(\fdp_{u,\zeta}) : u \in \Omega, \zeta < \mu \}$,
\\ if $v_0 \in [\vv]^{<\theta}$  
and $u_1 \in \Omega$, then for some $u_*$ we have: $u_1 \subseteq u_* \in \Omega$ and 
$\zeta < \mu \implies v_0 \subseteq \dom(\fdp_{u_*, \zeta})$ and $\bar{f}_{u_*}$ refines $\bar{f}_{u_1}$. 
\end{enumerate}
\end{claim}

\noindent\emph{Proof.} 
$\ba$ satisfies the $\mu^+$-c.c. and has maximal antichains of cardinality $\mu$, 
so for any $u \in \Omega$ we may choose $ \overline{f}^x_u = \{ (\fdp_{u,\epsilon}, w_{u,\epsilon}) : \epsilon < \mu \} \subseteq \mcf^4_u$, 
such that:
\begin{enumerate}
\item[] $\langle \mx_{\fdp_{u,\zeta}} : \zeta < \mu \rangle$ is a maximal antichain of $\ba$, so $\epsilon \neq \zeta < \mu \implies f_{u,\epsilon} \neq f_{u,\zeta}$.
\end{enumerate}
[We build such a partition by induction on $\zeta$, using the density of $\mcf^4_u$. In doing so 
we may assume, without loss of generality, that $0 \in \dom(f)$ for each $f$ used in this partition; 
then the partition will have size at least $\mu$, and since $\ba$ has the $\mu^+$-c.c., the construction will stop at some ordinal $<\mu^+$. 
Renumbering, we may assume the sequence is indexed by $\zeta < \mu$.]

Say that $\overline{f}_{u_i}$ \emph{refines} $\overline{f}_{u_j}$ if for each $\epsilon < \mu$, 
$\fdp_{u_i,\epsilon}$ extends $\fdp_{u_k,\zeta}$ for some $\zeta < \mu$. 
To ensure coherence across the family, \ref{d:suitable} ensures that we may enumerate 
$\Omega \times [\lambda]^{<\theta}$ as $\langle (u_i, v_i) : i < \lambda \rangle$. 
Build $\overline{f}_{u_i}$ by induction on $i < \lambda$ as follows. 
Arriving to $i$, 
\begin{enumerate}
\item[(a)] if $(\exists j < i)\left((u_i \subseteq u_j) \land (\zeta < \mu \implies v_i \subseteq \dom(\fdp_{u_j, \zeta}))\right)$ 
\\then let $j(i)$ be the least such $j$ and let $\overline{f}_{u_i}$ be a common refinement of 
$\overline{f}^x_u$ and $\overline{f}_{u_j}$. 

\item[(b)] if there is no such $j$, choose $\overline{f}_{u_i}$ such that it refines $\overline{f}^x_{u_i}$ and  
$\overline{f}_{u_k}$ whenever $k<i$ and $u_k \subseteq u_i$ \underline{and} $\zeta < \mu \implies \dom(\fdp_{u_i,\zeta}) \supseteq v_i$. 
We can do this because there are $\leq 2^{|u_i|} < \theta$ such $k$, recalling that $\sigma \leq \theta$ and $\sigma$ is 
(by our hypotheses in this section) compact, thus strongly inaccessible. 
\end{enumerate}
\noindent 
For each $u \in \Omega$, let 
$\overline{f}_u = \langle (\fdp_{u,\zeta}, w^\prime_{u,\zeta}) : \zeta < \mu \rangle$
be the resulting family.  
This completes the construction.  
\hfill{ \emph{Claim \ref{d:good-choice}}} $\qed$

\br \noindent For the remainder of the proof, fix the support just built as well as 
\begin{equation} \label{vv-eqn}
 \vv \subseteq 2^\lambda, |V| = \lambda \mbox{ such that } 
\bigcup \{ \dom(f_{u,\zeta}) : u \in \Omega, \zeta < \mu \} \subseteq \vv.
\end{equation} 

\br

\noindent Having built a support for the sequence $\bar{\mb}$, our next task is to define the proposed multiplicative 
refinement $\bar{\mb}^\prime$. Towards this, let us take stock. 
For each $u \in \Omega$ 
and $\zeta < \mu$, we may henceforth unambiguously write 
\begin{equation} 
q_{u,\zeta} \mbox{ for } q_{f_{u,\zeta}, w_{u,\zeta}}, ~r_{u,\zeta} \mbox{ for } r_{f_{u,\zeta}, w_{u,\zeta}}. 
\end{equation}

\begin{claim} \label{are-in-r}
For each $u \in \Omega$ and $\zeta < \mu$, 
\[ \mbox{ if }\mx_{f_{u,\zeta}} \leq \mb_{\clm(u)} \mbox{ then } ( u , w_{u,\zeta}, q_{u,\zeta}, r_{u,\zeta}) \in \mcr_{\xm}. \]
\end{claim} 

\noindent\emph{Proof}. 
We check Definition \ref{d:es3}. 
For \ref{d:es3}.1-2, $u \in [\lambda]^{<\sigma}$, $w \in [\lambda]^{<\theta}$, and $u \subseteq \clm(u) \subseteq w = \clm(w)$
by the definition of $\mcf^0$.
Towards \ref{d:es3}.3, $\qfw$ is always a complete type in the variables $\bar{x}_w$, and if 
$\mx_f \leq \mb_{\clm(u)}$ then $ \mx_f \leq \ma_{\clm(\emptyset)}$ by equation (\ref{bs-and-as}).
So \ref{d:es3}.3(a), 
holds by definition of $\qfw$ (really, of $\Gamma_1$) and 
\ref{d:es3}.3(b) 
holds by the remark after equation (\ref{gamma2}). 
\ref{d:es3}.4+5(a) follow from Observation \ref{r-ok}. \ref{d:es3}.5(b) is because $\rfw$ is a complete type (notice that the condition 
``if $w^\prime \subseteq w$ is $\mlx$-closed then $\mathfrak{C}_T \rstr \{ b^*_\alpha : \alpha \in w^\prime \} \preceq \mathfrak{C}_T$ and 
$\rfw(x,\overline{b}^*_{w^\prime})$ is a complete type over this elementary submodel'' does not ask for a given enumeration to have 
any closure properties, but simply that all formulas are decided). Finally \ref{d:es3}.5(c) is by the fact that $(f,w) \in \mcf^4_u$. 
\hfill{ \emph{Claim \ref{are-in-r}}} $\qed$

\br
The four-tuple from the statement of Claim \ref{are-in-r} 
is therefore in the domain of the function $G$ giving the intrinsic coloring for $\xm$, which was fixed 
as part of the presentation at the beginning of the proof. 
(We may trivially extend $G$ to all four-tuples arising from some $(f_{u,\zeta}, w_{u,\zeta})$ by setting $G$ to be $\infty$ 
if it is not otherwise defined.)
We will want to amalgamate certain such tuples later in the proof, but first we take advantage of the ultrapower setup to eliminate 
some extraneous noise. 
Define $E$ to be the equivalence relation on pairs $(u,\zeta) \in \Omega \times \mu$ given by: 
\begin{equation} \label{defn-of-e}
(u_1,\zeta_1) E (u_2, \zeta_2) \mbox{ iff } 
\end{equation}
\begin{enumerate}
\item $(u_1, w_{u_1,\zeta})$ and $(u_2, w_{u_2,\zeta})$ satisfy:
\begin{enumerate}
\item[(i)] $\otp(u_1) = \otp(u_2)$, $\otp(w_1) = \otp(w_2)$
\item[(ii)] if $\gamma \in w_1 \cap w_2$ then $\otp(\gamma \cap w_1) = \otp(\gamma \cap w_2)$  
\item[(iii)] if $\gamma \in w_1 \cap w_2$ then $\gamma \in u_1$ iff $\gamma \in u_2$
\item[(iv)] the order preserving map from $w_1$ to $w_2$ carries $u_1$ to $u_2$.
\end{enumerate}

\item there is an order preserving function $h$ from
$\dom(f_1)$ onto $\dom(f_2)$, s.t.:
\begin{enumerate}
\item[(i)] (implied) $\otp(\dom(f_1)) = \otp(\dom(f_2))$
\item[(ii)] $\gamma \in \dom(f_1) \implies f_2(h(\gamma))=f_1(\gamma)$
\item[(iii)] $\gamma \in \dom(f_1) \cap \dom(f_2) \implies h(\gamma) = \gamma$.
\end{enumerate}

\item $\zeta_1 = \zeta_2$. 

\item $G (u_1, w_{u_1,\zeta_1}, q_{u_1,\zeta_1}, r_{u_1,\zeta_1}) = 
G (u_2, w_{u_2,\zeta_2}, q_{u_2,\zeta_2}, r_{u_2,\zeta_2})$, i.e the values are both defined and equal or both $\infty$.  
\end{enumerate}

\noindent 
\begin{claim}  \label{mu-classes}
$E$ from $(\ref{defn-of-e})$ is an equivalence relation with $\mu$ classes, 
which we will list as  
\[ \langle E_\epsilon : \epsilon < \mu \rangle.\] 
\end{claim}

\noindent\emph{Proof}. 
For counting purposes we may assume $\zeta < \mu$ is fixed. 

For the first condition, the equivalence class of $(u, w_{u,\zeta})$ is fixed if we determine the ordinals $\gamma = \otp(w_{u,\zeta})$, $\delta = \otp(u)$, 
and determine which function from $\delta$ into $\gamma$ gives $u$.  There are $\leq \theta \leq \mu$ choices of $\gamma$, 
$\leq \sigma \leq \mu$ choices of $\delta$, and for any given $\gamma <\theta$, $\delta <\sigma$, we have 
by $\ref{d:suitable}\ref{x:12}$ that $|{^{\delta}\gamma}| < \theta \leq \mu$. So by $\ref{d:suitable}\ref{x:10}$, the  
total count is bounded by $\mu \cdot \mu \cdot \mu = \mu$.

For the second condition, the equivalence class of $\fdp_{u,\zeta}$ is fixed if we
first determine the ordinal $\otp(\dom(\fdp_{u,\zeta}))$, call it $\beta$, and then determine which function from $\beta$ into $\mu$ determines 
the values of the function. Since $\otp(\dom(\fdp_{u,\zeta}))$ is an ordinal $<\theta$, the number of possible results 
is bounded by $\theta \cdot \mu^{<\theta} = \theta \cdot \mu = \mu$. 


This shows the number of equivalence classes is $\leq \mu \cdot \mu = \mu$, and by our choice of $G$ it is exactly $\mu$. 
\hfill{ \emph{Claim \ref{mu-classes}}} $\qed$

\br
\noindent Given $\langle E_\epsilon : \epsilon < \mu \rangle$ from Claim \ref{mu-classes}, 
fix also a choice of representatives by specifying some 
function
\begin{equation}
\label{the-function-h}
h : \mu \rightarrow \Omega \times [\lambda]^{<\theta} \times \mu \times \mu. 
\end{equation}
We ask that $h(\epsilon) = (u_{h(\epsilon)}, w_{h(\epsilon)}, \zeta_{h(\epsilon)}, \xi_{h(\epsilon)})$ satisfies: 
for some $(u, \zeta) \in E_\epsilon$, $u = u_{h(\epsilon)}$, $\zeta = \zeta_{h(\epsilon)}$, $w_{h(\epsilon)} = w_{u,\zeta}$, 
and $\xi_{h(\epsilon)} = G(u, w, q_{u,\zeta}, r_{u,\zeta})$ or $\infty$. 

\br
The crucial set for each $\alpha <\lambda$, $\epsilon < \mu$ will be
\begin{equation} \label{e:mcv}
\mcv_{\alpha,\epsilon} = \{ u : u \in \Omega \mbox{ and }
(u, \zeta_{h(\epsilon)}) \in E_\epsilon \mbox{ and } \mx_{\fdp_{u, \zeta_{h(\epsilon)}}} \leq \mb_{\clm( u )} \}.  
\end{equation}

\noindent Recall $\vv$ from (\ref{vv-eqn}) above. Let 
\begin{equation} \label{alpha-eqn}
\alpha_* < 2^\lambda \mbox{ be such that }\mcv \subseteq \alpha_*.
\end{equation} 
Let $\code_\lambda$ denote some fixed coding function from ${^{\omega>}\lambda}$ to $\lambda$, 
and let $\code_\mu$ denote some fixed coding function from ${^{\omega>}\mu}$ to $\mu$. 
Let $\langle X_0, \dots, X_5 \rangle$ be a partition of $\lambda$ into six sets of cardinality $\lambda$. 
For $\gamma < \lambda$ and $n<6$, let $\rho(\gamma, X_n)$ denote the image of $\gamma$ under a fixed one-to-one  
map of $\lambda$ into $X_n$.  
Let $\tv$ denote the truth value of an expression (either $0$ or $1$). 
\begin{equation}
\label{defining-f-star}
\mbox{For each $u \in \Omega$, $\zeta < \mu$ define $f^* = f^*_{u,\zeta}$ as follows.}
\end{equation} 
\begin{enumerate}
\item $\dom(f^*) \subseteq \alpha + \lambda$ is of cardinality $<\theta$, $\rn(f^*) \subseteq \mu$, and $f^*$ is determined by the remaining conditions. 
\br
\item if $\gamma \in \dom(\fdp_{u,\zeta})$ then 
\[ f^*(\alpha + \rho(\gamma, X_0)) = \code_\mu( ~\langle \fdp_{u,\zeta}(\gamma), \otp(\gamma \cap \dom(\fdp_{u,\zeta})) \rangle ~). \]
\item if $\gamma \in w_{u,\zeta}$, then 
\[ f^*(\alpha + \rho(\gamma, X_1)) = \code_\mu(\tv(\gamma \in u), \otp(\gamma \cap u), \otp(\gamma \cap w_{u,\zeta})). \]
\item if $\beta \neq \alpha$ are from $w_{u,\zeta}$ then 
\[ f^*(\alpha + \rho(\gamma, X_2)) = \tv( ~\rho_\alpha(f_{u,\zeta}) = \rho_\beta(f_{u,\zeta}) ~). \]
\item if $v \subseteq u$, $\bar{v}$ is this set listed in increasing order, and $\code_\lambda(~ \bar{v} ~) = \gamma$, then 
\[ f^*(\alpha + \rho(\gamma, X_3)) = \tv( ~\mx_{f_{u,\zeta}} \leq \mb_v~ ). \]
\item if $v \in {^{\omega >}(w_{u,\zeta})}$, $\vp$ is the $k$-th $\ml$-formula in $|v|$ free variables under some enumeration 
fixed in advance, and $\gamma = \code_\lambda( ~\langle v, k \rangle~ )$, then  
\[ f^*(\alpha + \rho(\gamma, X_4)) = \tv( \mx_{f_{u,\zeta}} \leq \ma_{\vp(\bar{x}_v} ) = \tv(~ \vp(\bar{x}_v) \in q_{u,\zeta} ~). \]
\item if $v \in {^{\omega >}(w_{u,\zeta})}$, $\vp$ is the $k$-th $\ml$-formula in $1+|v|$ free variables under some enumeration 
fixed in advance, and $\gamma = \code_\lambda( ~\langle v, k \rangle~ )$, then  
\[ f^*(\alpha + \rho(\gamma, X_5)) = \tv(~ \vp(x,\bar{x}_v) \in r_{u,\zeta} ~). \]
\end{enumerate} 
\br
\noindent This completes the definition (\ref{defining-f-star}). Fix also a new maximal antichain:  
\begin{equation}
\mbox{let $\bar{\mc} = \langle \mc_\epsilon : \epsilon < \mu \rangle$ be given by $\mc_{\epsilon} = \mx_{\{(\alpha + \lambda, \epsilon)\}} \}$.}
\end{equation}

\noindent We arrive at the definition of $\bar{\mb}^\prime$. For each $\alpha < \lambda$, let
\begin{equation} \label{mb-prime-b}
\mb^\prime_{\{\alpha\}} = \left(  \bigcup \{ \mc_\epsilon \cap \mx_{f^*_{u, \zeta_{h(\epsilon)}}} \cap
 \mx_{\fdp_{u, \zeta_{h(\epsilon)}}}  ~:~  \epsilon <\mu, ~ u \in \uu_{\alpha, \epsilon} \}  \right) \cap \mb_{\clm(\{\alpha\})}. 
\end{equation}
\noindent Why is (\ref{mb-prime-b}) nonzero? For each $\epsilon < \mu$ such that 
$\uu_{\alpha, \epsilon} \neq \emptyset$, and each $u \in \uu_{\alpha, \epsilon}$,
\[  \mc_\epsilon \cap \mx_{f^*_{u, \zeta_{h(\epsilon)}}} \cap
 \mx_{\fdp_{u, \zeta_{h(\epsilon)}}} \cap \mb_{\clm(u)} ~ > ~ 0. \]
since the domains of the functions corresponding to 
$\mx_{\fdp_{u, \zeta_{h(\epsilon)}}}$, $\mc_\epsilon$ and $\mx_{f^*_{u, \zeta_{h(\epsilon)}}}$ are 
mutually disjoint, and adding $\mb_{\clm(u)}$  
is allowed by the definition of $\uu_{\alpha,\epsilon}$ in (\ref{e:mcv}). 
By monotonicity, $\mb_{\clm(u)} \leq \mb_{\clm(\{\alpha\})}$ for any $u \in \mcv_{\alpha, \epsilon}$. 
This verifies that (\ref{mb-prime-b}) is nonzero.

For each $u \in \Omega \setminus \emptyset$, define
\begin{align} \label{mb-prime-s}
\mb^\prime_u = \bigcap \{ \mb^\prime_{\{ \alpha \}} : \alpha \in u \}.  
\end{align}
Let $\mb^\prime_{\emptyset} = 1_{\ba}$. This completes the definition of 
the sequence $\bar{\mb}^\prime$:
\begin{equation}
\label{defn-of-b-prime}
\bar{\mb}^\prime = \langle \mb^\prime_u : u \in \Omega \rangle. 
\end{equation}
By construction, $\bar{\mb}^\prime$ is multiplicative, and if $|u| = 1$ then $\mb^\prime_u \leq \mb_u$. 

An immediate consequence (as $\langle \mc_\epsilon: \epsilon < \mu \rangle$ is a maximal antichain) 
of this definition is that whenever $\mc \in {\ba^+_{\alpha_*}}$ 
and $0 < \mc \leq \mc_\epsilon \cap \mb^\prime_{\{\alpha\}}$, 
\begin{align}\label{e:61}
\bigcup \{ \mc \cap \mx_{f^*_{u,\zeta_{h(\epsilon)}}} \cap \mx_{\fdp_{u, \zeta_{h(\epsilon)}}} :  u \in \mcv_{\alpha,\epsilon} \} \neq 0 
\end{align} 
so in particular there is $u \in \mcv_{\alpha,\epsilon}$ such that 
\begin{equation}
\label{e:62}
\mc \cap \mx_{f^*_{u,\zeta_{h(\epsilon)}}} \cap \mx_{\fdp_{u,\zeta_{h(\epsilon)}}} > 0 ~~\mbox{\hspace{3mm} thus \hspace{3mm}}~~ 
\mc \cap \mx_{\fdp_{u,\zeta_{h(\epsilon)}}} \cap \mb_{\clm(u)} > 0  
\end{equation}
where again the conjunct ``$\cap ~\mb_{\clm(u)}$'' is by the definition of $\mcv_{\alpha, \epsilon}$ in equation (\ref{e:mcv}). 
We now work towards proving that:

\begin{claim} \label{mult-reft}
$\bar{\mb}^\prime$ is a multiplicative refinement of $\bar{\mb}$. 
\end{claim}

\noindent\emph{Proof}.
Suppose for a contradiction that there is some 
$u_* \in \Omega$ such that 
\begin{equation} \label{the-cont}
0 < \mc \leq \bigcap_{\alpha \in u_*} \mb^\prime_{\{\alpha\}} \setminus \mb_{u_*}. 
\end{equation}
By continuity of $\bar{\mb}$, we may assume that $u_*$ is finite. 
Let $f \in \fin_{\mu,\theta}(2^\lambda)$ be such that $\mx_{f} \leq \mc$, and as $\langle \mc_\epsilon : \epsilon < \mu \rangle$ 
is a maximal antichain, without loss of generality there is some $\epsilon < \mu$ 
such that $\mx_f \leq \mc_\epsilon$. 
Since $\mx_f \leq \mb^\prime_{\{\alpha\}}$ by construction, necessarily $\mcv_{\alpha,\epsilon} \neq \emptyset$ for each $\alpha \in {u_*}$. 
Enumerate ${u_*}$ as $\langle \alpha_t : t < t_* \rangle$. 
By induction on $t \leq t_*$, we will choose functions $f_t \in \fin_{\mu, \theta}(2^\lambda)$ and sets $u_{t} \in \mcv_{\alpha_t, \epsilon}$ such that 
for each $t$: 
\begin{enumerate}
\item[(i)] $f_t \supseteq f$, \emph{thus} for each $t$, $\mx_{f_t} \leq \bigcap \{ \mb^\prime_{\{\alpha\}} : \alpha \in {u_*} \} \cap \mc_\epsilon$. 
\item[(ii)] $t^\prime < t \implies f_{t^\prime} \subseteq f_t$
\item[(iii)] $f_{t} \supseteq f_{u_t, \zeta_{h(\epsilon)}} \cup f^*_{u_t, \zeta_{h(\epsilon)}}$
\item[(iv)] $\mx_{f_t} \leq \mb_{\clm(u_t)}$.
\end{enumerate}
Let $f_{-1} = f$ so that $f_0$ is the case ``$-1 + 1$.''
Arriving to $t+1$, condition (i) implies that $\mx_{f_t} \leq \mb^\prime_{\{\alpha_t\}} \cap \mc_\epsilon$.
So adding the latter two onto a conjunction will not affect whether or not we get $0$.
Apply (\ref{e:61})-(\ref{e:62}) to choose $u_{t} \in \mcv_{\alpha_t, \epsilon}$ 
such that
\[ \mx_{f_\ell} \cap \mx_{\fdp_{u_t, \zeta_{h(\epsilon)}}} \cap \mx_{f^*_{u_t, \zeta_{h(\epsilon)}}} \cap \mb_{\clm(u)} > 0. \]
Let $f_{t+1} = f_t \cup \fdp_{u_t, \zeta_{h(\epsilon)}} \cup f^*_{u_t, \zeta_{h(\epsilon)}}$.
This completes the induction. Note that (iv) will be satisfied by the definition (\ref{e:mcv}). 

Let $f_* = f_{t_*}$ be the function so constructed. 

Recalling the definition of $E_\epsilon$ in \ref{defn-of-e}, since 
$\epsilon$ is fixed all the $\zeta_{h(\epsilon)}$'s are the same, so going forward we will write 
 $\zeta$ for $\zeta_{h(\epsilon)}$. 
Then by Claim \ref{are-in-r} and the fact that $\mx_{f_*} \leq \mb_{\clm(u_t)}$ for each $\alpha_t \in {u_*}$, 
we have that for each $t < t_*$, 
\begin{equation} \label{the-lot}
\xr_t := (u_t, w_{u_t, \zeta}, q_t = q_{u_t, \zeta}, r_t = r_{u_t, \zeta}) \in \mcr_{\xm}. 
\end{equation}
\noindent Let $\bar{\xr} = \langle \xr_t : t < t_* \rangle$.  Moreover, 
the pairs $(u_t, \zeta)$ all belong to the $E_\epsilon$ class of our equivalence relation from (\ref{defn-of-e}). 
So equation (\ref{the-lot}) ensures that each $\xr_t$ is in the domain of $G$ and item 4 of the definition of $E$ 
ensures that the value of $G$ is fixed. That is, 
\begin{equation} \label{g-is-constant}
G \rstr \langle \xr_t : t < t_* \rangle \mbox{ is constant. }
\end{equation}
Thanks to the $f^*_{u_t,\zeta}$ from (\ref{defining-f-star}), 
we may now find a good instantiation of $\bar{\xr}$:   
\begin{subclaim} \label{the-seq}
Let $w = \bigcup_t w_t$. We can find $\langle b^*_\alpha : \alpha \in w \rangle$ realizing $\bigcup_t q_t$ such that 
$\bar{b}^*_w$ is a good instantiation for $\bar{\xr}$ and 
\begin{equation}
\label{no-witness} 
\mathfrak{C} \models \neg (\exists x) \bigwedge_{\alpha \in u_*} \vp_\alpha(x,b^*_\alpha). 
\end{equation}
\end{subclaim} 

\noindent\emph{Proof}. 
As the $q_t$ are all induced by the same $f_* = f_{t_*}$ the set $q = \bigcup_{t<t_*} q_t$ is consistent. 
First we will show it is consistent with $\neg \vp_{u_*}$, recalling (\ref{phi-u}). Suppose not, so let 
$\Sigma_0 = \{ \vp_0(\bar{x}_{v_0}), \cdots, \vp_k(\bar{x}_{v_k}) \} \subseteq q$ be finite such that 
$\Sigma = \Sigma_0 \cup \{ \neg \vp_{u_*} \}$ is inconsistent. Let $v = v_0 \cup \cdots \cup v_k$. 
Let $\theta$ be an arbitrary finite conjunction of formulas from $q \rstr \bar{x}_{\clm(\emptyset)}$. 
Recalling notation (\ref{a-notation}), 
$\mx_{f_*}$ witnesses that the set 
\[ A_\theta \cap A_{\vp_0(\bar{x}_{v_0})} \cap \cdots \cap A_{\vp_k(\bar{x}_{v_k})} \setminus B_{\vp_u} > 0. \]
Let $t$ be any element of this set. Then $\{ a^*_\alpha[t] : \alpha \in v \}$ realizes $\Sigma$, contradiction. 
This proves we may find some sequence $\bar{b}^*_w$ of elements of $\mathfrak{C}$, possibly imaginary, realizing $q$ such that 
$b^*_\alpha = a^*_\alpha$ for $\alpha \in \clm(\emptyset)$ and $\models \neg (\exists x) \bigwedge_{\alpha \in u_*} \vp_\alpha(x,b^*_\alpha)$. 
Finally, let us check that this sequence is a good instantiation, \ref{d:good-inst}. 
Conditions (1)-(2) of that definition we've just checked. (3) 
is ensured by the fact that 
\[ \mx_{f_{t_*}} \leq \mx_{f^*_{u_t, \zeta}} \]
for each $u_t$, recalling (\ref{defining-f-star})6+7. 
(4) was ensured by $\mcf^4$. 
(5) was ensured by $\mcf^3$, and already checked by membership in $\mcr_{\xm}$, \ref{d:es3}(5). 
\hfill{\emph{Subclaim \ref{the-seq}}} $\qed$

\br

Let $\bar{b}^*_w$ be given by Subclaim \ref{the-seq}. By equation (\ref{g-is-constant}), we may apply Definition \ref{d:es4}, ``$G$ is an intrinsic coloring 
of $\mcr_{\xm}$'', to conclude that 
$\{ \vp_{\alpha_t}(x,b^*_{\alpha_t}) : t < t_* \} = \{ \vp_\alpha(x, b^*_\alpha) : \alpha \in u_* \}$ is consistent. 
This contradicts the choice of $\bar{b}^*$, specifically equation (\ref{no-witness}).  
This contradiction shows that $\mc$ from equation (\ref{the-cont}) cannot exist, i.e.   
$\bar{\mb}^\prime$ must be a multiplicative refinement of $\bar{\mb}$. This proves the Claim. 
\hfill{ \emph{Claim \ref{mult-reft}}} $\qed$

\br

The last part of the argument is to verify that $\bar{\mb}$ has the Key Property \ref{kp} using $\vv$ from (\ref{vv-eqn}), 
$\alpha_*$ from (\ref{alpha-eqn}), $\Omega_* = \{ u \in [\lambda]^{<\sigma} : u = \clm(u) \}$, and $\bar{\mb}^\prime$. 
Suppose $u \in \Omega_*$ and $f \in \fin_{\mu,\theta}(\alpha_*)$ are 
such that $\mx_f \leq \mb_u$. 
We hope to show  
\begin{equation} \label{eq1}
\mx_f \cap \mb^\prime_u > 0.
\end{equation} 
As $\mx_f \leq \mb_u$, by choice of $u \in \Omega_*$, 
\begin{equation}
\mx_f \leq \mb_{\clm(u)}.
\end{equation}   
Since $\overline{f}_{u}$ is a partition, after possibly extending $f$ we have  
\begin{equation}
\label{eq2}
\mx_f \leq \mx_{f_{u, \zeta_*}} \mbox{ for some }\zeta_* < \mu.
\end{equation} 
Choose $\epsilon < \mu$ such that $(u, \zeta_*) \in E_\epsilon$. Then $h(\epsilon) = (u, w_{h(\epsilon)}, \zeta_{h(\epsilon)} = \zeta_*, 
\xi_{h(\epsilon)})$. For each $\alpha \in u$, checking the definition (\ref{e:mcv}) we have that 
\begin{equation} \label{eq4}
u \in \uu_{\alpha,\epsilon}.
\end{equation} 
As $\bar{\mb}^\prime$ is a multiplicative sequence, $\mb^\prime_u = \bigcap_{\alpha \in u} \mb^\prime_{\{\alpha\}}$. 
So to show (\ref{eq1}) it would suffice to show that 
\begin{equation} \label{eq3}
\mx_f \cap \bigcap_{\alpha \in u} \mb^\prime_{\{\alpha\}}.
\end{equation} 
Recalling the definitions (\ref{mb-prime-b}), and equation (\ref{eq4}), the expression
\begin{equation} \label{eq5} 
\mc_\epsilon \cap \mx_{f^*_{u, \zeta_{h(\epsilon)}}} \cap
 \mx_{\fdp_{u, \zeta_{h(\epsilon)}}}
\end{equation}
appears as a disjunct for each $\mb^\prime_{\{\alpha\}}$, and by (\ref{e:mcv}) has 
nonzero intersection with $\mb_{\clm(u)}$, so 
to show (\ref{eq3}) it would suffice to show that 
\[   \mc_\epsilon \cap \mx_{f^*_{u, \zeta_{h(\epsilon)}}} \cap
 \mx_{\fdp_{u, \zeta_{h(\epsilon)}}} \cap \mb_{\clm(u)} \cap \mx_f > 0. \]
This is true because on one hand, $\mx_f \leq  \mx_{\fdp_{u, \zeta_{h(\epsilon)}}} \cap ~\mb_{\clm(u)}$ 
by our above argument, and on the other hand, the support of $\mx_f$, $c_\epsilon$, and $\mx_{f^*_{u, \zeta_{h(\epsilon)}}}$ 
are pairwise disjoint and thus cannot cause inconsistency. This completes the proof that $\mx_f \cap \mb^\prime_u > 0$.
\br

As $\bar{\mb}$ has the $(\lambda, \mu, \theta, \sigma)$-Key Property, in our optimal ultrafilter $\de_*$ it 
has a multiplicative refinement. Thus by Theorem \ref{t:separation}, the original type $p$ is realized in the ultrapower $M^I/\de$. 
This completes the proof of {{Theorem \ref{t:saturation}}}. 
\end{proof}

\newpage
\section{The ultrapower characterization of simple theories} \label{s:main-theorems}

We state our main result in two different ways, the first to underline the structure of the proof.  

\begin{theorem} \label{t:summary}
Assume $(\lambda, \mu, \theta, \sigma)$ are suitable and that $\sigma$ is an uncountable supercompact cardinal.  
There exists a regular ultrafilter $\de$ over $\lambda$ such that for every model $M$ in a countable signature, 
$M^\lambda/\de$ is $\lambda^+$-saturated if $Th(M)$ is $(\lambda, \mu, \theta, \sigma)$-explicitly simple, and 
$M^\lambda/\de$ is not $\mu^{++}$-saturated if $Th(M)$ is not simple. 
\end{theorem}

\begin{proof} 
As $\sigma$ is uncountable and supercompact, Theorem \ref{t:optimal} and Conclusion \ref{c:uf-ba} 
prove existence of a regular ultrafilter $\de$ on $I$ such that: 
$\de$ is built from $(\de_0, \ba, \de_*)$ where $\de_0$ is regular and excellent on $\ba = \ba^1_{2^\lambda, \mu, \theta}$ 
and $\de_*$ is $(\lambda, \mu, \theta, \sigma)$-optimal, and moreover $\de$ is not $\mu^{++}$-good for any non-simple theory.  
By Theorem \ref{t:saturation}, 
any such $\de$ is good for any $(\lambda, \mu, \theta, \sigma)$-explicitly simple theory. 
\end{proof}

\begin{theorem}[The ultrapower characterization of simple theories] \label{main-theorem}
Assume $(\lambda, \mu, \theta, \sigma)$ are suitable, $\mu^+ = \lambda$  
and $\sigma$ is an uncountable supercompact cardinal.  
Then there is a regular ultrafilter $\de$ on $\lambda$ such that for any model $M$ in a countable signature,
$M^\lambda/\de$ is $\lambda^+$-saturated whenever $Th(M)$ is simple, and $M^\lambda/\de$ is not $\lambda^+$-saturated whenever $Th(M)$ is not simple. 
\end{theorem}

\begin{proof} 
Apply Theorem \ref{t:summary} assuming in addition that $\mu^+ = \lambda$. 
Then every countable simple theory is $(\lambda, \mu, \theta, \sigma)$-explicitly simple by Theorem \ref{c:s-es}. 
\end{proof}

\begin{rmk} In Theorem \ref{main-theorem}, countable may clearly be weakened to $|T| < \sigma$. 
\end{rmk}

Assuming existence of an uncountable supercompact cardinal, Theorem \ref{main-theorem} has the 
following immediate consequence for the structure of Keisler's order. Thanks to the referee for suggesting 
the formulation.  

\begin{concl} \label{consequences} 
Assume there exists an uncountable supercompact cardinal. If $T$, $T^\prime$ are countable theories, $T$ is simple, and $T^\prime \trianglelefteq T$, 
then $T^\prime$ is simple. 
\end{concl}


\br
\section{Perfect ultrafilters} \label{s:perfect}

In this section we shall give a natural set-theoretic condition on ultrafilters, called `perfect', 
which essentially says that they solve as many problems as possible modulo the cardinal constraints. 
We will use perfect ultrafilters in \cite{MiSh:1050} in the case $\sigma = \aleph_0$ in applying the analysis of this paper to theories 
with trivial forking. Recall the definition of support, \ref{d:support} above. 

\begin{defn}[Perfect ultrafilters for the case $\sigma = \theta = \aleph_0$] 
\label{d:perfect-0} 
Let $(\lambda, \mu, \aleph_0, \aleph_0)$ be suitable. 
We say that an ultrafilter $\de_*$ on $\ba = \ba^1_{2^\lambda, \mu}$ 
is \emph{$(\lambda, \mu)$-perfect} when $(A)$ implies $(B)$:
\begin{enumerate}
\item[(A)] $\langle \mb_u : u \in [\lambda]^{<\aleph_0} \rangle$ is a monotonic sequence of elements of $\de_*$ 
\\ and  
$\supp(\bar{\mb})$ is a support for $\bar{\mb}$ of cardinality $\leq \lambda$, see $\ref{d:support}$, such that \\ for every $\alpha < 2^\lambda$ with 
$\bigcup \{ \dom(f) : \mx_f \in \supp(\overline{\mb}) \} \subseteq \alpha$, 
\\ there exists a multiplicative sequence 
\[ \langle \mb^\prime_u : u \in [\lambda]^{<\aleph_0} \rangle \]
of elements of $\ba^+$ 
such that
\begin{itemize} 
\item[(a)] $\mb^\prime_{u} \leq \mb_{u}$ for all $u \in [\lambda]^{<\aleph_0}$,  
\item[(b)] for every $\mc \in \ba^+_{\alpha, \mu} \cap \de_*$, 
no intersection of finitely many members of  
$\{ \mb^\prime_{\{i\}} \cup (1-\mb_{\{i\}}) : i < \lambda \}$ 
is disjoint to $\mc$. 
\end{itemize}
\item[(B)] there is a multiplicative sequence $\bar{\mb}^\prime = \langle \mb^\prime_u : u \in [\lambda]^{<\aleph_0} \rangle$ 
of elements of $\de_*$ which refines $\bar{\mb}$. 
\end{enumerate}
\end{defn}

\begin{obs} \label{o:support}
Suppose $\alpha < 2^\lambda$ is fixed, $D_\alpha$ is an ultrafilter on $\ba^1_{\alpha, \mu} \subseteq \ba = \ba^1_{2^\lambda, \mu}$, and 
$\langle \mb_u : u \in [\lambda]^{<\aleph_0} \rangle$ is a sequence of elements of $D_\alpha$. Suppose 
that there exists a multiplicative sequence 
$\langle \mb^\prime_u : u \in [\lambda]^{<\aleph_0} \rangle$
of elements of $\ba^+$ 
such that
\begin{itemize} 
\item[(a)] $\mb^\prime_{u} \leq \mb_{u}$ for all $u \in [\lambda]^{<\aleph_0}$,  
\item[(b)] for every $\mc \in \ba^+_{\alpha, \mu} \cap \de_\alpha$, 
no intersection of finitely many members of  
$\{ \mb^\prime_{\{i\}} \cup (1-\mb_{\{i\}}) : i < \lambda \}$ 
is disjoint to $\mc$. 
\end{itemize}
Then there is a multiplicative sequence $\langle \mb^{\prime\prime}_u : u \in [\lambda]^{<\aleph_0} \rangle$ 
such that (a), (b) hold with $\mb^\prime_u$, $\mb^\prime_{\{i\}}$ 
replaced by $\mb^{\prime\prime}_u$, $\mb^{\prime\prime}_{\{i\}}$ respectively, and such that 
some support of $\bar{\mb^{\prime\prime}}$ is contained in $\ba^1_{\alpha + \lambda, \mu}$.
\end{obs}

\begin{rmk}
Note that in $\ref{o:support}(b)$, omitting ``$1-\mb_{\{i\}}$'' gives an equivalent condition. 
\end{rmk}

\begin{proof}[Proof of Observation \ref{o:support}.]
Without loss of generality there is $\mcv$ of cardinality $\lambda$ such that some support of $\bar{\mb}^\prime$ is contained in 
$\{ \mx_f : f \in \fin_\mu(\mcv) \}$. Let $\pi$ be a permutation of $2^\lambda$ which is the identity on $\alpha$ and takes $\uu$ into $\alpha + \lambda$.  
This induces an automorphism $\rho$ of $\ba$ which is the identity on $\ba^1_{\alpha, \mu}$, so in particular is the identity on 
$\de_\alpha$ and thus on $\bar{\mb}$. For each $u \in [\lambda]^{<\aleph_0}$, let $\mb^{\prime\prime}_u = \rho(\mb^\prime_u)$. 
Then clearly $\bar{\mb^{\prime\prime}}$ fits the bill. 
\end{proof}

\begin{theorem}[Existence of perfect ultrafilters] \label{t:perfect-exists}
Let $(\lambda, \mu, \aleph_0, \aleph_0)$ be suitable. 
Let $\ba = \ba^1_{2^\lambda, \mu}$. Then there exists a $(\lambda, \mu, \aleph_0, \aleph_0)$-perfect ultrafilter on $\ba$.
\end{theorem}

\begin{proof} 
Begin by letting $\langle \bar{\mb}_\delta = \langle \mb_{\delta, u} : u \in [\lambda]^{<\aleph_0} \rangle : \delta < 2^\lambda \rangle$ be an enumeration of the monotonic sequences of elements of $\ba^+$, each occurring cofinally often.  
Let $z: 2^\lambda \rightarrow 2^\lambda$ be an increasing continuous function which satisfies: $z(0) \geq 0$ and for all 
$\beta < 2^\lambda$,  $z(\beta) + \lambda = z(\beta + 1)$.
By induction on $\delta < 2^\lambda$ we will construct 
$\langle D_\delta: \delta < 2^\lambda \rangle$, an increasing continuous sequence of filters with 
each $D_\delta$ an ultrafilter on $\ba_{z(\delta), \mu}$, to satisfy:   

\begin{enumerate}
\item[(*)] if $\delta = \beta + 1$, if it is the case that 

\begin{quotation}
\noindent $\langle \mb_{\beta, u} : u \in [\lambda]^{<\aleph_0} \rangle$ is a monotonic sequence of elements of $\de_\beta$ and   
there exists a choice of $\supp(\bar{b})$ with 
$\bigcup \{ \dom(f) : \mx_f \in \supp(\overline{\mb}) \} \subseteq \beta$
and there exists a multiplicative sequence 
\[ \langle \mb^\prime_u : u \in [\lambda]^{<\aleph_0} \rangle \]
of elements of $\ba^+$ 
such that
\begin{itemize} 
\item[(a)] $\mb^\prime_{u} \leq \mb_{\beta, u}$ for all $u \in [\lambda]^{<\aleph_0}$,  
\item[(b)] for every $\mc \in \ba^+_{z(\beta), \mu} \cap \de_\beta$, 
no intersection of finitely many members of 
$\{ \mb^\prime_{\{i\}} \cup (1-\mb_{\beta, \{i\}}) : i < \lambda \}$ 
is disjoint to $\mc$. 
\end{itemize}
\end{quotation}

\emph{then} there is a sequence $\bar{\mb}^{\prime\prime} = \langle {\mb}^{\prime\prime}_u : u \in [\lambda]^{<\aleph_0} \rangle$ 
of elements of $\ba^+$ such that:
\begin{enumerate}
\item[(i)] $\mb^{\prime\prime}_{u} \leq \mb_{\beta, u}$ for all $u \in [\lambda]^{<\aleph_0}$,  
\item[(ii)] for every $\mc \in \ba^+_{z(\beta), \mu} \cap \de_\beta$, 
no intersection of finitely many members of 
$\{ \mb^{\prime\prime}_{\{i\}} \cup (1-\mb_{\beta, \{i\}}) : i < \lambda \}$ 
is disjoint to $\mc$. 
\item[(iii)] some support of $\bar{\mb}^{\prime\prime}$ is contained in $\ba_{z(\delta), \mu}$, and 
\item[(iiv)] $D_\delta$ is an ultrafilter on $\ba_{z(\delta), \mu}$ which extends $D_\beta \cup \{ \mb^\prime_{u} : u \in [\lambda]^{<\aleph_0} \}$. 
\end{enumerate} 
\end{enumerate}
The induction may be carried out at limit stages because all of the $D_\delta$ are ultrafilters. 
Suppose $\delta = \beta + 1$. If $\bar{\mb}$ satisfies the quoted condition, then let 
$\bar{\mb}^{\prime\prime}$ be given by Observation \ref{o:support}, using $z(\beta)$ here for $\alpha$ there. 
Then (i), (ii), (iii) are satisfied, so we need to prove that 
\[ D_\beta \cup \{ \mb^{\prime\prime}_{u} : u \in [\lambda]^{<\aleph_0} \} \] 
has the finite intersection property. As $D_\beta$ is an ultrafilter on $\ba_{z(\beta), \mu}$, 
and $\bar{\mb}^\prime$ is a multiplicative sequence, it suffices to prove that for any 
$\mc \in \de_\beta$ and any finite $u \subseteq \lambda$, 
\[ \mc \cap \bigcap \{ \mb^{\prime\prime}_{\{i\}} : i \in u \} > 0. \]
As $\mb_{\{i\}} \in \de_\beta$ for each $i \in u$, we may assume that $\mc \cap (1-\mb_{\{i\}}) = 0$ for each $i \in u$. 
Then we are finished by assumption (ii). 
This completes the induction. Let $\de_* = \bigcup_{\delta < 2^\lambda} D_\delta$. 

Let us check that $\de_*$ is indeed a perfect ultrafilter.  If $\bar{\mb}$ satisfies condition \ref{d:perfect-0}(A), let $\uu$ 
be as there, and let $\delta = \beta + 1$ be an ordinal $< 2^\lambda$ such that $\bar{\mb}_\beta = \bar{\mb}$ and 
$\uu \subseteq \ba_{\beta, \mu}$, which is possible as we listed each sequence cofinally often. Then 
since $D_\beta$ was an ultrafilter, $\de_* \rstr \ba_{\beta, \mu} = \de_\beta$ so at stage $\delta$ 
condition (*) of the inductive hypothesis will be activated and 
we will have ensured that $\bar{\mb}$ has a multiplicative refinement in $\de_*$.  
\end{proof}

\begin{cor}
Suppose $(\lambda, \mu, \aleph_0, \aleph_0)$ are suitable.  and let $\de_*$ be an ultrafilter on $\ba^1_{2^\lambda, \mu}$. 
If $\de_*$ is $(\lambda, \mu)$-perfect, then it is $(\lambda, \mu)$-optimal. 
\end{cor}

\begin{proof}
We need to show that any sequence with the so-called Key Property \ref{kp} has a multiplicative refinement. 
Suppose then that $\bar{\mb} = \langle \mb_u : u \in [\lambda]^{<\sigma} \rangle$ is a monotonic sequence of elements of $\de_*$ with the Key Property, 
and fix a support $\supp(\bar{\mb})$ as given by that property. 
Let $\alpha$ be an ordinal $< 2^\lambda$ such that $\{ \mb_u : u \in [\lambda]^{<\aleph_0} \} \subseteq D_\alpha := \de_* \rstr \ba^1_{\alpha, \mu}$ and 
$\bigcup \{ \dom(f) : \mx_f \in \supp(\overline{\mb}) \} \subseteq \alpha$. Write $\ba_\alpha$ for $\ba^1_{\alpha, \mu}$ for the remainder of this proof. 
The Key Property guarantees the existence of a cofinal $\Omega \subseteq [\lambda]^{<\aleph_0}$  
and a sequence $\bar{\mb}^\prime = \langle \mb^\prime_{\{i\}} : i < \lambda \rangle$ 
of elements of $\ba^+$ which generates a multiplicative refinement 
$\langle \mb^\prime_{u} : u \in [\lambda]^{<\sigma}\rangle$ of $\bar{\mb}$ 
such that for each $f \in \fin_{\mu,\theta}(\alpha)$, 
and each $u \in \Omega$, if $\mx_f \leq \mb_u$ then we may extend $f \subseteq f^\prime \in \fin_{\mu,\theta}(2^\lambda)$ so that 
$\mx_{f^\prime} \leq \mb^\prime_u$.  

In order to guarantee that our perfect ultrafilter will have given $\bar{\mb}$ a multiplicative refinement, it will suffice to show 
that for every $\mc \in \ba^+_{\alpha} \cap \de_\alpha$, 
no intersection of finitely many members of  
$\{ \mb^\prime_{\{i\}} \cup (1-\mb_{\{i\}}) : i < \lambda \}$ 
is disjoint to $\mc$. 
Let such an $\mc$ be given, let $v \in [\lambda]^{<\aleph_0}$ and choose any $u$ with $v \subseteq u \in \Omega$. 
As $\mc \in \de_\alpha$ and $\mb_u \in \de_\alpha$, without loss of generality $\mc \leq \mb_u$. As the generators are 
dense in the completion, we may choose $\mx_f$ with $0 < \mx_f \leq \mc$ and $f \in \fin_{\mu,\theta}(\alpha)$. 
Then $\mx_f \leq \mb_u$ so by the Key Property, we may extend $f \subseteq f^\prime$ so that $\mx_{f^\prime} \leq \mb^\prime_u$.  
This proves that  $\mx_{f} \cap \bigcap \{ \mb^\prime_{\{i\}} : i \in u \} > 0$, as desired. 

Then $\bar{\mb}$ contains a multiplicative refinement by the definition of `perfect,' which proves that the ultrafilter is indeed optimal. 
\end{proof}

\begin{concl} \label{c:uf-ba2}
Let $(\lambda, \mu, \aleph_0, \aleph_0)$ be suitable.
Let $\ba = \ba_{2^\lambda, \mu}$. 
Then there is an ultrafilter $\de_*$ on $\ba$ such that:
\begin{enumerate}
\item[(a)] $\de_*$ is $(\lambda, \mu)$-perfect, and indeed $(\lambda, \mu)$-optimal.  
\item[(b)] if $\de$ is any regular ultrafilter built from $(\de_0, \ba, \de_*)$ where $\de_0$ is a regular 
$\lambda^+$-excellent filter on $\lambda$, we have that $\de$ is not good for any non-low or non-simple theory. 
\end{enumerate}
\end{concl}

\begin{proof}
It remains to justify clause (b) by quoting known results. 
\cite{Sh:c} VI.3.23 p. 364 proves that any ultrafilter constructed by means of such an independent family of functions 
where $\mu < \lambda$ will not be $\mu^+$-good, however the proof shows more: that it will not be $\mu^+$-flexible. 
An alternate discussion is given in \cite{MiSh:999} Section 9. 

The fact that an ultrafilter which is not flexible is not good for any non-low or non-simple theory was proved by 
\cite{mm4} and \cite{MiSh:998}. More precisely, in \cite{mm4} Section 8 it was proved that any regular ultrafilter which is 
good for a theory which has $TP_2$ or is simple and non low, must be flexible. In \cite{MiSh:998} it was proved that any regular 
ultrafilter which is good for a theory with $SOP_2$ is good, therefore a fortiori flexible. Since any non-simple theory has either 
$TP_2$ or $SOP_2$, this completes the proof. 
\end{proof}

\br

\section{Some further questions} \label{s:questions}

The theorems in this paper suggest a broad classification program for simple theories according to their ``explicit simplicity''. 
We believe the most urgent questions have to do with determining the identitity of the equivalence classes of simple theories in Keisler's order. 
The following specific natural questions also arise.  Assume $(\lambda, \mu, \theta, \sigma)$ are suitable.  

For the first question, it is natural to consider (as is written) the case where $\sigma = \theta = \aleph_0$, so there is very little forking, 
however this does not appear essential. 

\begin{qst}  \label{q:1}
Is there a countable simple theory which is $(\lambda, \mu, \aleph_0, \aleph_0)$-explicitly simple when $\mu = \aleph_1$ but not when $\mu = \aleph_0$, regardless 
of the value of $\lambda$?  What about other uncountable values of $\mu$? 
\end{qst}

Recall that by our arguments above the random graph requires only one color, and it is not difficult to produce examples of essentially the same complexity 
requiring finitely or countably many colors. 
In Question \ref{q:1}, ``regardless of the value of $\lambda$'' means that we ask essentially for a simple theory requiring an uncountable (i.e. $>|T|$) but constant number of colors. 

\begin{qst}[A maximal simple theory] 
Suppose $\sigma$ is uncountable. 
Is there a countable simple theory which is $(\lambda, \mu, \theta, \sigma)$-explicitly simple if and only if $\mu^+ = \lambda$? 
\end{qst}

Recall that in his paper \cite{keisler} 
Keisler had developed the notion of a ``versatile'' formula to describe when theories $T$ were saturated precisely by good regular ultrafilters.  
The next question asks whether something analogous can be done inside simplicity. 

\begin{qst} \label{q:versatile}
For which, if any, values of $(\lambda, \mu, \theta, \sigma)$ does there exist a simple theory $T$ which is 
saturated by a regular ultrafilter $\de$ on $\lambda$ iff $\de$ is $(\lambda, \mu, \theta, \sigma)$-optimal, 
or iff $\de$ is $(\lambda, \mu)$-perfect in the sense of \ref{d:perfect-0} above? 
\end{qst}

The ultrapower characterization of stable theories from \cite{Sh:c} Chapter VI proceeded by proving that a model of a stable theory is $\lambda^+$-saturated if and only if 
it is $\kappa(T)$ saturated and every maximal indiscernible sequence has cardinality at least $\lambda^+$. It would be interesting to develop, perhaps from the arguments above, 
an analogous characterization of saturation in simple theories. 

\begin{qst} 
Give an analogous characterization of the saturated models of simple theories. 
\end{qst}

Finally, we record the fundamental question of the minimum unstable Keisler class. 
The regular ultrafilters which saturate this class are known; see, for example, \cite{MiSh:996} Section 4. 

\begin{qst}
Give an internal model-theoretic characterization of the equivalence class of the random graph in Keisler's order. 
\end{qst}

\br




\begin{thebibliography}{00}

\bibitem{ak1} J. Ax and S. Kochen. ``Diophantine problems over local fields. I'' American Journal of Mathematics 87 (1965), pp. 605-630
%
\bibitem{ak2} J. Ax and S. Kochen.   ``Diophantine problems over local fields. II'' American Journal of Mathematics 87 (1965), pp. 631-648
%
\bibitem{ak3} J. Ax and S. Kochen.  ``Diophantine problems over local fields. III''  Annals of Mathematics, Ser. 2 83 (1966), pp. 437-456






\bibitem{cartan-a} H. Cartan, ``Th\'eorie des filtres.'' C. R. Acad. Sci. Paris, 205:595--598 (1937).

\bibitem{cartan-b} H. Cartan, ``Filtres et ultrafiltres.'' C. R. Acad. Sci. Paris, 205:777-779 (1937).

\bibitem{ck73} C. C. Chang and H. J. Keisler, \emph{Model Theory}, North-Holland, 1973.

\bibitem{c-h} G. Cherlin and E. Hrushovski, {\it Finite structures with few types.} Annals of Mathematics Studies 152. Princeton University Press, Princeton, NJ (2003). vi+193 pp.



\bibitem{dow} A. Dow, ``Good and OK ultrafilters.'' Trans. Amer. Math. Soc. 290:1 (1985), pp. 145--160.

\bibitem{ek} R. Engelking and M. Kar\l owicz, ``Some theorems of set theory and their topological consequences.''
Fund. Math. 57 (1965) 275--285.

\bibitem{ersov} Yu. L. Ersov, ``On the elementary theory of maximal normed fields.'' Doklady Akademii Nauk
SSSR, vol. 165 (1965); English translation, Soviet Mathematics, vol. 165 (1965), pp. 1390--1393.
%
\bibitem{fms} T. Frayne, A. Morel and D. Scott, ``Reduced direct products,'' Fund. Math. 51 (1962) pp. 195--248. 
%

\bibitem{f-k} G. Fichtenholz and L. Kantorovitch, ``Sur les op\'erations lin\'eaires dans l'espace des
fonctions born\'ees.'' Studia Math. 5 (1934), 69—-98.

\bibitem{GIL} R. Grossberg, J. Iovino, O. Lessmann, ``A primer of simple theories.'' Arch. Math. Logic 41 (2002), no. 6, 541--580.

\bibitem{hausdorff} F. Hausdorff, ``\"Uber zwei S\"atze von G. Fichtenholz und L. Kantorovitch.'' Studia
Math. 6 (1936), 18--19.

\bibitem{hrushovski1} E. Hrushovski, ``Pseudofinite fields and related structures.''  
Model Theory and Applications, Quaderni di Matematica, vol. 11, Aracne, Rome, 2002, pp. 151--212.



\bibitem{kanamori} A. Kanamori. \emph{The higher infinite}. Second ed., Springer-Verlag, Berlin, 2009.

\bibitem{keisler-x} H.J. Keisler, ``Ultraproducts and elementary classes.'' Indagationes Mathematicae, 23 (1961) pp. 477--495

\bibitem{keisler-1} H. J. Keisler, ``Good ideals in fields of sets.'' Annals of Math. (2) 79 (1964), 338--359.

\bibitem{keisler-survey} H. J. Keisler, ``A survey of ultraproducts.'' \emph{Logic, Methodology and Philosophy of Science},
Y. Bar-Hillel (Ed.), Proc. of the 1964 Int. Congress in Jerusalem, North-Holland, Amsterdam (1965), pp. 112–-126.


\bibitem{keisler-2} H. J. Keisler, ``Ideals with prescribed degree of goodness.'' Annals of Mathematics, Second Series, Vol. 81, No. 1 (Jan., 1965), pp. 112-116. 

\bibitem{keisler} H. J. Keisler, ``Ultraproducts which are not saturated.''
J. Symbolic Logic 32 (1967) 23--46.
%
%


\bibitem{kochen} S. Kochen, ``Ultraproducts in the theory of models.'' Annals of Math. 
(2) 74, No. 2 (1961), pp. 221--261.

\bibitem{koppelberg} S. Koppelberg, ``Cardinalities of ultraproducts of finite sets.'' 
J. Symb Logic, Vol. 45, No. 3 (Sep., 1980), pp. 574--584. 

\bibitem{kunen} K. Kunen, ``Ultrafilters and independent sets.'' Trans. Amer. Math. Soc. 172 (1972), 299--306.
\bibitem{los} J. \los, ``Quelques Remarques, Th\'eor\`emes et Problemes sur les Classes Definissables d'Algebres,'' 
in Mathematical Interpretation of Formal Systems, Studies in Logic, (1955) 98-113.

\bibitem{mm-thesis} M. Malliaris, Ph. D. thesis, University of California, Berkeley (2009).
\bibitem{mm1} M. Malliaris, ``Realization of $\vp$-types and Keisler's order.'' Ann. Pure Appl. Logic 157 (2009), no. 2-3, 220--224. 
\bibitem{mm2} M. Malliaris, ``The characteristic sequence of a first-order formula.'' Journal of Symbolic Logic, 75, 4 (2010) 1415-1440.
%
\bibitem{mm4} M. Malliaris, ``Hypergraph sequences as a tool for saturation of ultrapowers.'' Journal of Symbolic Logic, 
77, 1 (2012) 195--223.
%
\bibitem{mm5} M. Malliaris, ``Independence, order and the interaction of ultrafilters and theories.''  Ann. Pure Appl. Logic.
163, 11 (2012) 1580-1595. 



\bibitem{MiSh:978} M. Malliaris and S. Shelah, ``Regularity lemmas for stable graphs.'' Trans. Amer. Math Soc, 366 (2014), 1551--1585.

\bibitem{MiSh:996} M. Malliaris and S. Shelah, ``Constructing regular ultrafilters from a model-theoretic point of view.'' 
Trans. Amer. Math. Soc., electronically published on February 18, 2015, DOI: http://dx.doi.org/10.1090/S0002-9947-2015-06303-X (to appear in print).

\bibitem{MiSh:997} M. Malliaris and S. Shelah, ``Model-theoretic properties of ultrafilters built by independent families of functions.'' 
JSL 79, 1 (2014) 103--134.

\bibitem{MiSh:998} M. Malliaris and S. Shelah, ``Cofinality spectrum problems in model theory, set theory and general topology.'' 
J. Amer. Math. Soc., electronically published on April 9, 2015, DOI: http://dx.doi.org/10.1090/jams830 (to appear in print).  

\bibitem{MiSh:E74} M. Malliaris and S. Shelah, ``General topology meets model theory, on $\mathfrak{p}$ and $\mathfrak{t}$.'' {Proc Natl Acad Sci USA} (2013) 110:33, 13300-13305. 

\bibitem{MiSh:999} M. Malliaris and S. Shelah, ``A dividing line within simple unstable theories.'' Advances in Math 249 (2013) 250--288. 

\bibitem{MiSh:1009} M. Malliaris and S. Shelah, ``Saturating the random graph with an independent family of small range.'' 
In \emph{Logic Without Borders}, in honor of Jouko V\"{a}an\"{a}nen. Asa Hirvonen, Juha Kontinen, Roman Kossak, Andres Villaveces, eds. 
DeGruyter, 2015. 

\bibitem{MiSh:1050} M. Malliaris and S. Shelah, ``Keisler's order has infinitely many classes.'' http://arxiv.org/abs/1503.08341

\bibitem{MiSh:1051} M. Malliaris and S. Shelah, ``Model-theoretic applications of cofinality spectrum problems.'' 
http://arxiv.org/abs/1503.08338

\bibitem{MiSh:F1245} M. Malliaris and S. Shelah, ``An axiomatic approach to cofinality spectrum problems,'' manuscript F1245, in preparation. 

\bibitem{MiSh:F1484} M. Malliaris and S. Shelah, ``Boolean algebras underlying simple theories,'' manuscript F1484, in preparation. 

\bibitem{MiSh:F1499} M. Malliaris and S. Shelah. ``Open problems on ultrafilters and some connections to the continuum,'' manuscript F1499, in preparation. 


%
%
%
%
%
%
%

\bibitem{Moore} J. Moore, ``Model theory and the cardinal numbers $\mathfrak{p}$ and $\mathfrak{t}$.'' 
Proc. Natl. Acad. Sci. USA 2013 110 (33) 13238--13239. 

\bibitem{Negrepontis} S. Negrepontis. ``The existence of certain uniform ultrafilters.'' Annals of Mathematics, Second Series, Vol. 90, No. 1 (Jul., 1969), pp. 23--32.




\bibitem{Sh:13} S. Shelah, ``Every two elementarily equivalent models have isomorphic ultrapowers.'' Israel J. Math, 10 (1971) pp. 224--233.

\bibitem{Sh:14} 
S. Shelah, ``Saturation of ultrapowers and Keisler's order.'' Annals Math Logic 4 (1972) 75--114.  

\bibitem{Sh:93} S. Shelah, ``Simple unstable theories.'' Ann. Math. Logic 19:177--203 (1980).
%
\bibitem{Sh:c} 
S. Shelah, \emph{Classification Theory and the number of non-isomorphic models}, North-Holland, rev. ed. 1990 (first edition 1978). 

%
\bibitem{Sh500} 
S. Shelah, ``Toward classifying unstable theories.'' Annals of Pure and
Applied Logic 80 (1996) 229--255. 
%

\bibitem{tarski-aleph} A. Tarski, ``Une contribution \`a la th\'eorie de la mesure.'' Fund. Math. 15 (1930), 42--50.


\end{thebibliography}
\end{document}